\documentclass[11pt]{article}
\usepackage{latexsym,amsfonts,amsthm,amsmath,amscd,amssymb}
\usepackage[dvips]{graphicx}

\newtheorem{lemma}{Lemma}[section]
\newtheorem{thm}[lemma]{Theorem}
\newtheorem{rem}[lemma]{Remark}
\newtheorem{prop}[lemma]{Proposition}
\newtheorem{cor}[lemma]{Corollary}

\newtheorem{defn}[lemma]{Definition}

\newcommand\matR{{\mathbb{R}}}

\newcommand\matZ{{\mathbb{Z}}}
\newcommand\matC{{\mathbb{C}}}

\newcommand\calT{{\mathcal{T}}}
\newcommand\calN{{\mathcal{N}}}
\newcommand\calNtil{\widetilde{{\mathcal{N}}}}
\newcommand\calA{{\mathcal{A}}}
\newcommand\calAtil{\widetilde{{\mathcal{A}}}}

\newcommand{\duerom}{I\!I}
\newcommand{\trerom}{\duerom\!I}

\newcommand{\zetadue}{\matZ/_{\!2\matZ}}
\newcommand{\zetatre}{\matZ/_{\!3\matZ}}

\renewcommand{\hbar}{{\overline{h}}}

\newfont{\Got}{eufm10 scaled 1200}
\newcommand{\permu}{{\hbox{\Got S}}}
\newcommand{\compo}{\,{\scriptstyle\circ}\,}

\newcommand{\mycap} [1] {\caption{\footnotesize{#1}}}

\begin{document}

\title{Spin structures on $3$-manifolds\\ via arbitrary triangulations}

\author{Riccardo \textsc{Benedetti}\and Carlo~\textsc{Petronio}}

\maketitle

\begin{abstract}
\noindent
Let $M$ be an oriented compact $3$-manifold and let $\calT$ be a (loose) triangulation of $M$,
with ideal vertices at the components of $\partial M$ and possibly internal vertices.
We show that any spin structure $s$ on $M$
can be encoded by extra combinatorial structures on $\calT$.
We then analyze how to change these extra structures on $\calT$, and $\calT$ itself, without
changing $s$, thereby getting a combinatorial realization, in the usual
``objects/moves'' sense, of the set of all pairs $(M,s)$. Our moves have a local nature,
except one, that has a global flavour but is explicitly described anyway.
We also provide an alternative approach where the global move is replaced by
simultaneous local ones.
\end{abstract}

\noindent
Combinatorial
presentations of $3$-dimensional topological categories, such
as the description of closed oriented $3$-manifolds via surgery on framed links in $S^3$,
and many more, are among the main themes of geometric
topology, and in particular have proved crucial for the theory of quantum
invariants, initiated in~\cite{RT} and~\cite{TV}.

A combinatorial presentation of the set of all pairs $(M,s)$, with $M$ a closed oriented
$3$-manifold and $s$ a spin structure on $M$, was already contained in~\cite{LNM}.
This presentation was realized by selecting the (loose) triangulations of $M$
having only one vertex and supporting a $\Delta$-{\it complex}
structure (see~\cite{HATCHER}), also called a {\it branching}.
The viewpoint adopted in~\cite{LNM} was actually that of special spines, equivalent to that of
triangulations via duality (see Matveev~\cite{Matveev:AAM} and below).
For the special spine dual to a triangulation, a branching is precisely a structure of
\emph{oriented branched surface} (see Williams~\cite{Williams}), and this structure was used in~\cite{LNM}
to define a trivialization of the tangent bundle of $M$ along the $1$-skeleton
of the spine, whence a spin
structure on $M$, using constructions already proposed by Ishii~\cite{Ishii} and Christy~\cite{Christy}.

The construction just described easily extends to pairs $(M,s)$ with $M$ a
compact oriented $3$-manifold with non-empty boundary and $s$ a spin structure on $M$, using
\emph{branchable} triangulations of $M$ with \emph{ideal} vertices at the components of $\partial M$, and possibly
internal internal vertices. This approach however suffers from the drawback that not all triangulations
of $M$ are branchable: for instance, the canonical triangulation by two regular hyperbolic ideal tetrahedra of the
hyperbolic one-cusped manifold called the ``figure-eight-knot-sister'' is not branchable.
On one hand, one easily sees that any triangulation of $M$ has branchable subdivisions
(\emph{e.g.}, take a regular subdivision and define a branching
by choosing a total ordering of the vertices). On the other hand, in many
circumstances one is interested in sticking to a given triangulation of $M$, or to
consider the class of all vertex-efficient triangulations of $M$
(namely, the purely ideal triangulations for non-empty $\partial M$,
and the $1$-vertex triangulations for closed $M$).

Recently, generalized versions of the notion of branching (see the definitions below),
with the nice property of existing on {\it every} triangulation, have
emerged as useful devices to deal with simplicial formulas defined over
triangulations equipped with solutions of Thurston's $\text{PSL}(2,\matC)$
consistency equations (or variations of them~\cite{Luo12,Luo13}).
For instance, motivated by his work in progress on the entropy of
solutions of the homogeneous $\text{PSL}(2,\matR)$ Thurston equations,
Luo introduced the notion of $\zetadue$-taut structure on a triangulation,
and it turns out that a certain notion of \emph{weak branching},
widely employed below together with the underlying notion of
\emph{pre-branching}, easily allows to show that every triangulation admits
$\zetadue$-taut structures (see Remark~\ref{taut:rem}).
As another example, the same notions of weak and pre-branching
were exploited in~\cite{BBnew} to extend the construction of quantum hyperbolic invariants~\cite{BB2,BB3}
to an arbitrary hyperbolic one-cusped  manifold, over a canonical Zariski-open set
of the geometric component of its character variety.

In several instances Luo~\cite{Luo:conv} suggested that a combinatorial encoding
of spin structures based on arbitrary triangulations might be of use for
the construction of spin-refined invariants obtained from simplicial
formulas as those mentioned in the previous paragraph.
In this note we provide such a presentation, using the notion of
weak branching already alluded to.

The results established in this paper provide an ``objects/moves'' combinatorial presentation of the
set of all pairs $(M,s)$, with $M$ a compact oriented $3$-manifold
and $s$ a spin structure on $M$, in the following sense:
\begin{itemize}
\item Given any (loose) triangulation $\calT$ of $M$,
with ideal vertices at the components of $\partial M$ and
possibly internal vertices, and any $s$,
we encode $s$ by decorating $\calT$ with certain extra combinatorial structures;
\item We exhibit combinatorial moves on decorated triangulations relating to each other any
two that encode the same $(M,s)$.
\end{itemize}
We note that all our moves are explicitly described, but one of them has an intrinsically global nature.
On the other hand, in the second part of the
paper we will show that this move can actually be replaced, in a suitable sense,
by a combination of local ones. This last result is subtle and technically quite demanding,
it is based on some non-trivial algebraic constructions, and it unveils
unexpected coherence properties of the graphic calculus we use to encode weakly branched triangulations.

A first application of the technology developed in the present note appears in~\cite{BBnew},
where our results are used to solve a sign
indeterminacy in the phase anomaly of the quantum hyperbolic invariants
(see Remark~\ref{sign}).
We also note that adapting the arguments of~\cite[Chapter 8]{LNM}, the
results of this article can be used to provide an effective construction of
the Roberts spin-refined Turaev-Viro invariants~\cite{Roberts},
and of the related Blanchet spin-refined
Reshetikhin-Turaev invariants~\cite{blanchet} of the double of a manifold.

\section{Statements for triangulations}\label{tria:sec}
In this section we state some results that
provide in terms of arbitrary \emph{triangulations} a combinatorial encoding of spin structures
on oriented $3$-manifolds. The geometric construction underlying this encoding actually
employs certain objects called \emph{special spines}, and will be fully described in
Sections~\ref{spines:sec} and~\ref{moves:sec}. As a matter of fact, triangulations and special spines are equivalent to
each other via duality, but perhaps the majority of topologists is more familiar with
the language of triangulations, which is why we are anticipating our statements in this section.

\subsection{Triangulations, pre-branchings and weak branchings}
In this note $M$ will always be a connected, compact and oriented $3$-manifold, with or without
boundary. We also assume that $\partial M$ has no $S^2$ component (otherwise we canonically cap it with $D^3$).
We begin with several definitions. A \emph{triangulation} of $M$ is the datum $\calT$ of
\begin{itemize}
\item a finite number of oriented abstract tetrahedra, and
\item an orientation-reversing simplicial pairing of the $2$-faces of these tetrahedra
\end{itemize}
such that the space obtained by first gluing the tetrahedra along the pairings and then
removing open stars of the vertices is orientation-preservingly
homeomorphic to $M$ with some punctures (open balls removed).
Any number of punctures, including zero, is allowed (but a closed $M$
must be punctured at leats once).

A \emph{branching} on
an abstract oriented tetrahedron $\Delta$ is an orientation of
its edges such that no face of $\Delta$ is a cycle. Equivalently,
one vertex of $\Delta$ should be a source and one should be a sink, as illustrated in Fig.~\ref{tetraindicesduals:fig}-left.
Note that the figure shows the only two possible branched tetrahedra up to oriented isomorphism.
They are characterized by an \emph{index} $\pm1$, to define which one denotes by $v_j$ the vertex
of $\Delta$ towards which $j$ edges of $\Delta$ point, and one checks whether the ordering
$(v_0,v_1,v_2,v_3)$ defines the orientation of $\Delta$ or not.
\begin{figure}
    \begin{center}
    \includegraphics[scale=.6]{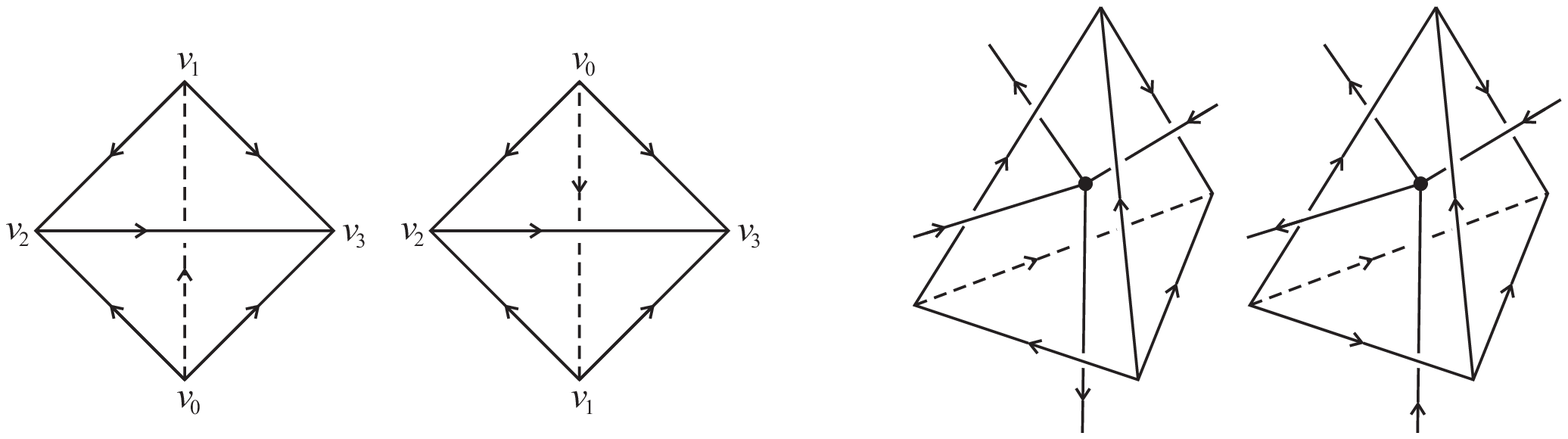}
    \end{center}
\vspace{-.5cm}
\mycap{Left: a branched tetrahedron of index $+1$ and one of index $-1$.
Right: a weak branching compatible with a pre-branching.
\label{tetraindicesduals:fig}}
\end{figure}
Each face of a branched abstract tetrahedron is endowed with the prevailing
orientation induced by its edges.

A \emph{pre-branching} on a triangulation $\calT$ is an orientation $\omega$ of
the edges of the gluing graph $\Gamma$ of $\calT$ (a $4$-valent graph) such at each vertex
two edges are incoming and two are outgoing. Given such an $\omega$, a \emph{weak branching}
$b$ compatible with $\omega$ is the choice of an abstract branching for each tetrahedron in
$\calT$ such that $\Gamma$ with its orientation $\omega$ is positively transversal
to each face of each tetrahedron in $\calT$, as in Fig.~\ref{tetraindicesduals:fig}-right.
Note that for such a $b$ when two faces are glued in $\calT$ either all three edge orientations are matched
or only one is, and in both the glued faces it is one of the prevailing two, as in Fig.~\ref{facegluing:fig}
\begin{figure}
    \begin{center}
    \includegraphics[scale=.67]{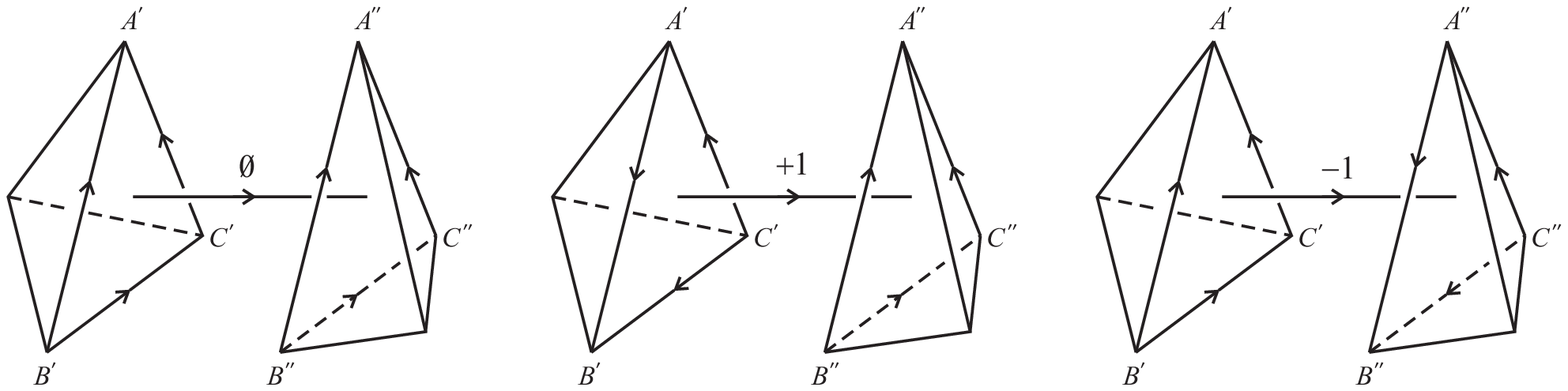}
    \end{center}
\vspace{-.5cm}\mycap{The three types of face-pairing in a weakly branched triangulation.\label{facegluing:fig}}
\end{figure}
(the labels $\emptyset,+1,-1$ are used below).

\subsection{Spin structure from a weak branching and a 1-chain}
All the constructions and results of the rest of this section will be explained and proved
in Sections~\ref{spines:sec} and~\ref{moves:sec} in the dual context of special spines.
Let a triangulation $\calT$ with pre-branching $\omega$ and compatible weak branching
$b$ be given. We will now define a chain $\overline{\alpha}(P,\omega,b)=\sum_e\alpha(e)\cdot e\in C_1\left(\calT;\zetadue\right)$, where
$e$ runs over all edges of $e$. The value of $\alpha(e)$ is
the sum of a fixed initial contribution $1$ plus
certain contributions of two different types; both contribution types are computed
in the group $G=\left(\frac12\cdot\matZ\right)/_{\!2\matZ}$, but for each of them the sum
is in $\zetadue$; here comes the description of the two types:
\begin{itemize}
\item Endow $e$ with an arbitrary orientation and in the abstract tetrahedra of $\calT$ consider the collection of
all the edges projecting to $e$ and of type $v_0v_2$ or $v_1v_3$; for each such abstract edge $\widetilde e$
take a contribution $+\frac12$ or $-\frac12$ depending on whether the projection from $\widetilde e$ to $e$
preserves or reverses the orientation;
\item Consider all the face-gluings as in Fig.~\ref{facegluing:fig} in which $e$ is involved
(with multiplicity) and take a contribution depending as follows
on the type $t$ of the gluing and on the position of $e$ within it:
\begin{itemize}
\item[$\triangleright$] $0$ if $t=\emptyset$, regardless of the position of $e$;
\item[$\triangleright$] $1$ if $t=\pm1$ and the orientation of $e$ is matched by the gluing;
\item[$\triangleright$] $\mp\frac12$ if $t=\pm1$ and the orientation of $e$ is not matched by the gluing.
\end{itemize}
\end{itemize}

\begin{prop}\label{alpha:tria:prop}
$\overline{\alpha}(P,\omega,b)$ is a coboundary, and to every $\overline{\beta}\in C_2\left(\calT;\zetadue\right)$
such that $\partial\overline{\beta}=\overline{\alpha}(P,\omega,b)$ there corresponds a spin
structure $s\left(\calT,\omega,b,\overline{\beta}\right)$ on $M$. Moreover
$s\left(\calT,\omega,b,\overline{\beta}_0\right)=s\left(\calT,\omega,b,\overline{\beta}_1\right)$ if and only if
$\overline{\beta}_0+\overline{\beta}_1$ is $0$ in $H_2\left(\calT;\zetadue\right)$.
\end{prop}

\begin{rem}\label{taut:rem}
\emph{Let $b$ be a weak branching compatible with a pre-branching $\omega$ on a triangulation $\calT$ of a manifold $M$.
If in each abstract tetrahedron of $\calT$ we choose the pair of opposite edges of types
$v_0v_2$ and $v_1v_3$ with respect to $b$, then the choice actually depends on $\omega$ only, not on $b$.
Moreover one sees that for all edges $e$ of $\calT$ in $M$ there is always an even number of abstract edges of types $v_0v_2$ or $v_1v_3$
projecting to $e$
(this corresponds to the fact that the contributions to $\overline{\alpha}(P,\omega,b)$ of the first type described
above are in $\zetadue$, and it is established in Proposition~\ref{real:obstruction:computation:prop} below).
It follows that, giving sign $-1$ to all the abstract edges $v_0v_2$ and $v_1v_3$, and sign $+1$ to the other edges,
we get a $\zetadue$-taut structure on $\calT$, as mentioned in the introduction.}
\end{rem}

\subsection{Triangulation moves preserving the spin structure}

The next results provide the combinatorial encoding of spin structures
announced in the title of the paper. From now on all chains $\overline{\beta}\in C_2\left(\calT;\zetadue\right)$
will be viewed up to 2-boundaries, without explicit mention.

\begin{prop}\label{vert:moves:tria:prop}
$s\left(\calT,\omega,b_0,\overline{\beta}_0\right)=s\left(\calT,\omega,b_1,\overline{\beta}_1\right)$ if and only if
$(b_0,\overline{\beta}_0)$ and $(b_1,\overline{\beta}_1)$ are related by the moves of
Fig.~\ref{tetravertmoves:fig}
\begin{figure}
    \begin{center}
    \includegraphics[scale=.6]{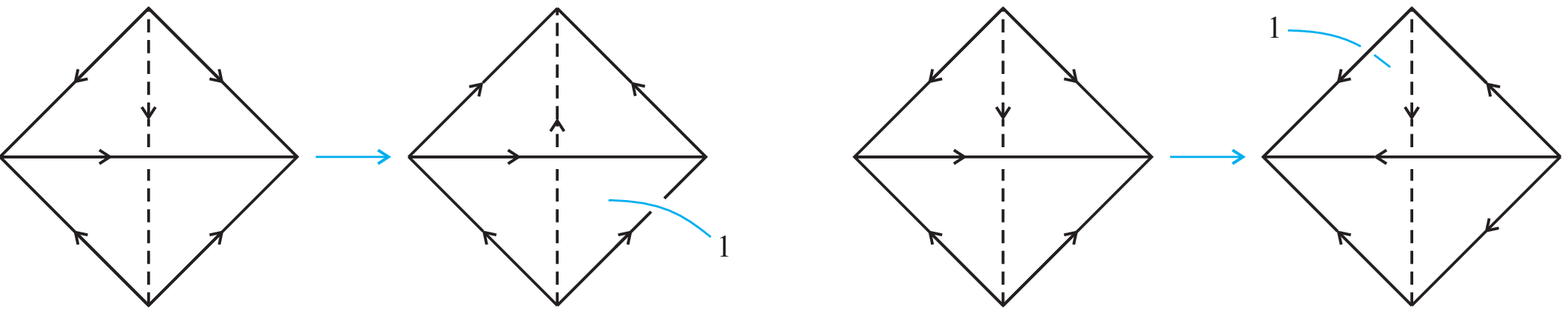}
    \end{center}
\vspace{-.5cm}\mycap{Moves preserving the pre-branching and the associated spin structure.
In both moves the ``$1$'' means that $1$ must be added to the coefficient in
$\overline{\beta}$ of the triangle to which ``$1$'' is attached; note that in both moves
it is the only one whose three edges all retain their orientation under the move.\label{tetravertmoves:fig}}
\end{figure}
(and their compositions and inverses).
\end{prop}

\begin{prop}\label{vert:circuit:prop}
$s\left(\calT,\omega_0,b_0,\overline{\beta}_0\right)=s\left(\calT,\omega_1,b_1,\overline{\beta}_1\right)$ if and only if
$(\omega_0,b_0,\overline{\beta}_0)$ and $(\omega_1,b_1,\overline{\beta}_1)$ are related by the moves of
Proposition~\ref{vert:moves:tria:prop} and additional moves
$\left(\calT,\omega,b,\overline{\beta}\right)\mapsto\left(\calT,\omega',b',\overline{\beta}'\right)$ described as follows:
\begin{itemize}
\item In the gluing graph of $\calT$ (which is oriented by $\omega$)
take an oriented simple circuit
$\gamma$ such that, for each tetrahedron it visits, the two faces it visits share the edge
$v_2v_3$ with respect to $b$, as in Fig.~\ref{tetracircuit:fig};
\begin{figure}
    \begin{center}
    \includegraphics[scale=.6]{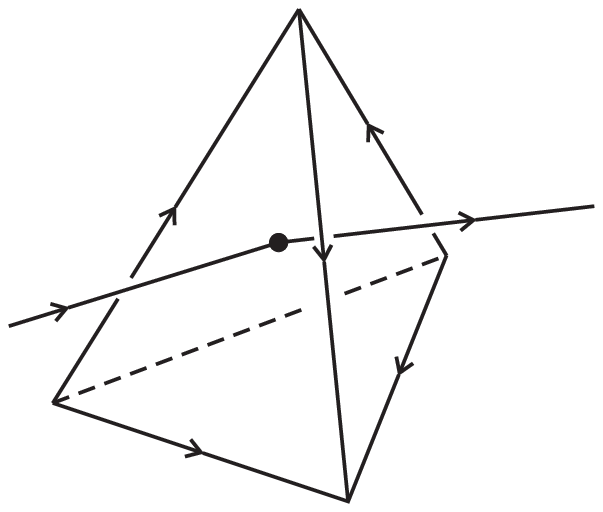}
    \end{center}
\vspace{-.5cm}\mycap{A circuit $\gamma$ in the gluing graph that in each tetrahedron visits faces sharing the edge $v_2v_3$.
The gluing encoded by an edge of $\gamma$ need not match edges of type $v_2v_3$ to each other.
\label{tetracircuit:fig}}
\end{figure}
\item Define $\omega'$ by reversing $\gamma$, define $b'$ by
reversing each edge $v_2v_3$ in each tetrahedron visited by $\gamma$,
and define $\overline{\beta}'$ by adding $1$ to the coefficient of each face of $\calT$
visited by $\gamma$ and incident to tetrahedra of distinct indices.
\end{itemize}
\end{prop}

\begin{prop}\label{tria:change:prop}
$s\left(\calT_0,\omega_0,b_0,\overline{\beta}_0\right)=s\left(\calT_1,\omega_1,b_1,\overline{\beta}_1\right)$ if and only if the quadruples
$\left(\calT_0,\omega_0,b_0,\overline{\beta}_0\right)$ and $\left(\calT_1,\omega_1,b_1,\overline{\beta}_1\right)$
are related by the moves of
Propositions~\ref{vert:moves:tria:prop} and~\ref{vert:circuit:prop} and those shown in Figg.~\ref{tetraMP:fig}
\begin{figure}
    \begin{center}
    \includegraphics[scale=.63]{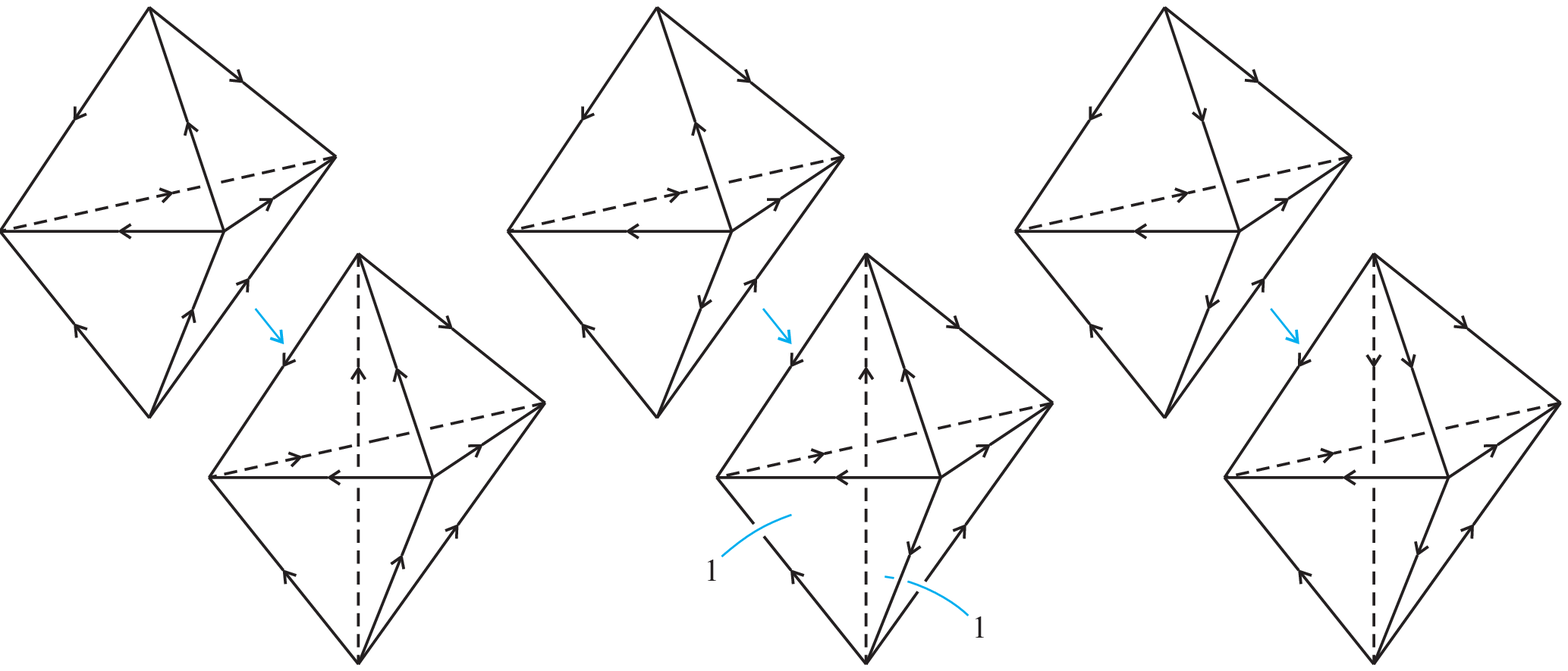}
    \end{center}
\vspace{-.5cm}\mycap{Moves preserving the spin structure. Note that in the central move the coefficients $1$ are
given to one internal and to one external face; coefficients $0$ are never shown.\label{tetraMP:fig}}
\end{figure}
and~\ref{tetrabubble:fig}.
\begin{figure}
    \begin{center}
    \includegraphics[scale=.65]{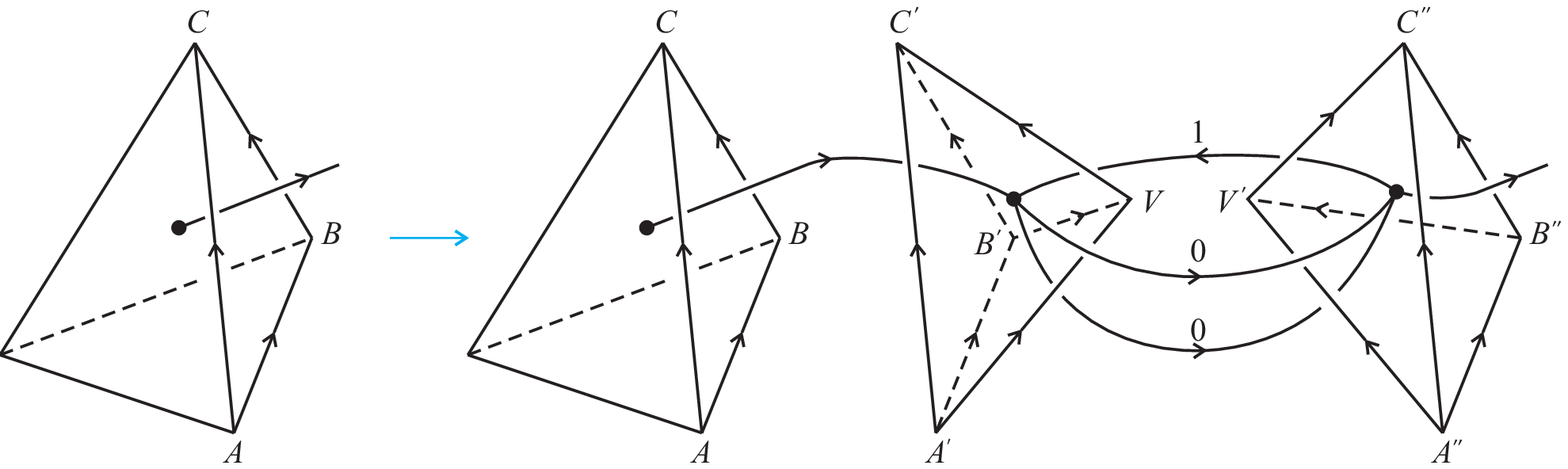}
    \end{center}
\vspace{-.5cm}\mycap{A move increasing by one the number of punctures and preserving the spin
structure. The coefficients of $ABV$, $ACV$ and $BCV$ in the 2-chain after the move are $0$, $0$ and $1$.\label{tetrabubble:fig}}
\end{figure}
\end{prop}

\begin{rem}
\emph{In this result one can avoid the move of Fig.~\ref{tetrabubble:fig}
if $\calT_0$ and $\calT_1$ have the same number of internal vertices and both consist
of at least two tetrahedra.}
\end{rem}

\section{Spin structures from weakly branched spines}\label{spines:sec}
We will now explain how the spin structure $s\left(\calT,\omega,b,\overline{\beta}\right)$
mentioned in the previous section is constructed. As announced, this employs the viewpoint of special spines,
which is dual to that of triangulations.

To a triangulation $\calT$ of $M$ we can associate the dual \emph{special spine} $P$
of $M$ minus some balls, as suggested in Fig.~\ref{duality:fig}.
The polyhedron $P$ is a compact $2$-dimensional one onto which
$M$ minus some balls collapses. Every point of $P$ has a neighbourhood
homeomorphic to the cone over a circle,
or over a circle with a diameter (in which case the point is said to belong to a \emph{singular edge}),
or over a circle with three radii (in which case the point is called a \emph{singular vertex}, and
the neighbourhood itself is called a \emph{butterfly}).
Moreover $P$ has vertices, its singular set $S(P)$ is a $4$-valent graph (actually, it is the gluing graph of $\calT$)
and the components of $P$ minus $S(P)$, that we call \emph{regions}, are homeomorphic to open discs.
Any such $P$ is called a \emph{special polyhedron}, and it is known that there can exist at most one \emph{thickening}
of $P$, namely a
punctured manifold $M$ collapsing onto $P$, in which case $P$ dually defines a triangulation of $M$.
Moreover one can add to $P$ an easy extra combinatorial structure, called
a \emph{screw-orientation} (see~\cite{BP:Manuscripta})
ensuring that $P$ is thickenable and that its thickening is oriented.
A screw-orientation for $P$ is an orientation of each edge $e$ of $P$
and a cyclic ordering of the three germs of regions incident to $e$, up to simultaneous reversal of both, with
obvious compatibility at vertices. All the special polyhedra we will consider will be embedded in an
oriented $3$-manifold or locally embedded in $3$-space, and we
stipulate from now on that the screw-orientation will always be the induced one, which allows us
to avoid discussing screw-orientation and orientation altogether.
\begin{figure}
    \begin{center}
    \includegraphics[scale=1]{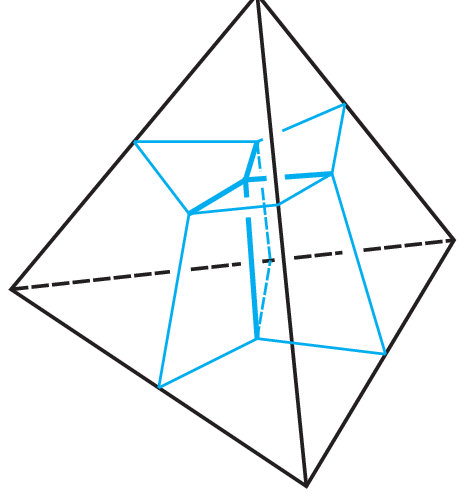}
    \end{center}
\vspace{-.5cm}\mycap{Duality between a tetrahedron and a \emph{butterfly} (the regular neighbourhood
of a vertex in a special spine).\label{duality:fig}}
\end{figure}

\subsection{Branched spines}
If an oriented tetrahedron $\Delta$ is branched, one can endow each wing of the dual butterfly $Y$ with the orientation
such that the edge of $\Delta$ dual to the wing is positively transversal to the wing. (Note that the ambient orientation is used here.)
One can moreover smoothen $Y$ along its singular set so that the positive transversal directions to the wings match,
as shown in Fig.~\ref{smoothvertices:fig},
where we show the butterflies dual to the branched tetrahedra of
Fig.~\ref{tetraindicesduals:fig}-left. We can further define along the singular set of $Y$
two vector fields $\nu$ (the positive transversal to the wings) and
$\mu_0$ (the descending vector field), and an orientation of the $4$ singular edges of the butterfly,
as shown in Fig.~\ref{numuonbutterfly:fig}.
Note that the orientation of an edge $e$ of a butterfly is always
given by the wedge of $\nu$ and $\mu_0$ along $e$, and it
is the prevailing orientation of the three induced by the wings incident to $e$.
\begin{figure}
    \begin{center}
    \includegraphics[scale=.6]{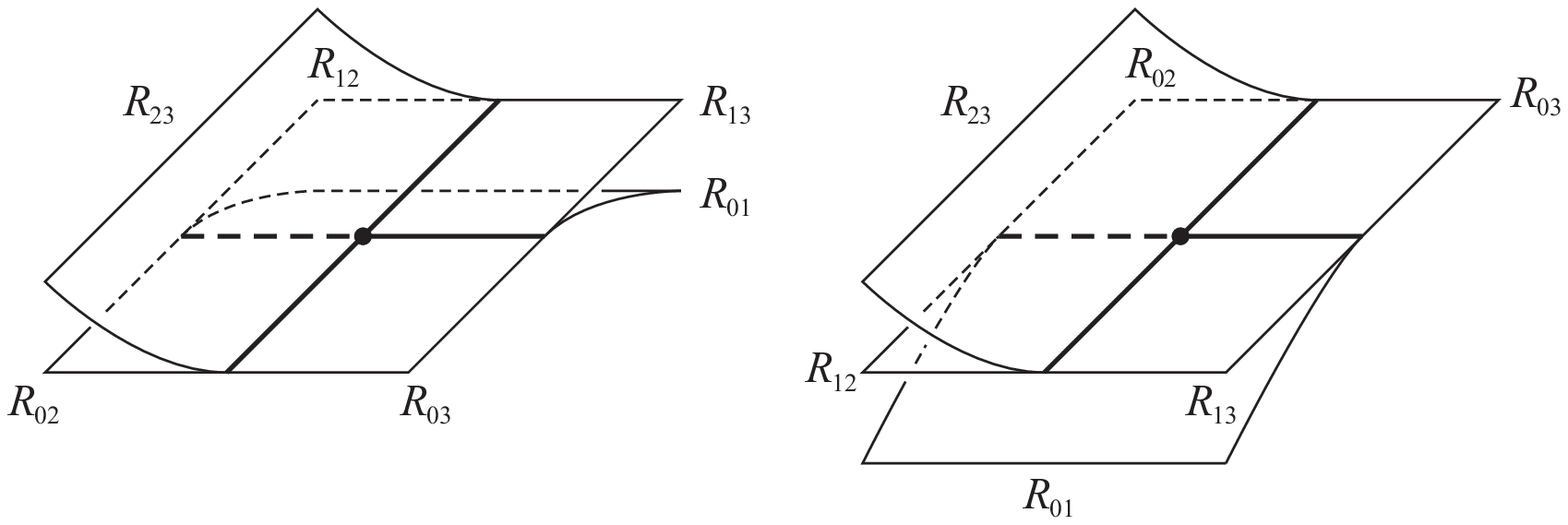}
    \end{center}
\vspace{-.5cm}\mycap{Smoothing of a butterfly carried by a branching of its dual tetrahedron $\Delta$.
Here $R_{ij}$ denotes the wing of the butterfly dual to the edge $v_iv_j$ of $\Delta$. \label{smoothvertices:fig}}
\end{figure}
\begin{figure}
    \begin{center}
    \includegraphics[scale=.6]{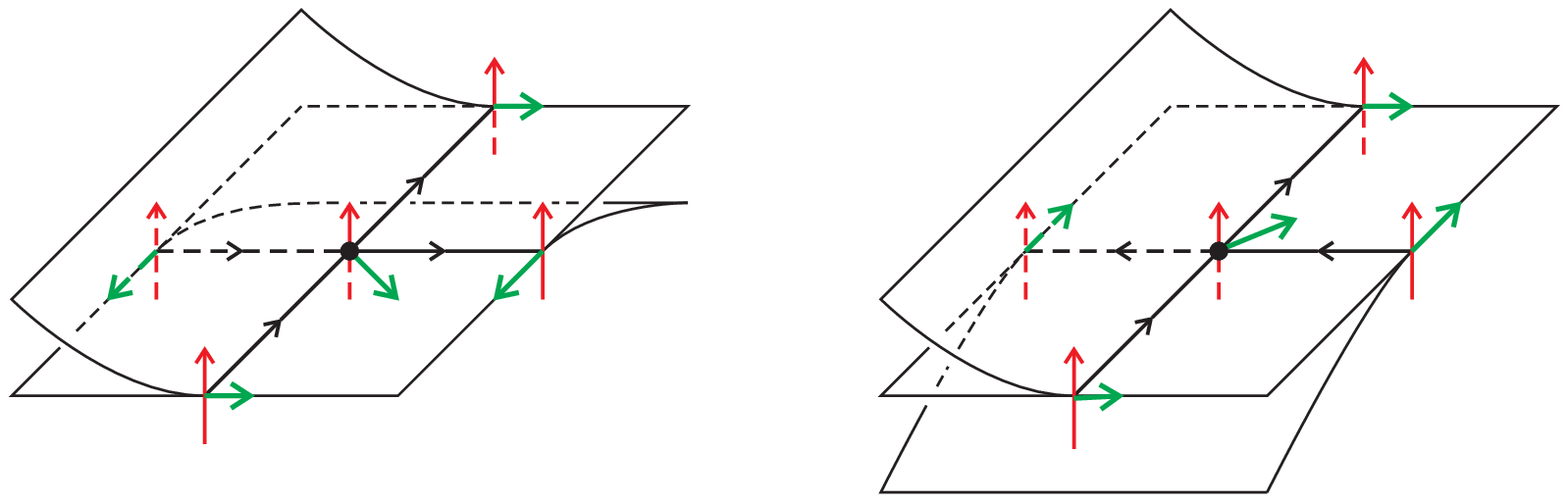}
    \end{center}
\vspace{-.5cm}\mycap{The fields $\nu$ (vertical) and $\mu_0$ (horizontal) along the singular set of a smooth butterfly, and
the orientation of its edges.\label{numuonbutterfly:fig}}
\end{figure}

\subsection{Weakly branched triangulations and the\\ induced frame along the dual $1$-skeleton}
Let us fix in this subsection a triangulation $\calT$ of an oriented manifold $M$ and
the special spine $P$ dual to $\calT$. If $\calT$ carries a \emph{global branching}, namely
if each tetrahedron in $\calT$ is endowed with a branching
so that all face-pairings match the edge orientations, then the frame $(\nu,\mu_0)$
extends to $S(P)$, as in Fig.~\ref{numubranchededge:fig} below.
However, a global branching does not always exist, and
we explain here how the structure of weak branching
still allows to globally define a frame along $S(P)$.

\begin{rem}
\emph{We will call \emph{frame} on a subset $X$ of $M$ a pair of linearly independent
sections defined on $X$ of the tangent bundle $TM$ of $M$; since $M$ is oriented, this
uniquely induces a trivialization of $TM$ on $X$.}
\end{rem}

Let us then take
a pre-branching $\omega$ of $\calT$, viewed as an orientation of $S(P)$
with two incoming and two outgoing edges at each vertex, and a weak branching
$b$ compatible with $\omega$. For an edge $e$ of $P$ the following
three possibilities (corresponding to those in Fig.~\ref{facegluing:fig}) occur:
\begin{itemize}
\item $e$ can be a \emph{branched edge} (type $\emptyset$), namely one along which the branchings defined at the ends are
compatible, as in Fig.~\ref{numubranchededge:fig};
the same figure shows how to (obviously) extend the frame $(\nu,\mu_0)$ along such an $e$;
\item If $e$ is not branched there is only one region $A$ incident to $e$ lying on the two-fold side
(namely, to the left of $e$) at both ends of $e$, and we say that:
\begin{itemize}
\item $e$ is a \emph{positive unbranched edge} (type $+1$) if $A$ is under at the beginning of $e$ and over at the end of $e$,
as in Fig.~\ref{numuextension:fig}-top/left;
\item $e$ is a \emph{negative unbranched edge} (type $-1$)  if $A$ is over at the beginning of $e$ and under at the end of $e$,
as in Fig.~\ref{numuextension:fig}-top/right.
\end{itemize}
In both cases we can again coherently define $\nu$ along $e$, by letting the transverse orientation of $A$
prevail on the other two, and accordingly define $\mu_0$, as illustrated in Fig.~\ref{numuextension:fig}-bottom.
\end{itemize}
\begin{figure}
    \begin{center}
    \includegraphics[scale=.6]{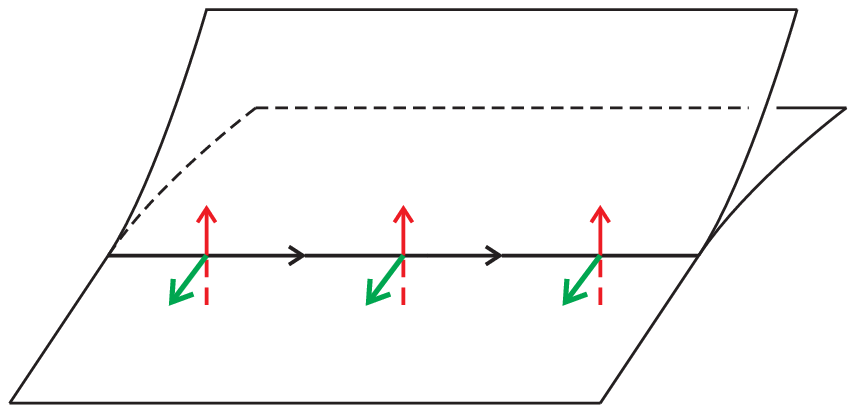}
    \end{center}
\vspace{-.5cm}\mycap{A branched edge and the extension of $(\nu,\mu_0)$ along it.\label{numubranchededge:fig}}
\end{figure}
\begin{figure}
    \begin{center}
    \includegraphics[scale=.6]{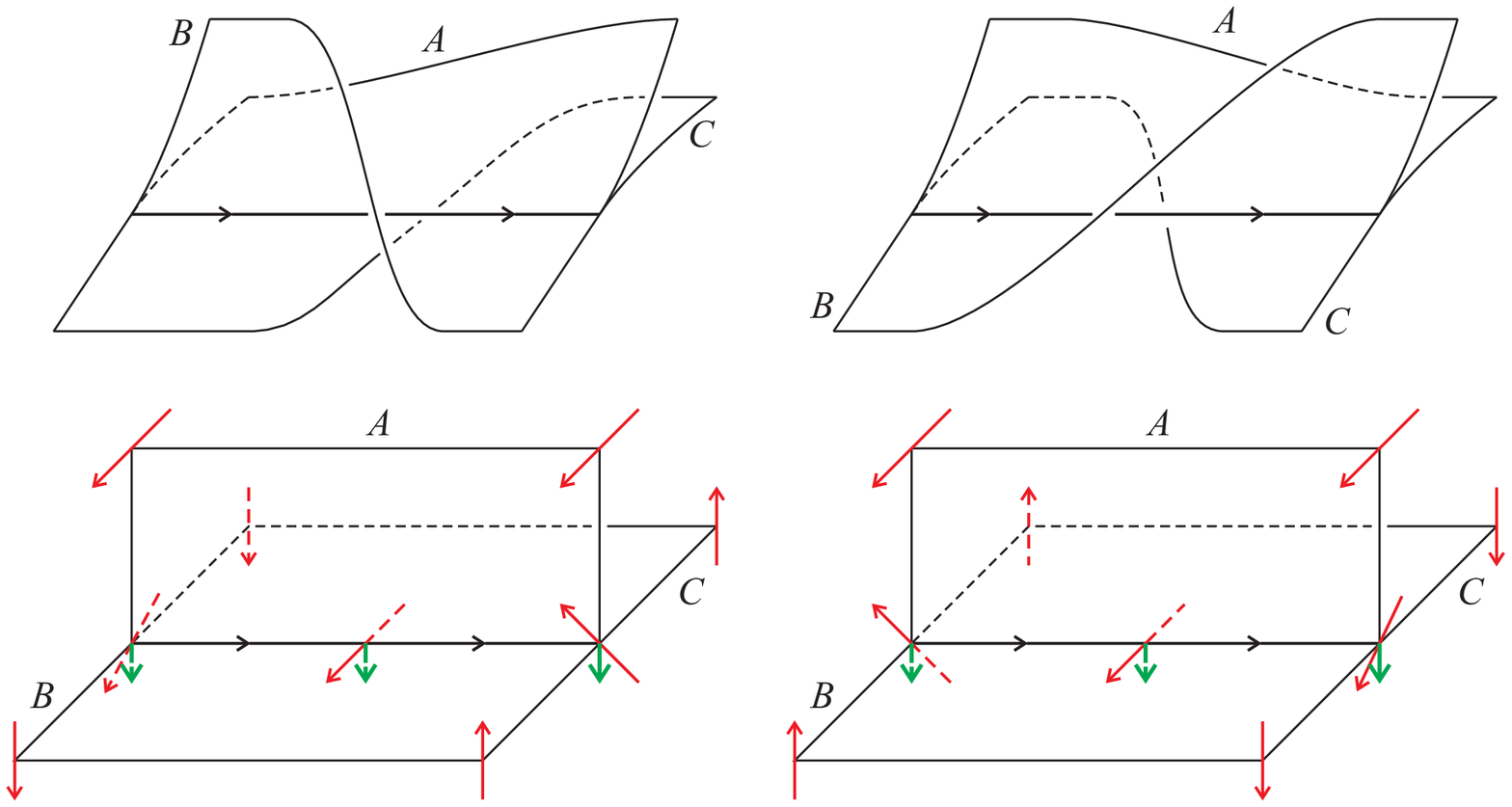}
    \end{center}
\vspace{-.5cm}\mycap{Top-left: a positive unbranched edge. Top-right: a negative one.
Bottom: the corresponding extensions of $(\nu,\mu_0)$. \label{numuextension:fig}}
\end{figure}

For a technical but important reason to a spine $P$ with pre-branching $\omega$ and compatible weak branching $b$
we actually associate a frame $\varphi=(\nu,\mu)$ that is obtained from the above-described $(\nu,\mu_0)$ by
adding to $\mu_0$ a full rotation around $\nu$ along each unbranched edge of $P$, as
shown in Fig.~\ref{correction:fig}.
\begin{figure}
    \begin{center}
    \includegraphics[scale=.6]{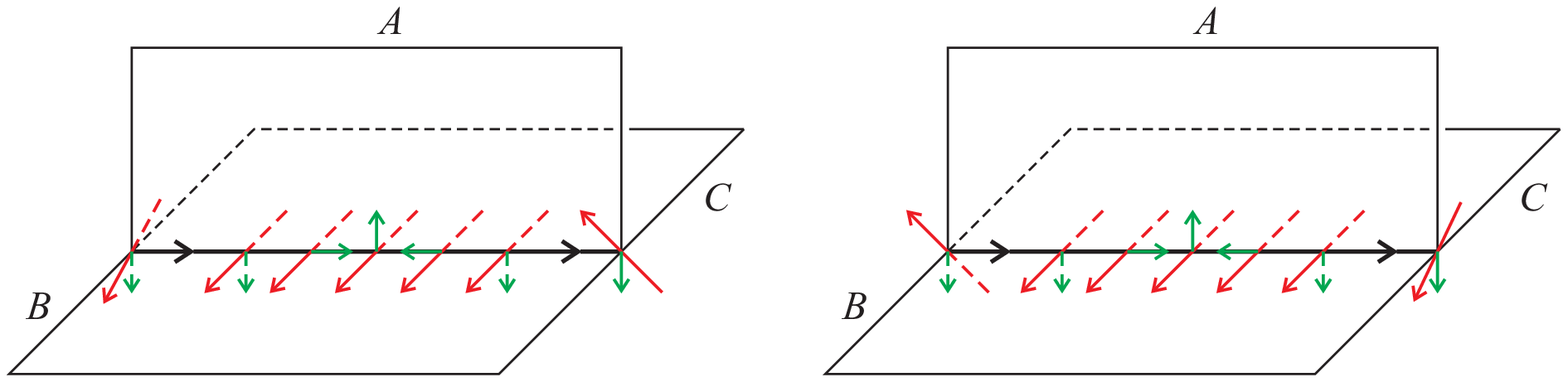}
    \end{center}
\vspace{-.5cm}\mycap{The field $\mu$ obtained by adding a full rotation to $\mu_0$ along each
(positive or negative) unbranched edge of $P$.\label{correction:fig}}
\end{figure}
We summarize the main points of our construction in the following:

\begin{defn}\label{main:defn}
\emph{Let $\calT$ be a triangulation of a compact oriented $3$-manifold $M$, and let $P$ be the dual spine of
 $M$ minus some balls .
A \emph{pre-branching} on $P$ is an orientation $\omega$ of its edges such that at each vertex two germs of edges are incoming
and two are outgoing. A \emph{weak branching} on $\calT$ compatible with $\omega$ is a choice $b$ of a branching for each
tetrahedron of $\calT$, such that $b$ induces $\omega$
at each vertex of $P$ according to Fig.~\ref{numuonbutterfly:fig}. The frame $\varphi(P,\omega,b)=(\nu,\mu)$ defined
along $S(P)$ is given by the pair $(\nu,\mu_0)$ at the vertices of $P$ as in
Fig.~\ref{numuonbutterfly:fig}, with extension $(\nu,\mu_0)$ along the edges of $P$ as in
Figg.~\ref{numubranchededge:fig} (branched edges) and~\ref{numuextension:fig}
(unbranched edges), and correction from $(\nu,\mu_0)$ to $(\nu,\mu)$
along the unbranched edges as in Fig.~\ref{correction:fig}.}
\end{defn}

\begin{rem}
\emph{For every triangulation $\calT$ the dual spine $P$ always admits some pre-branching $\omega$.
Given $\omega$, for a compatible weak branching $b$ there are $4$ independent choices at each tetrahedron of $\calT$.
The frame $\varphi(P,\omega,b)$ is well-defined up to homotopy on $S(P)$.}
\end{rem}

\subsection{Graphs representing weakly branched triangulations}
In this subsection we introduce a convenient graphic encoding for weakly branched triangulations
that we will later use to prove (the dual version of) Proposition~\ref{alpha:tria:prop}.
Let $\calN$ be the set of finite $4$-valent graphs $\Gamma$ with the following extra structures:
\begin{itemize}
\item Each edge of $\Gamma$ is oriented and bears a colour $\emptyset$, $+1$ or $-1$;
\item At each vertex of $\Gamma$ a planar structure as in Fig.~\ref{Wvertices:fig} left/right is given.
\begin{figure}
    \begin{center}
    \includegraphics[scale=.6]{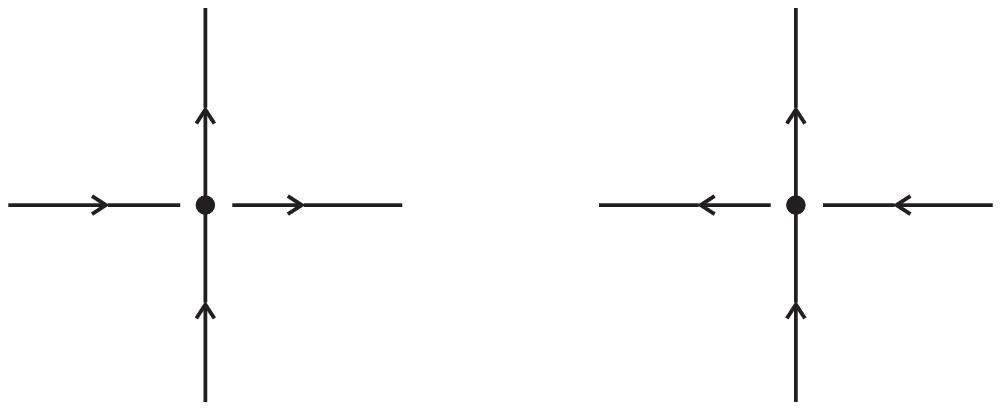}
    \end{center}
\vspace{-.5cm}\mycap{Planar structure of index $+1$ (left) or $-1$ (right) at a vertex of a graph in $\calN$.\label{Wvertices:fig}}
\end{figure}
\end{itemize}

Let $\calT$ be a weakly branched triangulation of an oriented $3$-manifold $M$, and let $P$
be the dual spine of  $M$ minus some balls. We can turn $S(P)$ into a graph $\Gamma(\calT)\in\calN$
by associating to a branched tetrahedron of $\calT$ as
Fig.~\ref{tetraindicesduals:fig}-left (or to a smoothed vertex of the dual spine $P$ as in
Fig.~\ref{numuonbutterfly:fig}) a vertex as in Fig.~\ref{Wvertices:fig},
and giving colour $\emptyset,+1,-1$ to each edge depending on its type.

The procedure just described can of course be reversed, namely to a graph $\Gamma\in\calN$ we can
associate a weakly branched triangulation $\calT(\Gamma)$ of an oriented manifold $M$.
Some examples of how to explicitly construct the spine $P$ dual to $\calT(\Gamma)$
along the edges of $\Gamma$ are illustrated in Fig.~\ref{Wreconstruct:fig}.
\begin{figure}
    \begin{center}
    \includegraphics[scale=.6]{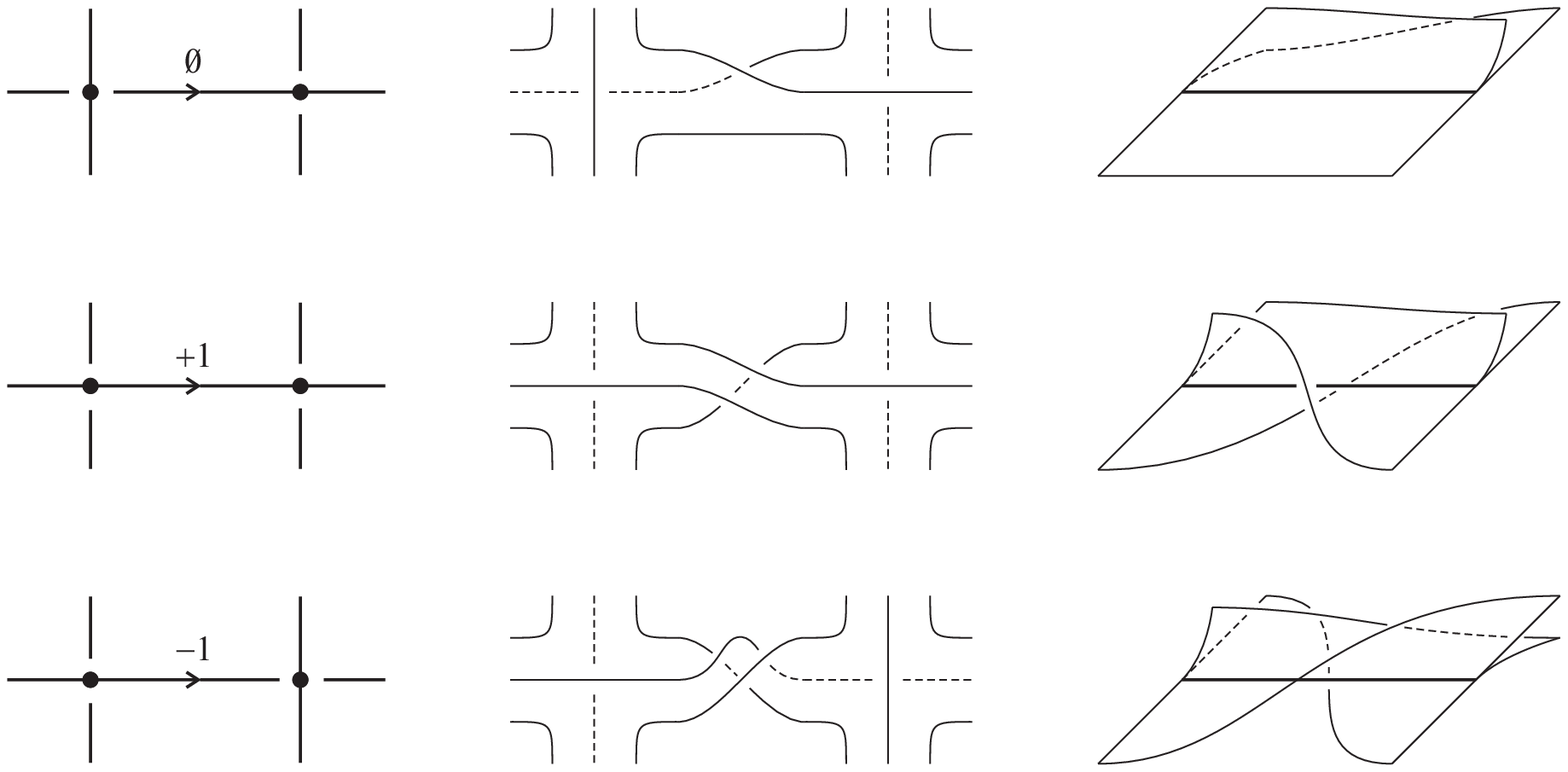}
    \end{center}
\vspace{-.5cm}\mycap{Reconstruction of a weakly branched spine from a graph in $\calN$.\label{Wreconstruct:fig}}
\end{figure}
(Recall that $P$ is determined by the attaching circles of its regions to $S(P)$,
which is what we show in Fig.~\ref{Wreconstruct:fig}-centre, and that the screw-orientation
of $P$ is induced by the local embedding in $3$-space, shown in Fig.~\ref{Wreconstruct:fig}-right.)
We summarize our construction as follows:

\begin{prop}\label{W:bijection:prop}
The set $\calN$ of decorated graphs corresponds bijectively to the set
of triples $(P,\omega,b)$ with $P$ an oriented special spine, $\omega$ a pre-branching on $P$
and $b$ a weak branching compatible with $\omega$.
\end{prop}

For later purpose we now need to extend the set of graphs $\calN$ to some $\calNtil$, by allowing $2$-valent
vertices besides the $4$-valent (decorated) ones, and insisting that the edge orientations should match
through the $2$-valent vertices. By interpreting each $2$-valent vertex as
    \includegraphics[scale=.25]{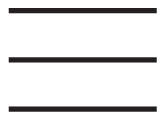}
or
    \includegraphics[scale=.25]{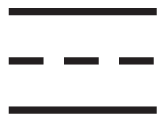}
we can then associate as above to each element $\widetilde{\Gamma}$ of $\calNtil$ a weakly branched special spine.
On the other hand we can define the \emph{fusion} of two edges separated by a valence-2 vertex
by interpreting the set of colours $\{\emptyset,+1,-1\}$ as $\zetatre$ and postulating that
colours sum up under fusion. Applying fusion as long as possible to
$\widetilde{\Gamma}\in\calNtil$ we then get some $\Gamma\in\calN$.
The following result can be easily verified ---see Fig.~\ref{Wassoc:fig}
\begin{figure}
    \begin{center}
    \includegraphics[scale=.6]{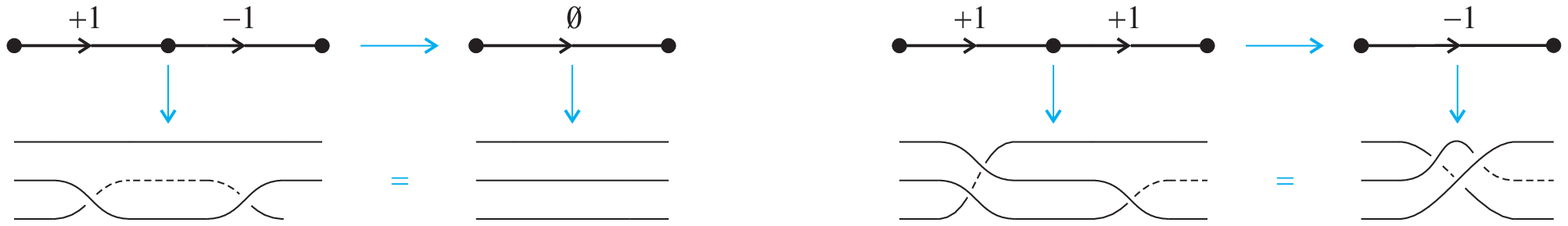}
    \end{center}
\vspace{-.5cm}\mycap{To each $\widetilde{\Gamma}\in\calNtil$ one can uniquely associate a weakly branched special spine,
also given by the graph $\Gamma\in\calN$ obtained from $\widetilde{\Gamma}$ by fusing the edges
through valence-$2$ vertices.\label{Wassoc:fig}}
\end{figure}

\begin{prop}\label{W:assoc:prop}
The weakly branched special spine associated to $\widetilde{\Gamma}\in\calNtil$ is independent of
the interpretation of the $2$-valent vertices, and it coincides with the spine corresponding
to the graph $\Gamma\in\calN$ obtained from $\widetilde{\Gamma}$ by edge-fusion.
\end{prop}

We conclude this subsection by explaining why
have defined $\varphi(P,\omega,b)=(\nu,\mu)$ not simply
as $(\nu,\mu_0)$, but adding instead a full twist to $\mu_0$ along unbranched edges:

\begin{prop}\label{numuadditivity:prop}
Take $\widetilde{\Gamma}\in\calNtil$ and let $\Gamma\in\calN$ be
obtained from $\widetilde{\Gamma}$ by fusing edges through valence-$2$ vertices.
Then the frames $\left(\widetilde{\nu},\widetilde{\mu}\right)$
and $(\nu,\mu)$ carried by $\widetilde\Gamma$ and by $\Gamma$ are homotopic to each other.
\end{prop}

\begin{proof}
We have to show that when we
fuse two coloured edges into one the frame $(\nu,\mu)$ defined by the fusion
is homotopic to the concatenation of the frames defined by the two edges.
Recall that the colour of the combination is the sum of the colours, and note that the conclusion
is obvious when one of the edge colours is $\emptyset$.
When the two edge colours are opposite to each other one can examine Fig.~\ref{numuextension:fig}
and see that the concatenation of the two frames $(\nu,\mu_0)$ is homotopic to a constant frame;
at the level of $(\nu,\mu)$ we would have to add two full twists to $\mu_0$, which amounts
to nothing, and the conclusion follows. We are left to deal with the sum of two edges with identical colour.
We deal with the case $+1+1$, since $-1-1$ is similar. The frames $(\nu,\mu_0)$ corresponding to
$+1+1$ and to $-1$ are shown in Fig.~\ref{numuadditivity:fig},
and recognized to differ by a full twist. When passing to $(\nu,\mu)$ we have to add two full twists
to $\mu$ (that is, nothing) in the $+1+1$ configuration, and one full twist in the $-1$ configuration,
thus getting homotopic frames.
\end{proof}
\begin{figure}
    \begin{center}
    \includegraphics[scale=.6]{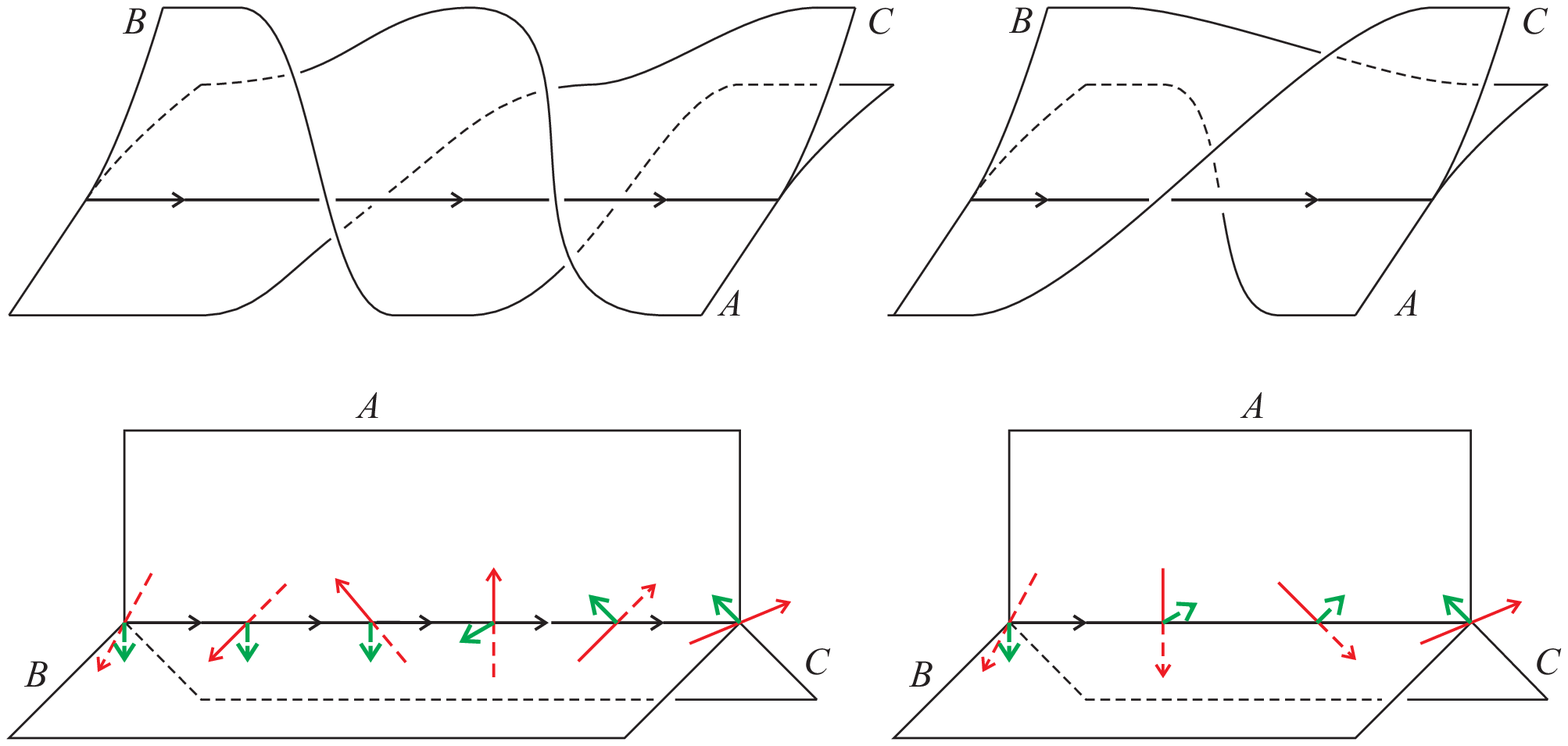}
    \end{center}
\vspace{-.5cm}\mycap{The frames $(\nu,\mu_0)$ corresponding to $+1+1$ and to $-1$.\label{numuadditivity:fig}}
\end{figure}

\subsection{Obstruction computation}
Let us now denote by $\alpha(P,\omega,b)\in C^2\left(P;\zetadue\right)$ the obstruction to extending
$\varphi(P,\omega,b)$ to a frame defined on $P$.
To define $\alpha(P,\omega,b)$, note that $TM$ can always be trivialized as $\textrm{GL}^{+}(3;\matR)\times R$
on each open region $R$ of $P$, and $\alpha(P,\omega,b)(R)$ is the element of $\pi_1(\textrm{GL}^{+}(3;\matR))=\zetadue$
represented by the restriction of $\varphi(P,\omega,b)$ to
(a loop parallel to) $\partial R$.
The next result shows that the chain $\overline{\alpha}(P,\omega,b)=\sum_e\alpha(e)\cdot e\in C_1\left(\calT;\zetadue\right)$
introduced in Section~\ref{tria:sec} is dual to $\alpha(P,\omega,b)$, namely that $\alpha(e)=\alpha(P,\omega,b)\left(R\right)$
if $R$ is the region of $P$ dual to an edge $e$ of $\calT$.

\begin{prop}\label{real:obstruction:computation:prop}
Given $\Gamma\in\calN$ decorate the attaching circles of the regions of the special spine $P$ defined by $\Gamma$ as follows:
\begin{itemize}
\item At each vertex of $\Gamma$ put arrows as in Fig.~\ref{Rmethod:fig}-top;
\begin{figure}
    \begin{center}
    \includegraphics[scale=.6]{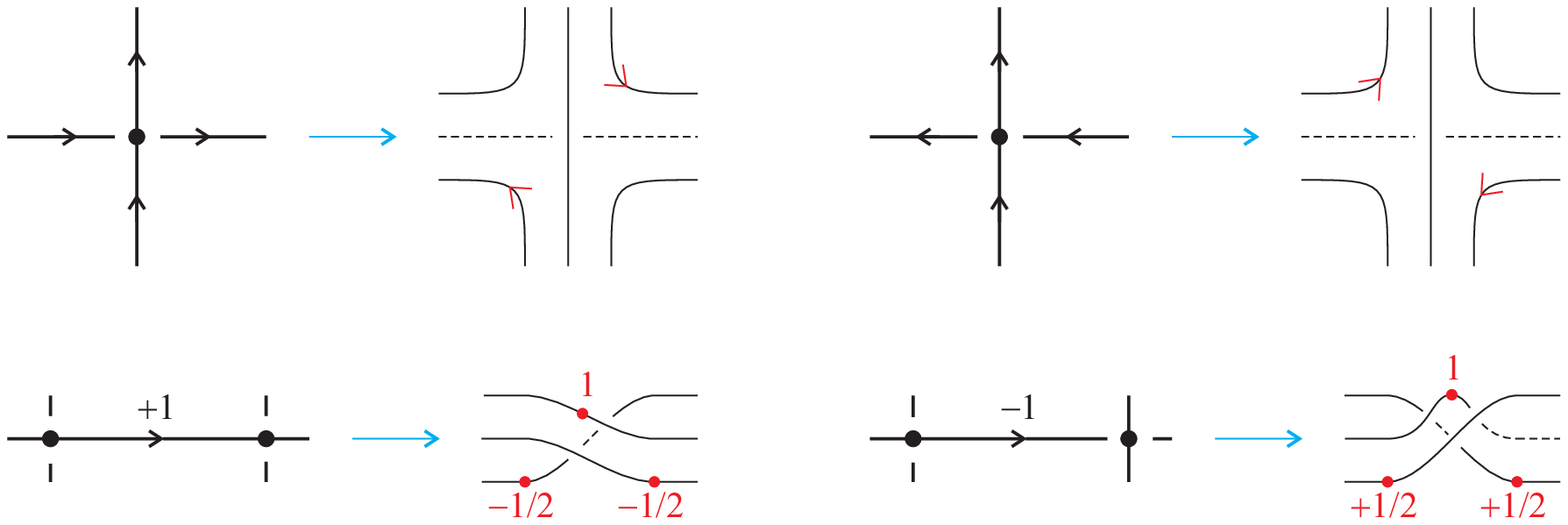}
    \end{center}
\vspace{-.5cm}\mycap{Decoration of the attaching circles of the regions near vertices and edges.\label{Rmethod:fig}}
\end{figure}
\item At each edge $e$ of $\Gamma$, if the edge colour is $\emptyset$, put nothing, while if
the edge colour is $\pm1$ put a weight $1$ on the region that lies to the left of $e$ at both ends of $e$,
and $\mp{\textstyle{\frac12}}$ on the two other regions (see two examples in Fig.~\ref{Rmethod:fig}-bottom).
\end{itemize}
Then $\alpha(P,\omega,b)(R)\in\zetadue$ is computed
as $1$ plus the sum of the numerical contributions along $\partial R$ plus
the sum of contributions from arrows, turned numerical as follows:
choose for $\partial R$ an arbitrary orientation and give each arrow value
$+{\textstyle{\frac12}}$ or $-{\textstyle{\frac12}}$ depending on whether it agrees or not with the orientation.
Moreover, both the sum of the numerical contributions and that of the contributions from arrows turned numerical
belong to $\zetadue$.
\end{prop}

\begin{proof}
Recall first that $\varphi(P,\omega,b)=(\nu,\mu)$ is obtained from $(\nu,\mu_0)$ by adding a full
twist to $\mu_0$ along the edges of $P$ having colour $\pm 1$. It is then sufficient to show
that the obstruction $\alpha_0$ to extending $(\nu,\mu_0)$ is computed by decorating
the regions of $P$ as in Fig.~\ref{Rmethod:fig}-top near the vertices and as in
Fig.~\ref{reducedRmethodedges:fig} near the edges.
\begin{figure}
    \begin{center}
    \includegraphics[scale=.6]{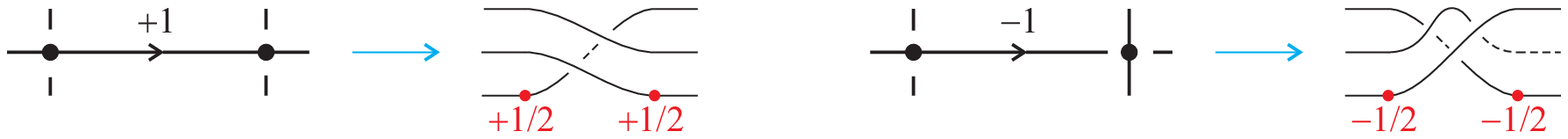}
    \end{center}
\vspace{-.5cm}\mycap{Reduced decoration near edges, used to compute $\alpha_0$.\label{reducedRmethodedges:fig}}
\end{figure}

Let us now pick a region $R$, give it some orientation, and compute $\alpha_0(R)$. Thanks to the orientation of $R$ and
of the ambient manifold $M$, for a vector at some point of $\partial R$ the positions shown in Fig.~\ref{positions:fig}
are well-defined.
\begin{figure}
    \begin{center}
    \includegraphics[scale=.6]{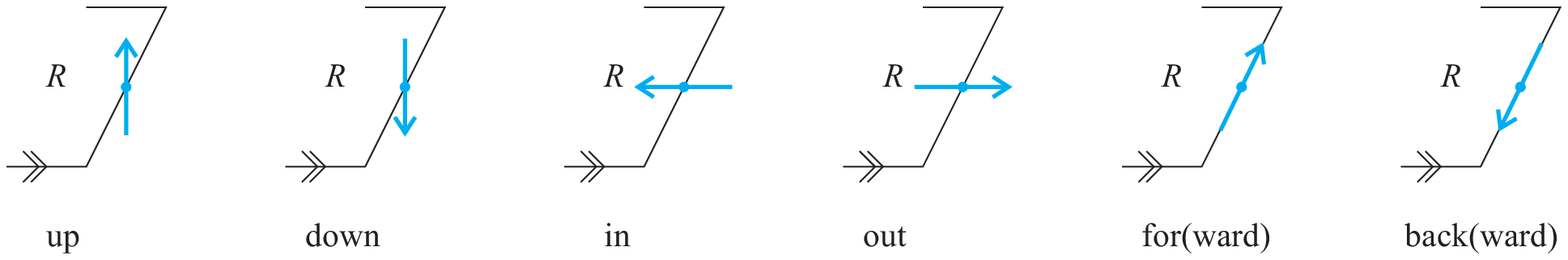}
    \end{center}
\vspace{-.5cm}\mycap{Positions of a vector on the boundary of an oriented region.\label{positions:fig}}
\end{figure}
We now analyze how the positions of $\nu$ and $\mu_0$ change as $\partial R$ travels near a vertex or edge of $P$.

From Fig.~\ref{numuonbutterfly:fig} one sees that $(\nu,\mu_0)$ does not change at a vertex except if
$\partial R$ is in one of the two positions indicated by arrows in Fig.~\ref{Rmethod:fig}-top
(the sink and the source quadrants of the vertex); for these positions,
we have $4$ different possibilities, two as follows
\begin{center}
\begin{tabular}{c|c||c|c}
&Position of $R$&&Position of $R$\\ \hline
$V_1$&Sink quadrant of vertex&$V_2$&Source quadrant of the vertex
\end{tabular}
\end{center}
with  $\partial R$ oriented as the arrow in Fig.~\ref{Rmethod:fig}-top, and two more
$\overline{V}_1$ and $\overline{V}_2$ with opposite orientation of $\partial R$;
the corresponding changes of $\nu$ and $\mu_0$ are
\begin{center}
\begin{tabular}{c|c|c||c|c|c}
&$\Delta\nu$&$\Delta\mu_0$&&$\Delta\nu$&$\Delta\mu_0$\\ \hline\hline
$V_1$&\small{up$\to$up$\to$up}&\small{out$\to$back$\to$in}&
$\overline{V}_1$&\small{down$\to$down$\to$down}&\small{in$\to$for$\to$out}\\ \hline
$V_2$&\small{up$\to$up$\to$up}&\small{in$\to$for$\to$out}&
$\overline{V}_2$&\small{down$\to$down$\to$down}&\small{out$\to$back$\to$in}
\end{tabular}
\end{center}
and this description applies whatever the index of the vertex.

Turning to $\Delta(\nu,\mu_0)$ along an edge $e$,
of course nothing happens if $e$ is branched or $e$ is unbranched but $R$ is in
position $A$ in Fig.~\ref{numuextension:fig};
otherwise we have $8$ possibilities, $4$ with $\partial R$ concordant with $e$ and $R$
in the following position
\begin{center}
\begin{tabular}{c|c||c|c}
&Position of $R$&&Position of $R$\\ \hline\hline
$E_1$&$B$ in Fig.~\ref{numuextension:fig}-left&$E_2$&$C$ in Fig.~\ref{numuextension:fig}-left\\ \hline
$E_3$&$B$ in Fig.~\ref{numuextension:fig}-right&$E_4$&$C$ in Fig.~\ref{numuextension:fig}-right
\end{tabular}
\end{center}
and $4$ more $\overline{E}_j$ with $\partial R$ discordant with $e$;
the corresponding $\Delta(\nu,\mu_0)$ is
\begin{center}
\begin{tabular}{c|c|c||c|c|c}
&$\Delta\nu$&$\Delta\mu_0$&&$\Delta\nu$&$\Delta\mu_0$   \\ \hline\hline
\!$E_1$\!&\small{up$\to$in$\to$down}&\small{out$\to$up$\to$in}&\!$\overline{E}_1$\!&\small{up$\to$in$\to$down}&\small{in$\to$down$\to$out}  \\ \hline
\!$E_2$\!&\small{down$\to$out$\to$up}&\small{in$\to$down$\to$out}&\!$\overline{E}_2$\!&\small{down$\to$out$\to$up}&\small{out$\to$up$\to$in}  \\ \hline
\!$E_3$\!&\small{down$\to$in$\to$up}&\small{in$\to$up$\to$out}&\!$\overline{E}_3$\!&\small{down$\to$in$\to$up}&\small{out$\to$down$\to$in}  \\ \hline
\!$E_4$\!&\small{up$\to$out$\to$down}&\small{out$\to$down$\to$in}&\!$\overline{E}_4$\!&\small{up$\to$out$\to$down}&\small{in$\to$up$\to$out.}
\end{tabular}
\end{center}
The value of $\alpha_0(R)$ will be given in $\pi_1(\textrm{GL}^{+}(3;\matR))=\zetadue=\{0,1\}$ by $1$ plus
some contribution of each configuration $V_i,\overline{V}_i,E_j,\overline{E}_j$, but:
\begin{itemize}
\item The $V_i,\overline{V}_i,E_j,\overline{E}_j$ cannot appear in arbitrary order: only some concatenations are possible;
\item The individual $V_i,\overline{V}_i,E_j,\overline{E}_j$ do not make sense in $\pi_1(\textrm{GL}^{+}(3;\matR))$
but some of their concatenations do, when $(\nu,\mu_0)$ is the same at the two ends of the configuration.
\end{itemize}
The idea of the proof is then to assign to each $V_i,\overline{V}_i,E_j,\overline{E}_j$ a value $\pm{\textstyle{\frac12}}$ so
that, whatever concatenation is possible and makes sense in $\pi_1(\textrm{GL}^{+}(3;\matR))$, its
geometrically correct value in $\pi_1(\textrm{GL}^{+}(3;\matR))$ is the sum of the values of the
$V_i,\overline{V}_i,E_j,\overline{E}_j$ appearing in it.
Turning to the details, the possible concatenations are
\begin{equation}\label{possible:conc:eq}
\begin{aligned}
\left\{V_1,\overline{E}_2,\overline{E}_3\right\}+\left\{V_2,\overline{E}_1,\overline{E}_4\right\} & & &
\left\{V_2,E_2,E_3\right\}+\left\{V_1,E_1,E_4\right\}\\
\left\{\overline{V}_1,\overline{E}_1,\overline{E}_4\right\}+\left\{\overline{V}_2,\overline{E}_2,\overline{E}_3\right\} & & &
\left\{\overline{V}_2,E_1,E_4\right\}+\left\{\overline{V}_1,E_2,E_3\right\}.
\end{aligned}
\end{equation}
and some concatenations that readily make sense in $\pi_1(\textrm{GL}^{+}(3;\matR))$ are
\begin{equation*}
\begin{array}{ccc}
V_1+V_2=V_2+V_1=1 && \overline{V}_1+\overline{V}_2=\overline{V}_2+\overline{V}_1=1 \\
E_1+E_2=E_2+E_1=1 && \overline{E}_1+\overline{E}_2=\overline{E}_2+\overline{E}_1=1 \\
E_1+E_3=E_3+E_1=0 && \overline{E}_1+\overline{E}_3=\overline{E}_3+\overline{E}_1=0 \\
E_2+E_4=E_4+E_2=0 && \overline{E}_2+\overline{E}_4=\overline{E}_4+\overline{E}_2=0 \\
E_3+E_4=E_4+E_3=1 && \overline{E}_3+\overline{E}_4=\overline{E}_4+\overline{E}_3=1;
\end{array}
\end{equation*}
see for instance Fig.~\ref{E1plusE2:fig}
\begin{figure}
    \begin{center}
    \includegraphics[scale=.6]{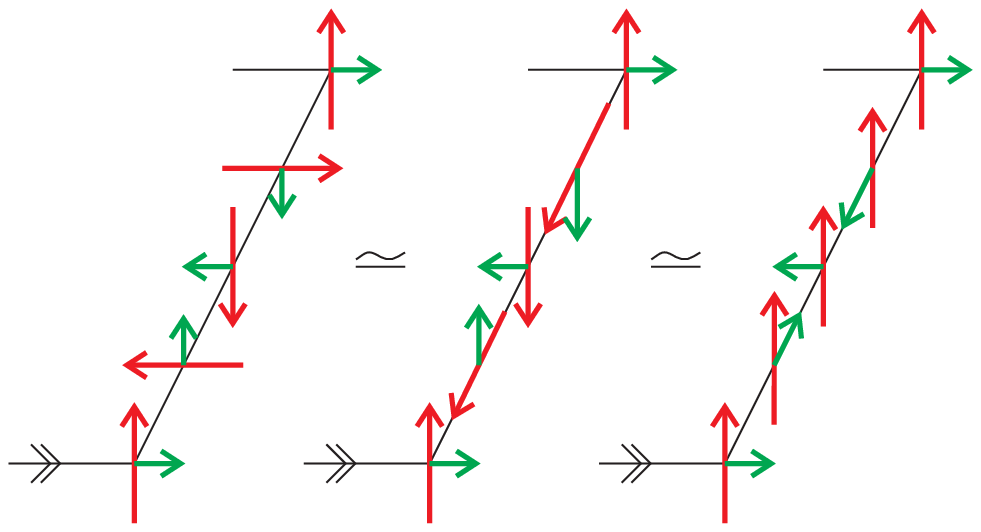}
    \end{center}
\vspace{-.5cm}\mycap{Proof that $E_1+E_2=1$.\label{E1plusE2:fig}}
\end{figure}
for $E_1+E_2=1$, where the concatenation is shown on the left and then homotoped to $1\in\pi_1(\textrm{GL}^{+}(3;\matR))$.

These relations (subject to the condition that all
$V_i,\overline{V}_i,E_j,\overline{E}_j$ should be assigned  $\pm{\textstyle{\frac12}}$ as a value) are equivalent to
\begin{equation}\label{first:values:eq}
\begin{array}{ccc}
V_1=V_2 && \overline{V}_1=\overline{V}_2 \\
E_1=E_2=-E_3=-E_4 && \overline{E}_1=\overline{E}_2=-\overline{E}_3=-\overline{E}_4
\end{array}\end{equation}
(note that the relations $E_4=-E_1$, $E_3=-E_2$, $\overline{E}_4=-\overline{E}_1$, $\overline{E}_3=-\overline{E}_2$
come from the algebra but they are also geometrically clear). We now claim that
\begin{equation*}\label{hard:equality:eq}
E_1+\overline{V}_1+\overline{E}_2+V_2=1
\end{equation*}
which is proved in Fig.~\ref{hardproof:fig}.
\begin{figure}
    \begin{center}
    \includegraphics[scale=.6]{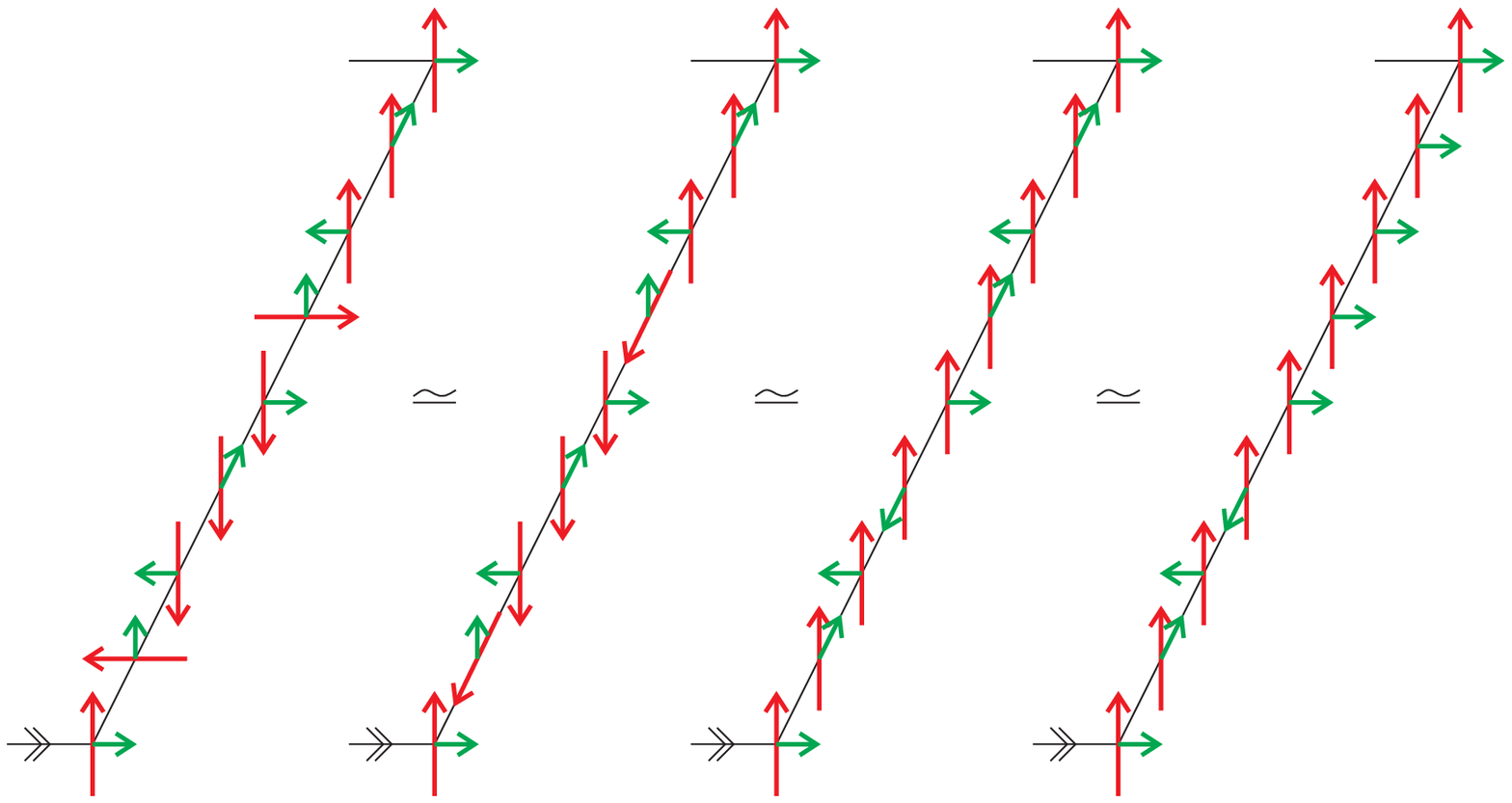}
    \end{center}
\vspace{-.5cm}\mycap{A concatenation giving $1\in\pi_1(\textrm{GL}^{+}(3;\matR))$.\label{hardproof:fig}}
\end{figure}
Taking into account~(\ref{first:values:eq}) the last condition is equivalent to any one of the following
\begin{equation}\label{second:values:eq}
\begin{array}{ccc}
V_1=E_1=\overline{E}_1=-\overline{V}_1 & \qquad & V_1=E_1=\overline{V}_1=-\overline{E}_1\\
\overline{V}_1=E_1=\overline{E}_1=-V_1 & \qquad & V_1=\overline{E}_1=\overline{V}_1=-E_1.
\end{array}
\end{equation}
Choosing one of the relations~(\ref{second:values:eq}) and combining it with~(\ref{first:values:eq})
one can now compute the correct value
of any possible concatenation according to~(\ref{possible:conc:eq}).
Let us now choose $V_1=E_1=\overline{E}_1=+{\textstyle{\frac12}}$ and $\overline{V}_1=-{\textstyle{\frac12}}$, and note that
the concatenation rules~(\ref{possible:conc:eq}) imply that the total number of
$V_1,V_2,\overline{V}_1,\overline{V}_2$ found along $\partial R$ is even (see also below).
The desired computation rule and the last assertion of the statement easily follow.
\end{proof}

\subsection{Remarks on the computation of the obstruction}
At the end of the proof of Proposition~\ref{real:obstruction:computation:prop} one can also choose
$V_1=\overline{V}_1=E_1=+{\textstyle{\frac12}}$ and $\overline{E}_1=-{\textstyle{\frac12}}$, which implies that
$\alpha(P,\omega,b)$ can be also computed by decorating the attaching circles of the regions as in Fig.~\ref{Cmethod:fig}.
\begin{figure}
    \begin{center}
    \includegraphics[scale=.6]{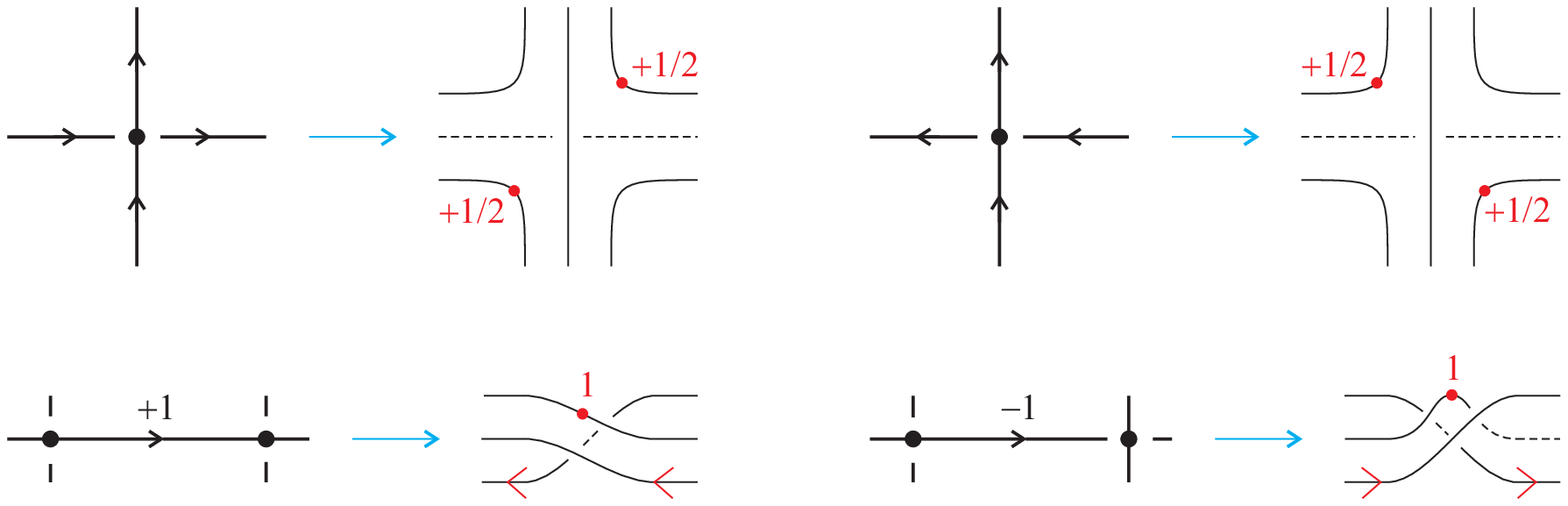}
    \end{center}
\vspace{-.5cm}\mycap{Alternative method to compute $\alpha(P,\omega,b)$.\label{Cmethod:fig}}
\end{figure}
More generally,
if we indicate by $c_i,\overline{c_i}$ the number of configurations $C_i,\overline{C_i}$ along $\partial R$, we have
that $\alpha_0(R)$ is equal to $1$ plus
\begin{eqnarray*}
& & {\textstyle{\frac12}}\left(
    +\left(v_1+v_2\right)
    +\left(\overline{v}_1+\overline{v}_2\right)
    +\left(e_1+e_2-e_3-e_4\right)
    -\left(\overline{e}_1+\overline{e}_2-\overline{e}_3-\overline{e}_4\right)
        \right)\\
&=& {\textstyle{\frac12}}\left(
    +\left(v_1+v_2\right)
    +\left(\overline{v}_1+\overline{v}_2\right)
    -\left(e_1+e_2-e_3-e_4\right)
    +\left(\overline{e}_1+\overline{e}_2-\overline{e}_3-\overline{e}_4\right)
        \right)\\
&=& {\textstyle{\frac12}}\left(
    +\left(v_1+v_2\right)
    -\left(\overline{v}_1+\overline{v}_2\right)
    +\left(e_1+e_2-e_3-e_4\right)
    +\left(\overline{e}_1+\overline{e}_2-\overline{e}_3-\overline{e}_4\right)
        \right)\\
&=& {\textstyle{\frac12}}\left(
    -\left(v_1+v_2\right)
    +\left(\overline{v}_1+\overline{v}_2\right)
    +\left(e_1+e_2-e_3-e_4\right)
    +\left(\overline{e}_1+\overline{e}_2-\overline{e}_3-\overline{e}_4\right)
        \right)
\end{eqnarray*}
and these expressions are recognized to be equivalent to each other because
$$\begin{array}{ccc}
v_1+v_2+\overline{e}_1+\overline{e}_2+\overline{e}_3+\overline{e}_4 & \qquad & v_1+v_2+e_1+e_2+e_3+e_4\\
\overline{v}_1+\overline{v}_2+\overline{e}_1+\overline{e}_2+\overline{e}_3+\overline{e}_4 & \qquad &
\overline{v}_1+\overline{v}_2+e_1+e_2+e_3+e_4
\end{array}$$
are all even numbers, thanks
to~(\ref{possible:conc:eq}). This implies that
$v_1+v_2+\overline{v}_1+\overline{v}_2$ is also even (as noted above), and
$e_1+e_2+e_3+e_4+\overline{e}_1+\overline{e}_2+\overline{e}_3+\overline{e}_4$
is even as well (which is clear, since it counts the number of up/down switches of $\nu$).

\begin{rem}\label{additivity:rem}
\emph{The main reason why we have defined $\varphi(P,\omega,b)=(\nu,\mu)$ not as $(\nu,\mu_0)$, but rather adding a full
twist to $\mu_0$ along unbranched edges, was to have additivity of the frames with respect to edge-fusion,
as explained in Proposition~\ref{numuadditivity:prop}. Coherently with this we now have that the obstruction
$\alpha(P,\omega,b)$ is also additive, namely it can be computed at the level of the graphs in $\calNtil$,
which would be false for $\alpha_0$.
Two examples of additivity (that again holds independently of the interpretation
of the $2$-valent vertices) are shown in Fig.~\ref{additivity:fig}.
\begin{figure}
    \begin{center}
    \includegraphics[scale=.6]{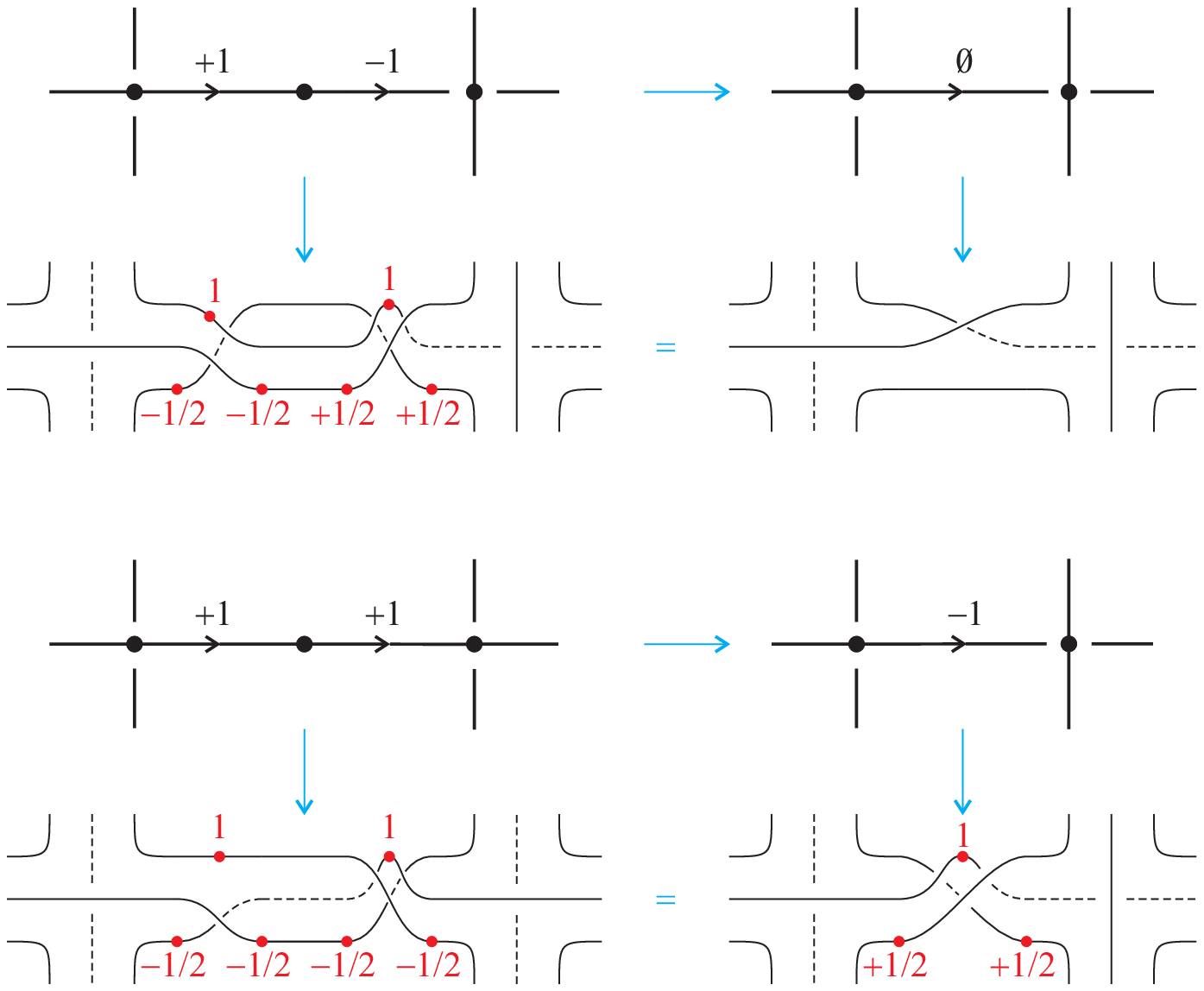}
    \end{center}
\vspace{-.5cm}\mycap{Additivity of the computation of $\alpha$.\label{additivity:fig}}
\end{figure}}
\end{rem}

\begin{rem}\label{sign}
\emph{Extending results of~\cite{BB2,BB3}, in~\cite{BBnew}
certain quantum hyperbolic invariants $\mathcal H_N(\mathcal P)$ have been
constructed for a variety of \emph{patterns} $\mathcal P$, with $N\geq 3$ an odd integer.
A \emph{pattern} consists of an oriented compact $3$-manifold $M$ with
(possibly empty) toric boundary, and an elaborated extra
structure on $M$, which includes a $\textrm{PSL}(2,\matC)$-character. Each
invariant is computed as a state sum over a suitably decorated weakly branched
triangulation of $M$ with some number $k$ of punctures, and it is
well-defined up to a phase {\it anomaly}.  Namely, for $N\equiv1\ (\text{mod}\ 4)$
up to multiplication by an $N$-th root of unity,
while for $N\equiv3\ (\text{mod}\ 4)$ up to multiplication by an $N$-th root of unity {\it and} a sign.
And it turns our that in the latter case the sign ambiguity can be removed by multiplying the state sum by
$(-1)^{k-\alpha(P,\omega,b)([P])}$, where
$(P,\omega,b)$ is the weakly branched spine dual to the triangulation, and
$[P]\in C_2(P;\zetadue)$ is the sum of all the regions of $P$.}
\end{rem}

\subsection{Spin structures from cochains}

We close this section with a result that dualizes to Proposition~\ref{alpha:tria:prop}.

\begin{prop}\label{coboundary:use:prop}
The class of $\alpha=\alpha(P,\omega,b)$ vanishes in $H^2\left(P;\zetadue\right)$.
For every $\beta\in C^1\left(P;\zetadue\right)$ such that $\delta\beta=\alpha$
a spin structure $s(P,\omega,b,\beta)$ is well-defined as the homotopy class
of the frame $(\nu,\beta(\mu))$ on $S(P)$, where $(\nu,\mu)=\varphi(P,\omega,b)$ and $\beta(\mu)$ is obtained
by giving a full twist to $\mu$
along all the edges $e$ of $P$ such that $\beta(e)=1$. Moreover
$s(P,\omega,b,\beta_0)=s(P,\omega,b,\beta_1)$ if and only if $\beta_0+\beta_1$
vanishes in $H^1\left(P;\zetadue\right)$.
\end{prop}

\begin{proof}
All three assertions are general topological facts.
To prove the first one, let
$(\overline{\nu},\overline{\mu})$ be any given spin structure on $M$,
namely a frame on $S(P)$ that extends to $P$ and is viewed up to homotopy on $S(P)$.
Homotoping $(\overline{\nu},\overline{\mu})$ we can suppose it coincides with $(\nu,\mu)$
at the vertices of $P$, so we can define $\beta\in C^1\left(P;\zetadue\right)$ where
$\beta(e)$ is the difference between $(\nu,\mu)$ and $(\overline{\nu},\overline{\mu})$ along $e$.
Since the obstruction to extending $(\overline{\nu},\overline{\mu})$ to a region $R$ of $P$
vanishes, we see that the obstruction $\alpha(R)$ to extending $(\nu,\mu)$ to $R$ is the
the sum of $\beta(e)$ for all the edges $e$ of $P$ contained in $\partial R$, namely
$\delta\beta=\alpha$.

The second assertion is now easy: if $\delta\beta=\alpha$ then the obstruction to
extending $(\nu,\beta(\mu))$ to $P$ vanishes.

Turning to the third assertion, it is first of all evident that if $v$ is a vertex of $P$ and
$\widehat{v}\in C^0\left(P;\zetadue\right)$ is its dual then the frames
on $S(P)$ carried by some $\beta$ with $\delta\beta=\alpha$ and by $\beta+\delta\widehat{v}$ are
homotopic on $S(P)$, with homotopy supported near $v$. Conversely, suppose
$\beta_0,\beta_1$ with $\delta\beta_0=\delta\beta_1=\alpha$ give frames $\left(\nu^{(0)},\mu^{(0)}\right)$
and $\left(\nu^{(1)},\mu^{(1)}\right)$ that are homotopic on $S(P)$ via
$\left(\nu^{(t)},\mu^{(t)}\right)_{t\in[0,1]}$. If $v$ is a vertex of $P$, by construction
$\left(\nu^{(0)},\mu^{(0)}\right)$ equals
$\left(\nu^{(1)},\mu^{(1)}\right)$ at $v$, so we can view $\left(\nu^{(t)},\mu^{(t)}\right)_{t\in[0,1]}$ at
$v$ as an element $\gamma(v)$ of $\pi_1(\textrm{GL}^{+}(3;\matR))=\zetadue$.
We then have $\gamma\in C^0\left(P;\zetadue\right)$ and $\beta_1=\beta_0+\delta\gamma$, whence the conclusion.
\end{proof}

Note that the previous result
is coherent with the known fact that the set of spin structures on $M$ is an affine space
over $H^1\left(P;\zetadue\right)=H^1\left(M;\zetadue\right)$.

\section{Spine moves preserving the spin structure}\label{moves:sec}
We will establish in this section the dual versions of
Propositions~\ref{vert:moves:tria:prop} to~\ref{tria:change:prop}.
From now on we will regard any $\beta\in C^1\left(P;\zetadue\right)$ such
that $\delta\beta=\alpha(P,\omega,b)$ up to coboundaries.
To discuss when two quadruples $(P,\omega,b,\beta)$ define the same
$s(P,\omega,b,\beta)$ we can then describe right to left how the quadruple
must change, and we have already dealt with the change of $\beta$.

Before proceeding further we introduce a convenient graphic encoding for the
quadruples $(P,\omega,b,\beta)$. Namely we define
$\calN_{\textrm{w}}$ as the set
of all graphs $\Gamma$ as in $\calN$, with the extra structure of \emph{weight} in $\zetadue$ attached to each edge of $\Gamma$.
A natural correspondence between $\calN_{\textrm{w}}$ and the set of all
quadruples $(P,\omega,b,\beta)$,
with $(P,\omega,b)$ as in Proposition~\ref{W:bijection:prop} and $\beta\in C^1\left(P;\zetadue\right)$,
is obtained by interpreting the weight of an edge as the value
of $\beta$ on it. Note that for $\Gamma\in\calN_{\textrm{w}}$ the edge colours
belong to $\zetatre=\{\emptyset,+1,-1\}$ and the weights to $\zetadue=\{0,1\}$, so
no confusion between colours and weights is possible. Colours $\emptyset$ and weights $0$
will often be omitted. We can similarly define $\calNtil_{\textrm{w}}$ as the set of
graphs in $\calNtil$ with weights in $\zetadue$ attached to the edges,
stipulating that weights sum up in $\zetadue$ when two edges are fused together.

\medskip

\subsection{The vertex moves}
The next result dualizes to Proposition~\ref{vert:moves:tria:prop}.

\begin{prop}\label{vertex:moves:real:prop}
Two graphs in $\calN_{\text{\emph{w}}}$ define quadruples $(P_j,\omega_j,b_j,\beta_j)$ for $j=0,1$ with
$P_1=P_0$, $\omega_1=\omega_0$ and $s(P_0,\omega_0,b_0,\beta_0)=s(P_1,\omega_1,b_1,\beta_1)$ if and only if
they are obtained from each other by repeated applications of the moves $I$ and $\duerom$ of Fig.~\ref{vertexmoves:fig}
\begin{figure}
    \begin{center}
    \includegraphics[scale=.6]{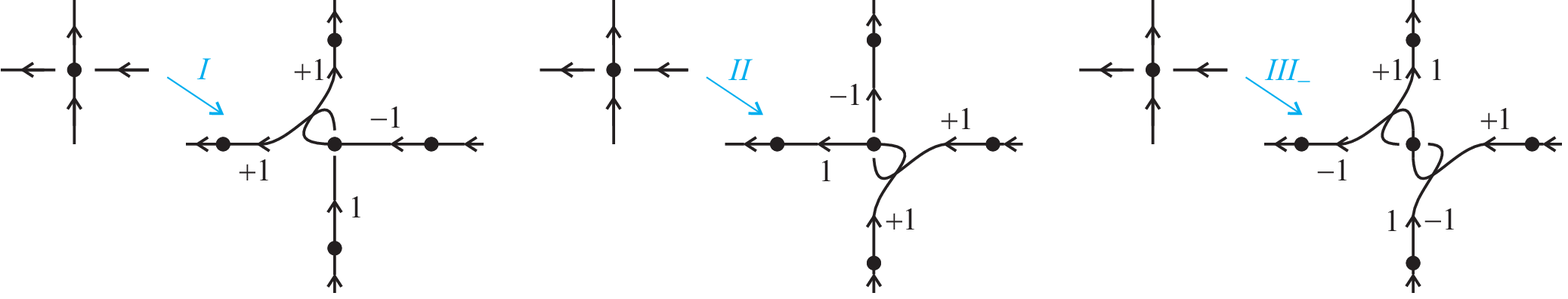}
    \end{center}
\vspace{-.5cm}\mycap{Moves that change the weak branching while preserving the pre-branching and the spin structure.
Recall that $\pm1$ are colours in $\zetatre$ while $1$ is a weight in $\zetadue$. \label{vertexmoves:fig}}
\end{figure}
(and their inverses, followed by the reduction from $\calNtil$ to $\calN$).
\end{prop}

\begin{proof}
We must prove that the moves $I$ and $\duerom$ generate all possible changes at a vertex $V$ of  a
weak branching compatible with a given pre-branching and, taking weights into account, that the associated spin structure
is preserved. For both indices $\varepsilon=\pm1$ of $V$ there
are 3 such possible changes; for $\varepsilon=-1$ they are given by the moves $I$, $\duerom$ and
$\trerom_-$ (already shown in Fig.~\ref{vertexmoves:fig}), which can be realized as
$\trerom_-=I\cdot\overline{\duerom}=\duerom\cdot\overline{I}$, with $\overline{I}$ and $\overline{II}$ the
inverses of $I$ and $\duerom$, and products written with the move applied first on the left;
for $\varepsilon=+1$ the 3 possible changes are given by $\overline{I}$, $\overline{\duerom}$
and $\trerom_+=\overline{\duerom}\cdot I=\overline{I}\cdot\duerom$.

It is then sufficient to show that the moves $I$ and $\duerom$ correctly represent one change
of weak branching and preserve the spin structure, which we will do explicitly only for $I$.
Ignoring the frame, the proof that $I$ preserves the pre-branched spine
is contained in Fig.~\ref{vertexproofnew:fig}-left.
\begin{figure}
    \begin{center}
    \includegraphics[scale=.6]{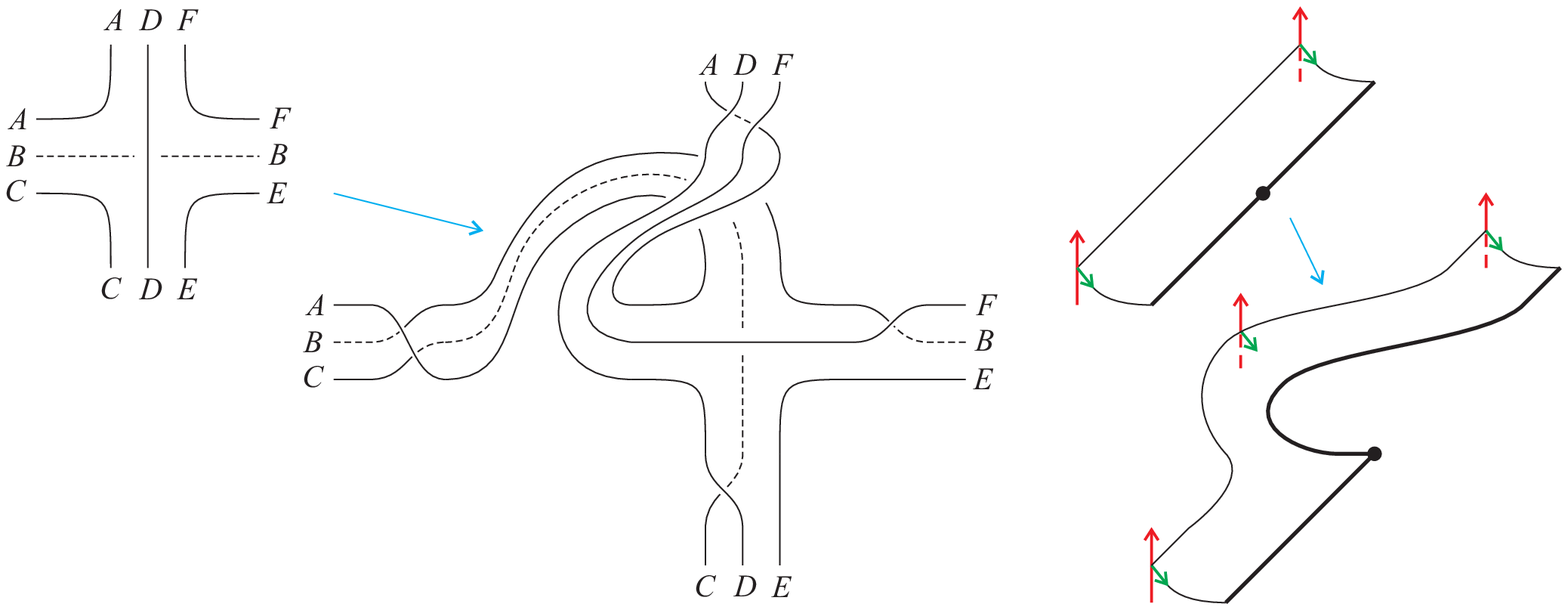}
    \end{center}
\vspace{-.5cm}\mycap{Left: move $I$ preserves the pre-branched spine.
Right: the frame $(\nu,\mu_0)$ is unchanged under move $I$ on the region $D$.\label{vertexproofnew:fig}}
\end{figure}
Turning to the frames, thanks to Proposition~\ref{numuadditivity:prop}, we can carry out a completely
local analysis. Moreover we note that locally before the move the frame $(\nu,\beta(\mu))$ coincides
with $(\nu,\mu_0)$, while after the move the frame $(\nu,\beta(\mu))$ is obtained from
$(\nu,\mu_0)$ by giving a full twist to $\mu_0$ along each of the $4$ involved edges
(three edges have colour $\pm1$ and weight $0$, the fourth edge has colour $\emptyset$ and weight $1$).
These four twists are induced by a homotopy, so it will be enough to show that the frames
$(\nu,\mu_0)$ before and after the move coincide up to homotopy. Showing this on a single global picture is too
complicated, so we confine ourselves to proving that $(\nu,\mu_0)$ is unchanged up to homotopy
separately on the boundary of each of the regions $A,B,C,D,E,F$ of Fig.~\ref{vertexproofnew:fig}-left. This is very easy for all
the regions except $A$, see for instance Fig.~\ref{vertexproofnew:fig}-right
for $D$. In Fig.~\ref{vertexproofframehard:fig}
\begin{figure}
    \begin{center}
    \includegraphics[scale=.6]{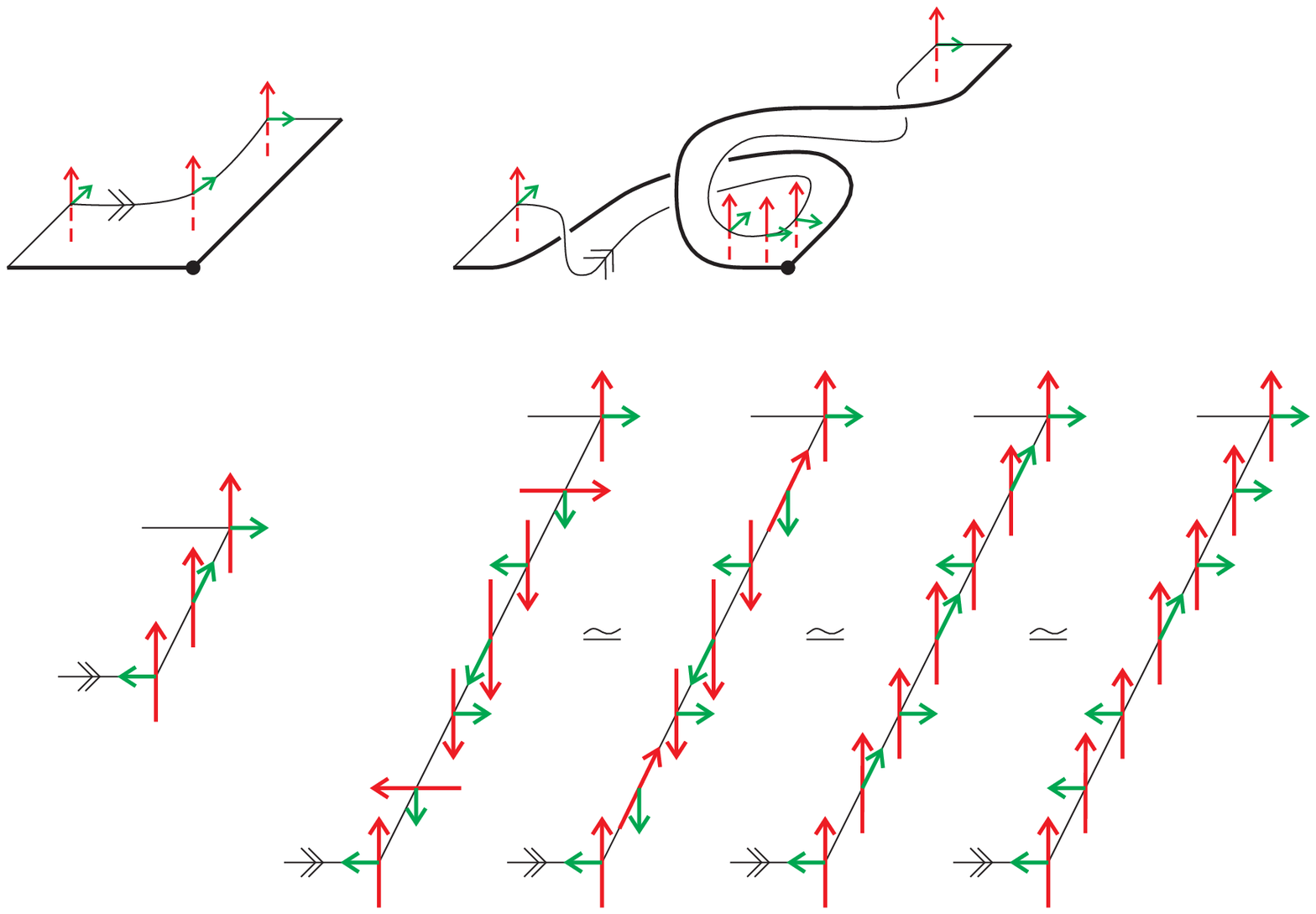}
    \end{center}
\vspace{-.5cm}\mycap{The frame $(\nu,\mu_0)$ is unchanged up to homotopy on $A$.
Left: before the move; right: after the move. Top: locally embedded configuration;
bottom: abstract configuration.\label{vertexproofframehard:fig}}
\end{figure}
we treat instead the case of the region $A$.
\end{proof}

\begin{rem}\label{vertex:obstruction:rem}
\emph{Let the change of weak branching on the pre-branched spine $(P,\omega)$ in move $I$
be given by $b\mapsto b'$.
The difference  $\Delta\alpha=\alpha(P,\omega,b)+\alpha(P,\omega,b')$ is then
computed locally, and Proposition~\ref{vertex:moves:real:prop} implies that $\Delta\alpha=\delta\widehat{e}$,
with $e$ as in Fig.~\ref{vertexobstruction:fig}.
\begin{figure}
    \begin{center}
    \includegraphics[scale=.6]{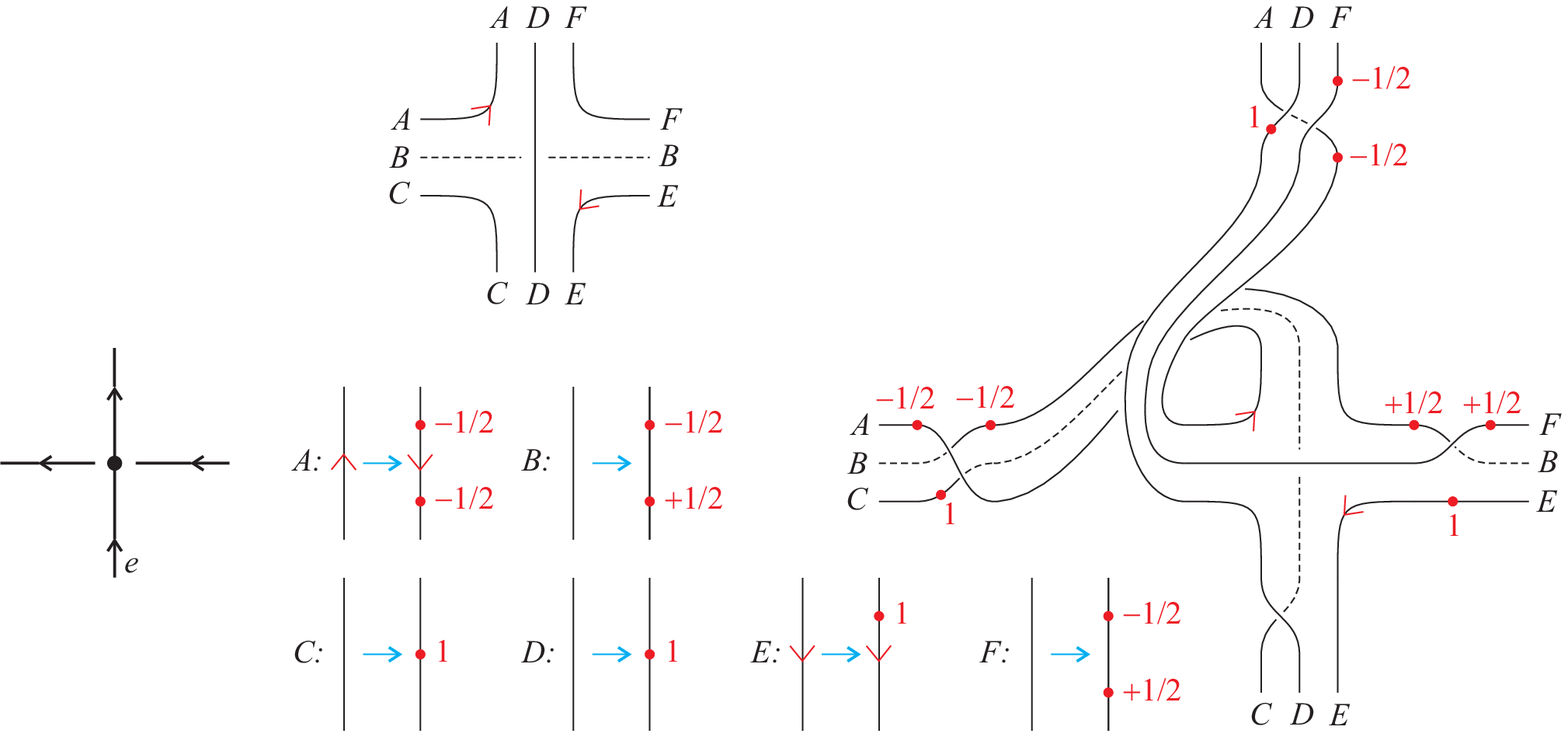}
    \end{center}
\vspace{-.5cm}\mycap{Variation of $\alpha$ with move $I$.\label{vertexobstruction:fig}}
\end{figure}
This fact can actually be checked directly, as in the rest of Fig.~\ref{vertexobstruction:fig},
since the picture shows that $\Delta\alpha$ is $0$ on $A,B,F$ and $1$ on $C,D,E$.}
\end{rem}

\subsection{The circuit move}
The next result dualizes to Proposition~\ref{vert:circuit:prop}.

\begin{prop}\label{circuit:move:prop}
Two graphs in $\calN_{\text{\emph{w}}}$ define quadruples $(P_j,\omega_j,b_j,\beta_j)$ for $j=0,1$ with
$P_1=P_0$ and $s(P_0,\omega_0,b_0,\beta_0)=s(P_1,\omega_1,b_1,\beta_1)$ if and only if
they are related by the moves of Proposition~\ref{vertex:moves:real:prop} plus
moves of the form $\Gamma\mapsto\Gamma'$, where:
\begin{itemize}
\item $\Gamma$ contains a simple oriented circuit $\gamma$ that at all its vertices is an overarc;
\item $\Gamma'$ is obtained from $\Gamma$ by reversing the orientation of the edges in $\gamma$
and adding $1$ to the weights of the edges of $\gamma$ whose ends have distinct indices.
\end{itemize}
\end{prop}

\begin{proof}
Suppose that $\omega_0$ and $\omega_1$ are distinct pre-branchings on the same spine $P$.
The union of the edges of $P$ on which $\omega_0$ and $\omega_1$ disagree can be expressed as
a disjoint union of simple circuits oriented by $\omega_0$. It is then sufficient to consider the situation
of two weak branchings $\omega,\omega'$ that differ only on a simple circuit $\gamma$ oriented by $\omega$, and then
iterate the procedure. Moreover, having already described  how to obtain from each other any two
pairs $(b,\beta)$ yielding the same spin structure on a given $(P,\omega)$, it is now sufficient
to find one specific weak branching $b$ on $(P,\omega)$ and one
$b'$ on $(P,\omega')$ and to describe a move $\beta\mapsto\beta'$ such that
$s(P,\omega,b,\beta)=s(P,\omega',b',\beta')$. This move will be that of the statement, implying the conclusion.
To describe the move we note that indeed via Proposition~\ref{vertex:moves:real:prop}
we can arrange so that $\gamma$ contains overarcs only in a graph $\Gamma\in\calN$ giving a weak branching $b$
on $(P,\omega)$. Examining Fig.~\ref{tetraindicesduals:fig}-left and~\ref{Wvertices:fig} one readily sees that the weak branching $b'$
obtained by reversing $\gamma$ is derived from $b$ by switching the orientation of the edge $v_2v_3$ in the tetrahedra dual to the edges
in $\gamma$. We are only left to show that the move $\beta\mapsto\beta'$ such that
$s(P,\omega,b,\beta)=s(P,\omega',b',\beta')$ consists in adding to $\beta$ the $1$-cochain $\Delta\beta$ given by the
duals of the edges in $\gamma$ having endpoints with distinct indices.
We will prove this in a slightly indirect way, in the spirit of Remark~\ref{vertex:obstruction:rem},
by computing $\Delta\alpha=\alpha(P,\omega,b)+\alpha(P,\omega',b')$ and showing that $\Delta\alpha=\delta(\Delta\beta)$.

We begin by noting that there are $9$ possible positions of a region $R$ with respect to an edge $e$ of $\gamma$,
as shown in Fig.~\ref{gammaedges:fig}.
\begin{figure}
    \begin{center}
    \includegraphics[scale=.6]{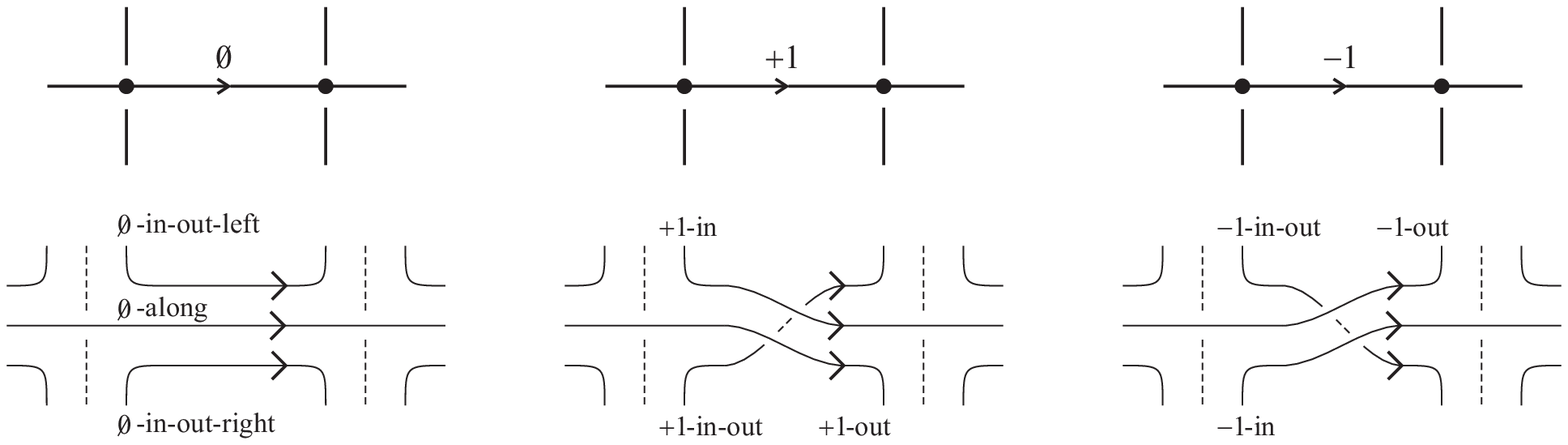}
    \end{center}
    \vspace{-.5cm}\mycap{The edges in the circuit $\gamma$ and the corresponding regions of $P$.\label{gammaedges:fig}}
\end{figure}
One can now check that the contributions carried by $e$ to $(\Delta\alpha)(R)$, depending on the indices
$\pm1/\pm1$ of the ends of $e$, are as given in the following table (with $\partial R$ oriented
as in in Fig.~\ref{gammaedges:fig}):
\begin{center}
\begin{tabular}{c||c|c|c|c}
                            &   $+1/+1$ & $-1/-1$   & $+1/-1$   & $-1/+1$   \\ \hline\hline
$\emptyset$-in-out-right    &   $0$     & $0$       & $1$       & $1$       \\ \hline
$\emptyset$-along           &   $0$     & $0$       & $0$       & $0$       \\ \hline
$\emptyset$-in-out-left     &   $0$     & $0$       & $1$       & $1$       \\ \hline
$+1$-in-out                 &   $0$     & $0$       & $1$       & $1$       \\ \hline
$+1$-in                     &   $1$     & $0$       & $1$       & $0$       \\ \hline
$+1$-out                    &   $1$     & $0$       & $0$       & $1$       \\ \hline
$-1$-in-out                 &   $0$     & $0$       & $1$       & $1$       \\ \hline
$-1$-in                     &   $1$     & $0$       & $1$       & $0$       \\ \hline
$-1$-out                    &   $1$     & $0$       & $0$       & $1$
\end{tabular}
\end{center}
See Fig.~\ref{someDelta:fig}
\begin{figure}
    \begin{center}
    \includegraphics[scale=.6]{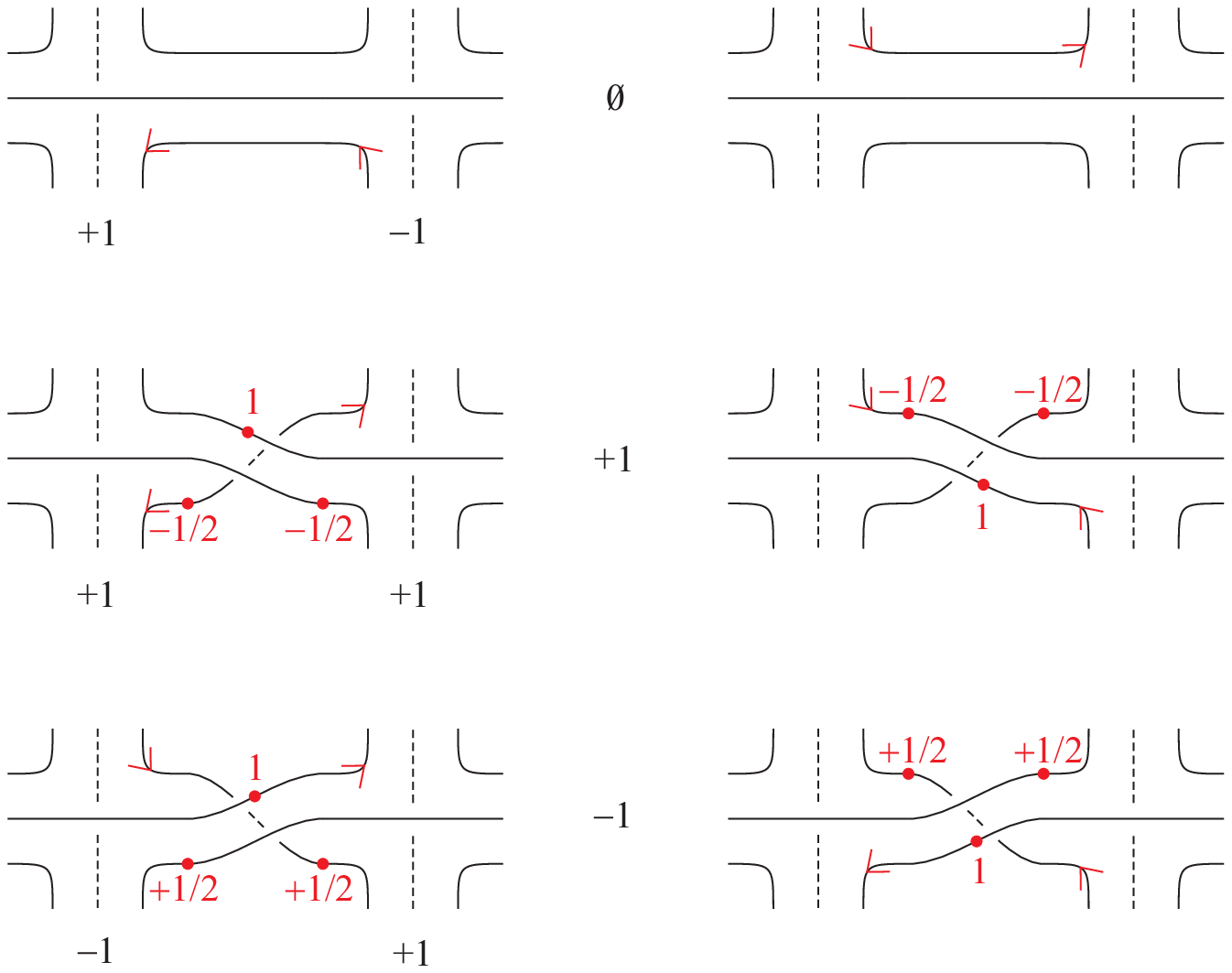}
    \end{center}
\vspace{-.5cm}\mycap{Three computations of the contribution of an edge $e$ in $\gamma$ to $(\Delta\alpha)(R)$.
The colour $\emptyset/+1/-1$ of $e$ (unchanged by the switch of $\gamma$) is shown in the middle.
On the left we see the indices of the ends of $e$ before the switch, and the local
computation of $\alpha(R)$. On the right the computation of $\alpha'(R)$ after the switch.\label{someDelta:fig}}
\end{figure}
for the explicit computation of some of these values.

To conclude we must now show that the total $(\Delta\alpha)(R)$ obtained by summing the
contributions given by the various edges of $\gamma$ equals (mod 2) the number of
edges in $\gamma$ visited by $\partial R$ and having ends with distinct indices.
If $\partial R$ visits only one edge, \emph{i.e.},~if it is of type in-out, the
conclusion is evident from the values in the table. Otherwise $(\Delta\alpha)(R)$ is the sum of only two possibly non-0 contributions,
one from the edge of $\gamma$ where $\partial R$ enters and one from the edge of $\gamma$ where it leaves.
More precisely, as one sees from the table, there is an ``in'' contribution depending only on the index of the vertex of $\gamma$ where
$\partial R$ enters (contribution $1$ for index $+1$ and contribution $0$ for index $-1$),
and an ``out'' contribution depending only on the index of the vertex of $\gamma$ where
$\partial R$ leaves (again, contribution $1$ for index $+1$ and contribution $0$ for index $-1$).
This implies that indeed $(\Delta\alpha)(R)$ has the desired value, and the proof is complete.
\end{proof}

\subsection{The bubble and the 2-3 move}
The next result dualizes to Proposition~\ref{tria:change:prop}.

\begin{prop}\label{spine:change:prop}
Two graphs in $\calN_{\text{\emph{w}}}$ define quadruples $(P_j,\omega_j,b_j,\beta_j)$ for $j=0,1$ with
$s(P_0,\omega_0,b_0,\beta_0)=s(P_1,\omega_1,b_1,\beta_1)$ if and only if
they are obtained from each other by the moves of
Propositions~\ref{vertex:moves:real:prop} and~\ref{circuit:move:prop} and those shown in Fig.~\ref{MPbubblemoves:fig}.
\begin{figure}
    \begin{center}
    \includegraphics[scale=.6]{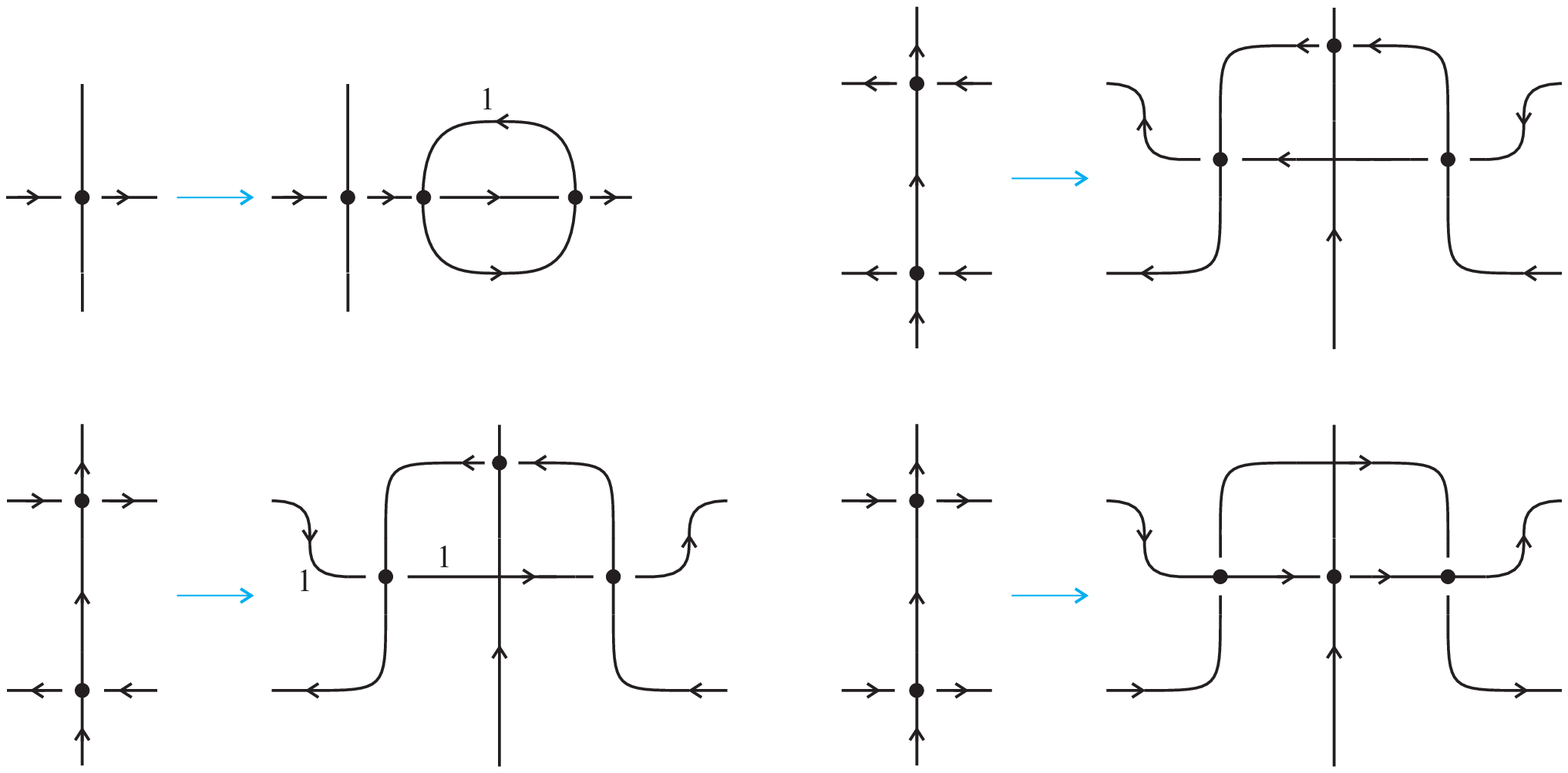}
    \end{center}
\vspace{-.5cm}\mycap{Moves on $\calNtil_{\textrm{w}}$ preserving the associated spin structure; the edges
entirely contained in the picture have colour $\emptyset$ and weight $0$.\label{MPbubblemoves:fig}}
\end{figure}
\end{prop}

\begin{proof}
Two special polyhedra are spines
(in the punctured sense) of the same manifold without boundary spheres if and only if
they are related by bubble and $2$-$3$ moves~\cite{Matveev:calculus, Pierg:calculus}.
It is then sufficient to
prove the following:
\begin{itemize}
\item Using the moves $I$ and $\duerom$ any edge $e$ with distinct ends $V_0,V_1$ of a graph in $\calN_{\textrm{w}}$ can
be transformed into one to which a move in Fig.~\ref{MPbubblemoves:fig} applies;
\item At the level of spines the moves in Fig.~\ref{MPbubblemoves:fig} translate the bubble and the $2$-$3$ move,
and at the level of quadruples $(P,\omega,b,\beta)$ represented by
graphs in $\calN_{\textrm{w}}$ the associated spin structure is unchanged under these moves.
\end{itemize}
With $\varepsilon$ being the index of a vertex,
the following steps establish the first assertion:
\begin{enumerate}
\item If $e$ is an underpass at some $V_j$, apply to each such $V_j$ the move $\trerom_{\varepsilon(V_j)}$
(with $\trerom_+$ the analogue of $\trerom_-$ for a vertex of index $+1$);
this allows to assume that $e$ is an overpass at $V_0$ and $V_1$;
\item If the colour of $e$ is now $+1$, act as follows:
\begin{enumerate}
\item If $\varepsilon(V_0)=\varepsilon(V_1)=-1$, apply $I$ to $V_0$ and $\duerom$ to $V_1$;
\item If $\varepsilon(V_0)=-1$ and $\varepsilon(V_1)=+1$, apply $\overline{\duerom}$ to $V_1$;
\item If $\varepsilon(V_0)=+1$, apply $\overline{I}$ to $V_0$;
\end{enumerate}
\item If the colour of $e$ is now $-1$, act as follows:
\begin{enumerate}
\item If $\varepsilon(V_0)=\varepsilon(V_1)=+1$, apply $\overline{I}$ to $V_0$ and $\overline{\duerom}$ to $V_1$;
\item If $\varepsilon(V_0)=+1$ and $\varepsilon(V_1)=-1$, apply $\duerom$ to $V_1$;
\item If $\varepsilon(V_0)=-1$, apply $I$ to $V_0$;
\end{enumerate}
\item The colour of $e$ is now $\emptyset$, and we want to exclude the case $\varepsilon(V_0)=+1$ and
$\varepsilon(V_1)=-1$, for which we apply $\overline{I}$ to $V_0$
and $\duerom$ to $V_1$;
\item Up to coboundaries we turn the weight of $e$ to $0$.
\end{enumerate}
For the second assertion, once again we start by an indirect
argument in the spirit of Remark~\ref{vertex:obstruction:rem}, showing that the weights appearing in the moves
compensate for the variation of the obstruction $\alpha\in C^2\left(P;\zetadue\right)$, which is done in Fig.~\ref{MPbubbleproof:fig}.
\begin{figure}
    \begin{center}
    \includegraphics[scale=.6]{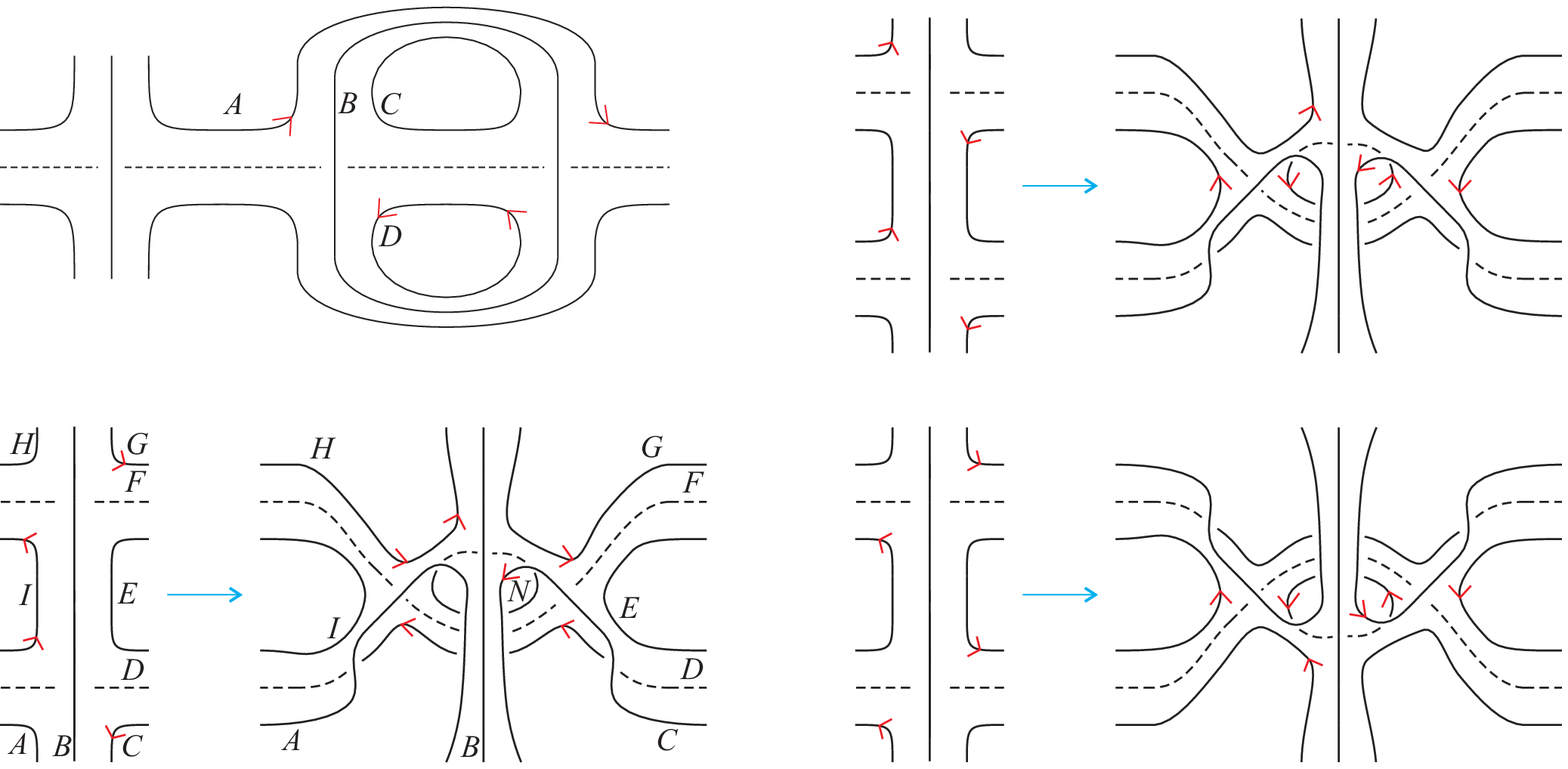}
    \end{center}
\vspace{-.5cm}\mycap{Proof that under the moves of Fig.~\ref{MPbubblemoves:fig} the change in the obstruction $\alpha$ is compensated
by the weights on the edges. For the two moves on the right this is easy: on all the regions that survive $\alpha$
keeps the same value, and on the newborn region it has value $0$. For the top-left move
$\Delta\alpha$ is $1$ on $A,B,C$ and $0$ on $D$, while for the bottom-left move $\Delta\alpha$ is $0$ on $A,B,C,E,F,G$
and $1$ on $D,H,I,N$, and indeed for both cases these values are given by the weights in the moves.
\label{MPbubbleproof:fig}}
\end{figure}
A more direct arguments for the move of Fig.~\ref{MPbubblemoves:fig}-top/left
is carried out in Fig.~\ref{bubbleproof:fig};
\begin{figure}
    \begin{center}
    \includegraphics[scale=.6]{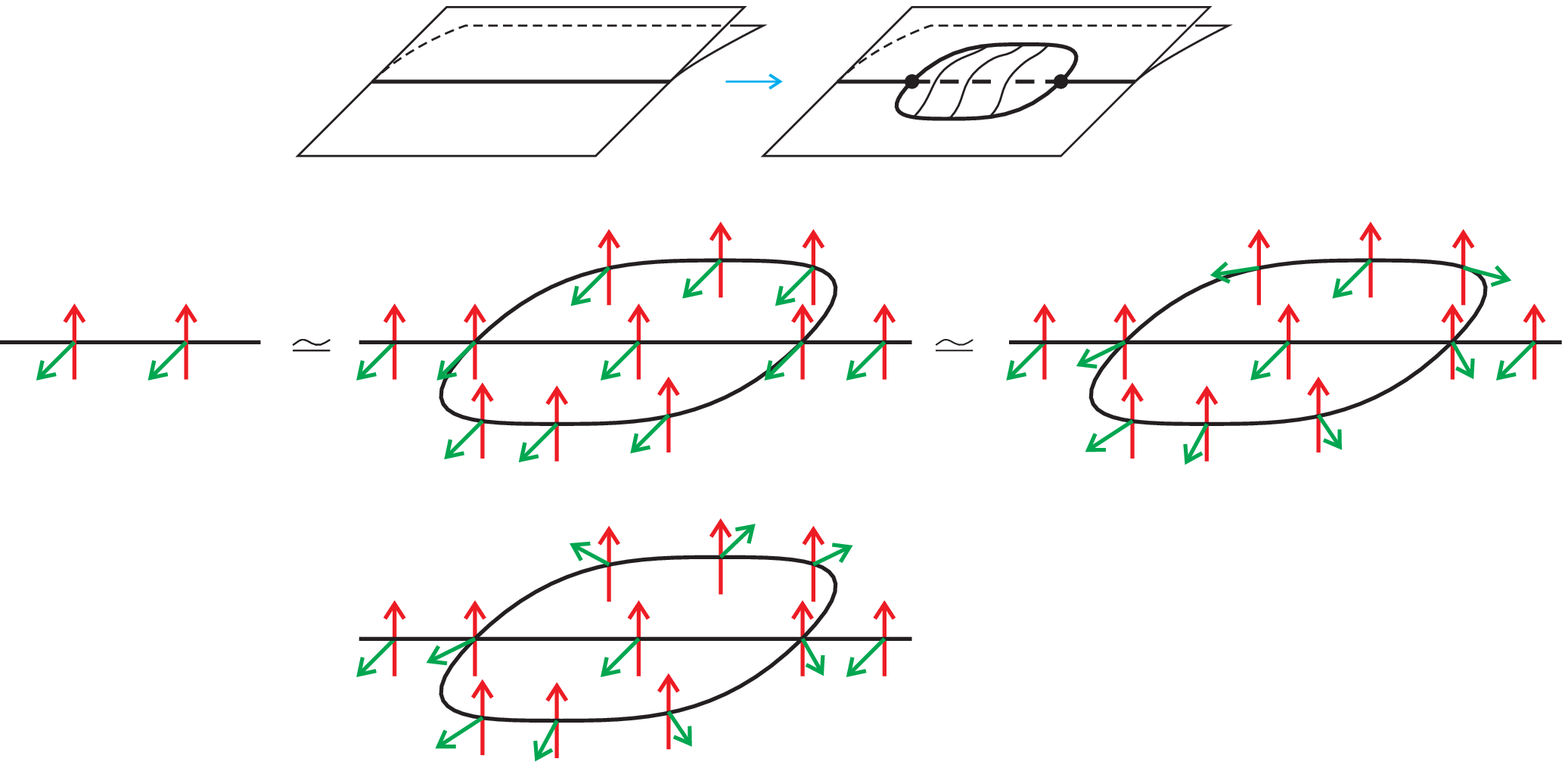}
    \end{center}
\vspace{-.5cm}\mycap{Top: the branched bubble move. Middle: frames carrying the same spin structure before and after the move.
Bottom: the frame carried by the spine after the move, that becomes the previous one
taking into account the weight $1$ appearing in the move of Fig.~\ref{MPbubblemoves:fig}-left.\label{bubbleproof:fig}}
\end{figure}
for the other moves the argument follows from~\cite{LNM}.
\end{proof}

\section{Arbitrarily branched graphs\\ and the corresponding moves}\label{graphic:local:sec}

In this section we show that the global move of Proposition~\ref{circuit:move:prop} can be replaced, in a suitable sense,
by a simultaneous combination of local ones.

\subsection{Graphs representing arbitrarily branched triangulations}
We introduce now a set $\calA$ of decorated graphs via which we can encode an
\emph{arbitrarily branched} triangulation, namely a triangulation in
which each tetrahedron is endowed with a branching, without any compatibility whatsoever.
Each vertex of a graph $\Gamma$ in $\calA$ will be given a planar structure as in
Fig.~\ref{Wvertices:fig}, which corresponds to giving the dual tetrahedron a branching.
Note that each edge of $\Gamma$ then has an orientation defined at each of its ends.
We are left to choose colours for the edges of $\Gamma$ in order to encode the face-pairings,
or equivalently the attaching circles to $S(P)=\Gamma$ of the regions of the dual spine $P$.
To do so we note that dual to a germ $e$ of edge of $S(P)$ at some vertex
there is a branched triangle. We can now label by $0,1,2$
the vertices of this triangle according
to the number of incoming edges, and dually the germs of regions incident to $e$.
We show in Fig.~\ref{regionlabels:fig}-left (in a cross-section)
this abstract labeling rule,
\begin{figure}
    \begin{center}
    \includegraphics[scale=.6]{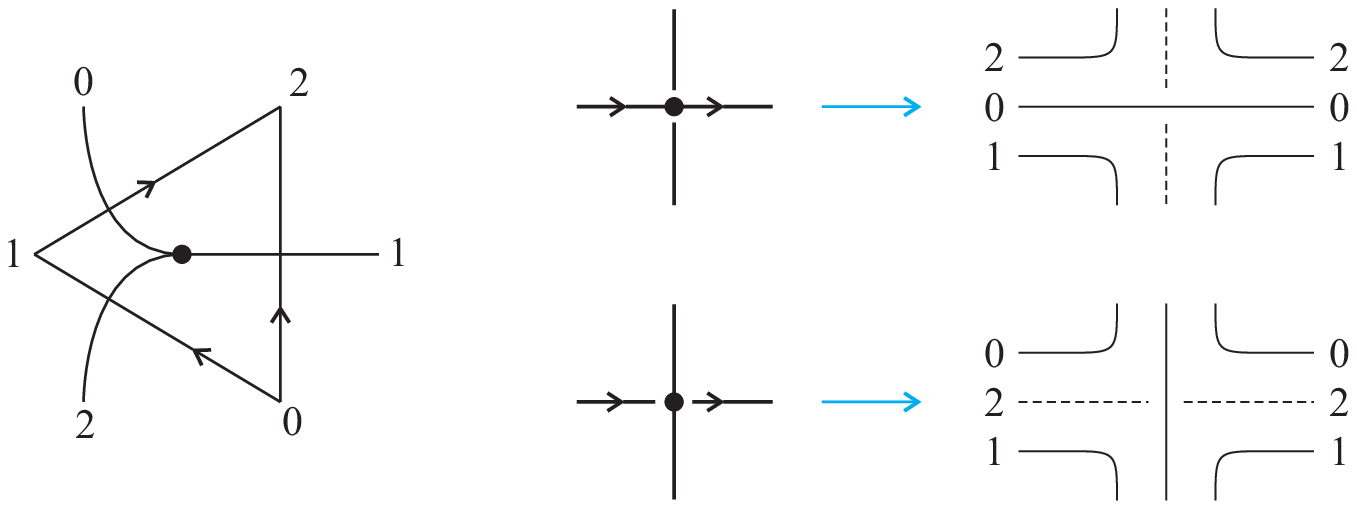}
    \end{center}
\vspace{-.5cm}\mycap{Labels for the germs of region near a (branched) vertex.\label{regionlabels:fig}}
\end{figure}
and in Fig.~\ref{regionlabels:fig}-right its concrete consequences.
One can now easily check the following:

\begin{lemma}
Let $e$ be an edge of $\Gamma$ and let $n(e)$ be the number of regions incident to $e$
having the same label at both ends of $e$.
\begin{itemize}
\item If the two ends of $e$ are consistently oriented then $n(e)=0$ or $n(e)=3$;
\item If the two ends of $e$ are inconsistently oriented then $n(e)=1$.
\end{itemize}
\end{lemma}

This implies that we can give an edge $e$ of $\Gamma$ the following colours in $\permu_3$
(see Fig.~\ref{Areconstruct:fig} for some examples):
\begin{itemize}
\item If the two ends of $e$ are consistently oriented
colour $$\sigma\in\permu^+_3=\{\emptyset,\ (0\,1\,2),\ (0\,2\,1)\}$$ if region $j$ at the first end of $e$ is matched to region
$\sigma(j)$ at the second end;
\item If the two ends of $e$ are inconsistently oriented, colour $$\tau\in\permu^-_3=\{(0\,1),\ (0\,2),\ (1\,2)\}$$ if region $j$ at one
end is matched to region $\tau(j)$ at the other end.
\end{itemize}

\begin{figure}
    \begin{center}
    \includegraphics[scale=.6]{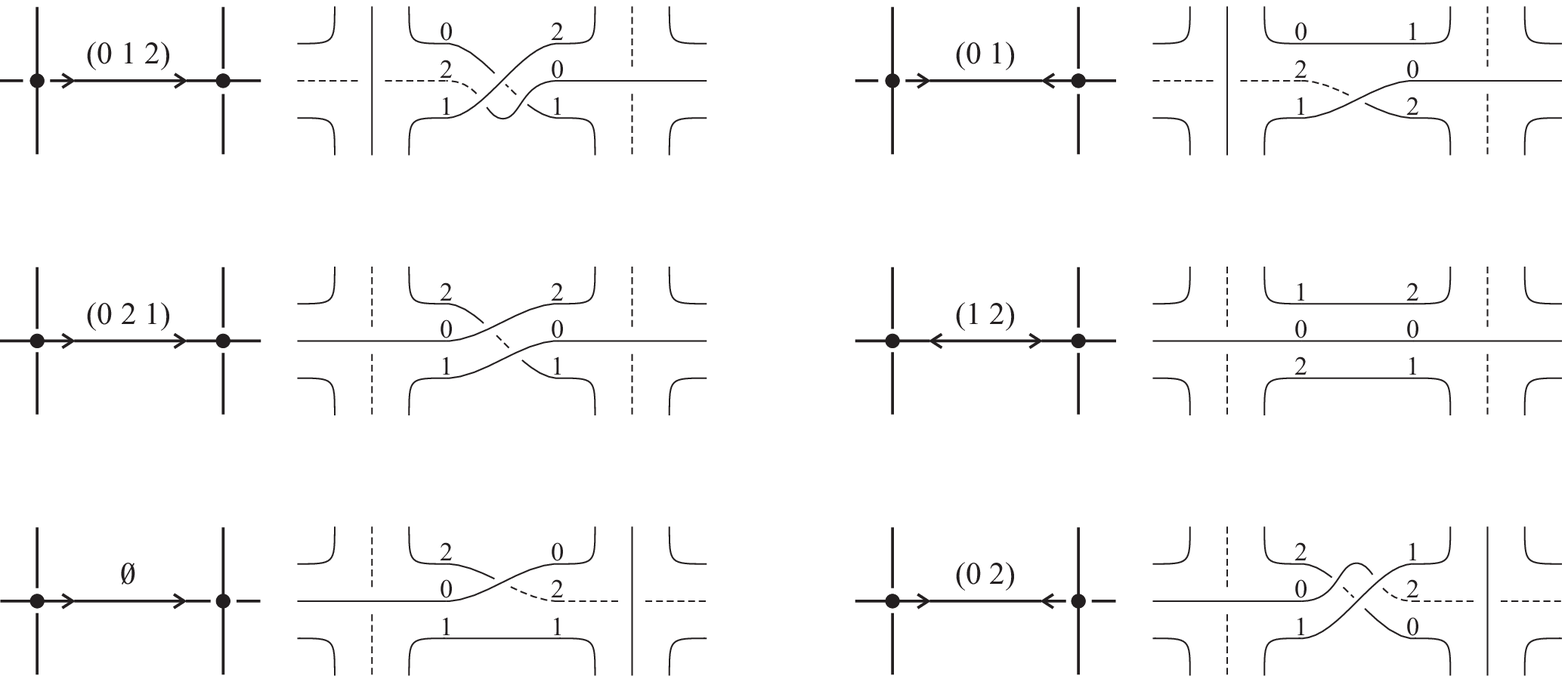}
    \end{center}
\vspace{-.5cm}\mycap{Meaning of the edge colours for a graph in $\calA$.\label{Areconstruct:fig}}
\end{figure}

\begin{rem}
\emph{A graph $\Gamma$ in $\calA$ defines a weakly branched triangulation if and only if
all the edges are consistently oriented. In this case $\Gamma$ is converted
into a graph in $\calN$ representing the same weakly branched triangulation by the colour-replacements
$(0\,1\,2)\mapsto+1$ and $(0\,2\,1)\mapsto-1$.}
\end{rem}

From now on we will call \emph{even} (respectively, \emph{odd})
an edge of a graph in $\calA$ with colour in
$\permu_3^+$ (respectively, in $\permu_3^-$), or, equivalently,
with consistently (respectively, inconsistently)
oriented ends.

\subsection{Graphs with multiply coloured edges}
As we did for $\calN$, to define moves on $\calA$ it is convenient to enlarge it
to some $\calAtil$ by allowing valence-2 vertices; edges are again decorated by
an orientation at each of their ends and a colour (in $\permu^+_3$
if the orientations match, in $\permu^-_3$ if they do not), but we also
insist that orientations should match across the valence-2 vertices. Note that
if we choose for each $2$-valent vertex of $\widetilde{\Gamma}\in\calAtil$ an interpretation as
    \includegraphics[scale=.25]{overinterpret.eps}
or as
    \includegraphics[scale=.25]{underinterpret.eps}
we can associate to $\widetilde{\Gamma}$ an arbitrarily branched special spine.

We now define a projection $\calAtil\to\calA$ by illustrating
in Fig.~\ref{calAtilmult:fig}
\begin{figure}
    \begin{center}
    \includegraphics[scale=.65]{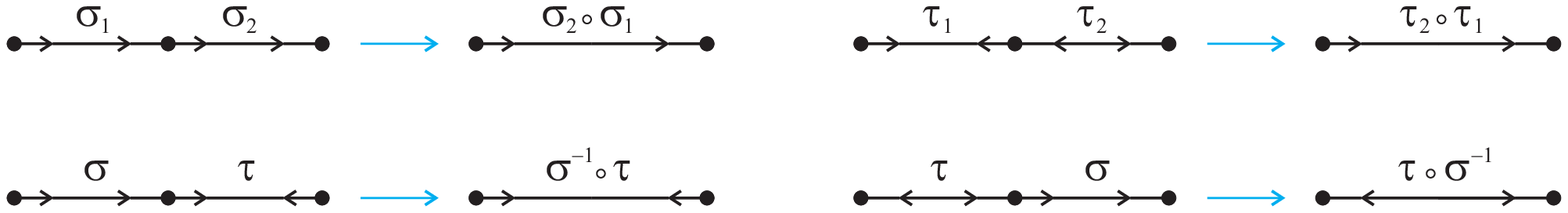}
    \end{center}
\vspace{-.5cm}\mycap{How to fuse together two edges of a graph in $\calAtil$.\label{calAtilmult:fig}}
\end{figure}
how fuse together two edges sharing
a valence-2 vertex; note that $\sigma,\sigma_1,\sigma_2\in\permu_3^+$ and
$\tau,\tau_1,\tau_2\in\permu_3^{-}$; moreover $\sigma^{-1}\compo\tau=\tau\compo\sigma$ and
$\tau\compo\sigma^{-1}=\sigma\compo\tau$, which gives alternative ways of expressing
the fusion rules. We have the following:

\begin{prop}\label{Assoc:prop}
The fusion rules of Fig.~\ref{calAtilmult:fig} are associative, so each graph $\widetilde{\Gamma}\in\calAtil$ defines
a unique $\Gamma\in\calA$. Moreover the arbitrarily branched spine associated to
$\widetilde{\Gamma}$ is well-defined regardless of the interpretation of the
valence-$2$ vertices as
    \includegraphics[scale=.25]{overinterpret.eps}
or
    \includegraphics[scale=.25]{underinterpret.eps},
and it coincides with the arbitrarily branched spine associated to $\Gamma$.
\end{prop}

The first assertion of this result follows from the second one, that can be established with some patience;
see some examples in Fig.~\ref{Aassoc:fig}.
\begin{figure}
    \begin{center}
    \includegraphics[scale=.6]{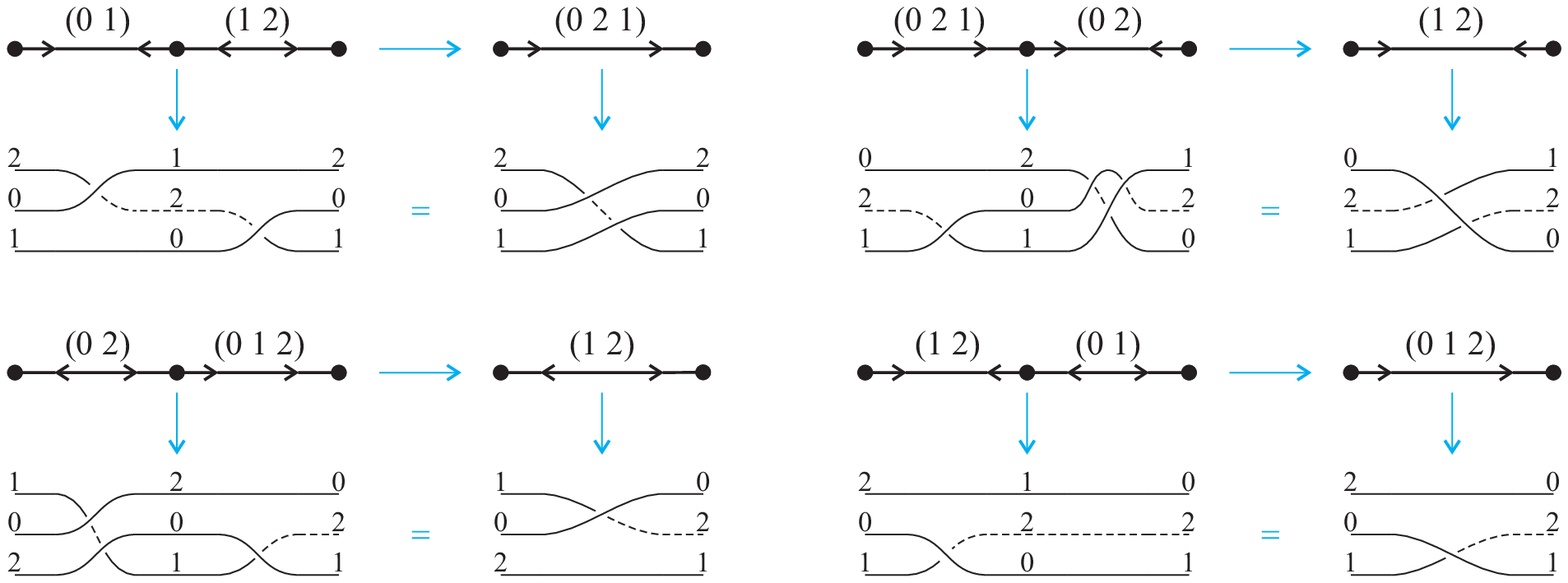}
    \end{center}
\vspace{-.5cm}\mycap{To each $\widetilde{\Gamma}\in\calAtil$ one can uniquely associate an arbitrarily branched special spine,
also given by the graph $\Gamma\in\calA$ obtained from $\widetilde{\Gamma}$ by
fusing the edges through valence-$2$ vertices.\label{Aassoc:fig}}
\end{figure}

\subsection{A new move}
Let us consider the move on graphs in $\calAtil$ described in Fig.~\ref{newmovenew:fig}-left.
In Fig.~\ref{newmovenew:fig}-right we show that the move preserves the spine (or triangulation) encoded by
the graph, while of course changing the arbitrary branching. The following result (that will also follow
from the rest of this section) is not difficult to show:
\begin{figure}
    \begin{center}
    \includegraphics[scale=.6]{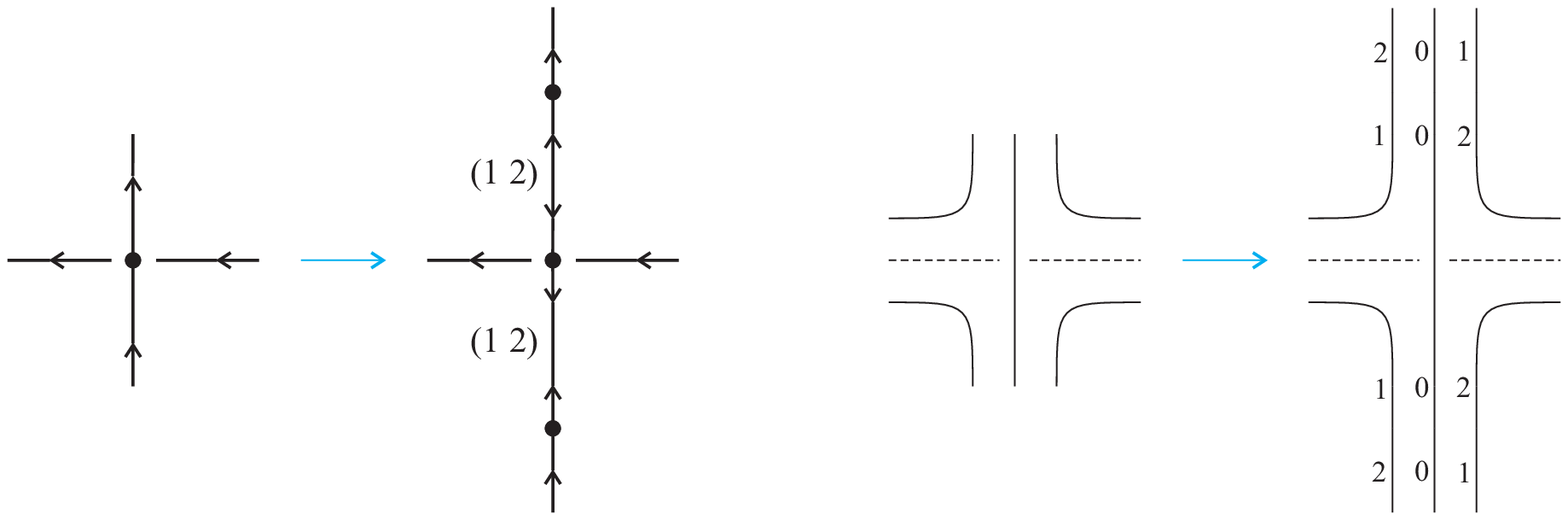}
    \end{center}
\vspace{-.5cm}\mycap{A move on $\calAtil$.\label{newmovenew:fig}}
\end{figure}

\begin{prop}
Any two arbitrary branchings on the same triangulations are related by compositions of the
moves $I$ and $\duerom$ (ignoring weights), that of Fig.~\ref{newmovenew:fig}, and their inverses.
\end{prop}

Since for a single tetrahedron there are 24 different branchings, this result means that
at each vertex using
the moves $I$ and $\duerom$ and that of Fig.~\ref{newmovenew:fig}
one can create all 24 possible configurations. See for instance Fig.~\ref{newmovebisnew:fig}.
\begin{figure}
    \begin{center}
    \includegraphics[scale=.6]{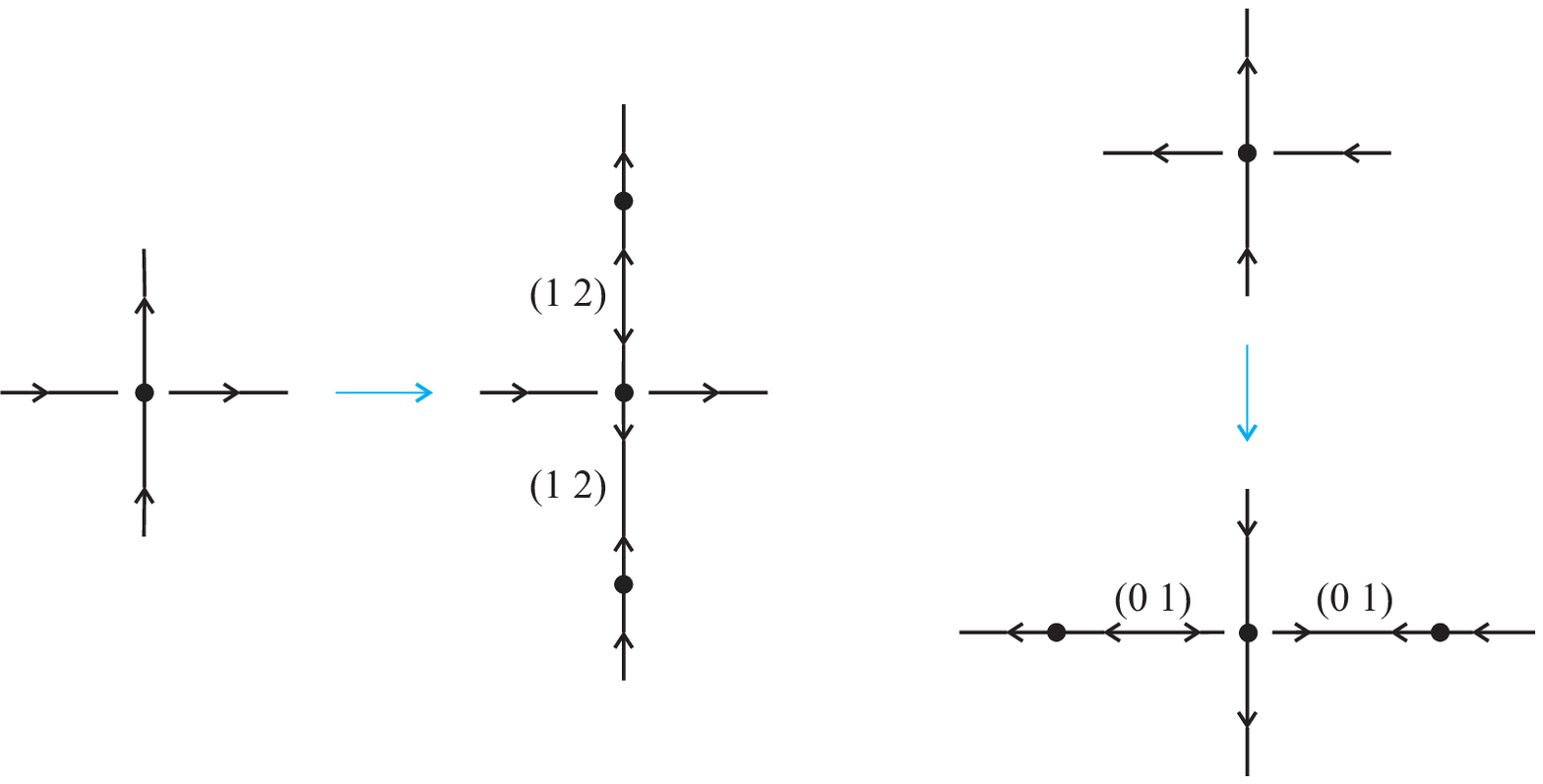}
    \end{center}
\vspace{-.5cm}\mycap{The inverse of the move of Fig.~\ref{newmovenew:fig} and one generated by those
in Fig.~\ref{newmovenew:fig} and~\ref{vertexmoves:fig}.\label{newmovebisnew:fig}}
\end{figure}

\subsection{Weighted graphs and weighted fusion}
We define $\calAtil_{\textrm{w}}$ as the set of graphs in $\calAtil$ with \emph{weights}
attached to the edges. The weight
of an edge $e$ is given by an \emph{internal orientation} and by a
\emph{numerical weight} in the group $G=\left(\frac12\cdot\matZ\right)/_{\!2\matZ}$, with the following restrictions:
\begin{itemize}
\item If $e$ is even then its internal orientation matches those at its ends (so it is not shown in the pictures)
and the numerical weight is $0$ or $1$;
\item If $e$ is odd the numerical weight is $\pm\frac12$.
\end{itemize}
Note that there is a natural inclusion $\calNtil_{\textrm{w}}\subset \calAtil_{\textrm{w}}$
The numerical part of a system of weights will be viewed up to $1$-coboundaries
with values in $\zetadue$, namely the numerical weights of all $4$ edges incident to a vertex
can simultaneously change by $1$.
We next define the weighted fusion rules of Fig.~\ref{weightedcalAtilmult:fig}.
\begin{figure}
    \begin{center}
    \includegraphics[scale=.6]{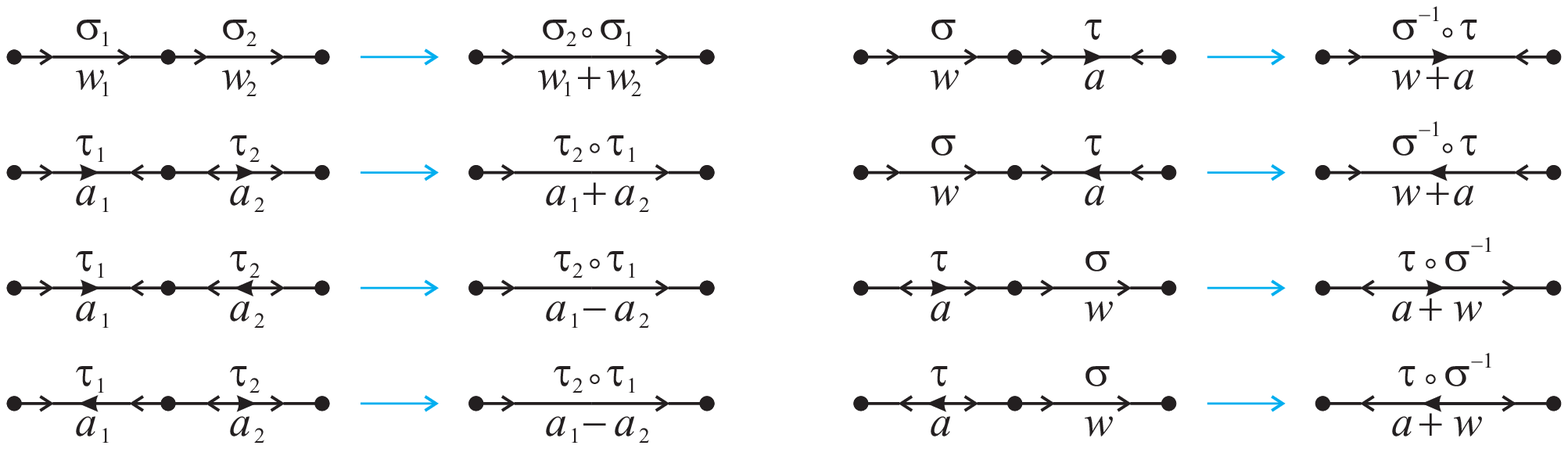}
    \end{center}
\vspace{-.5cm}\mycap{Edge fusion rules for graphs in $\calAtil_{\textrm{w}}$.\label{weightedcalAtilmult:fig}}
\end{figure}

\begin{rem}\label{not:all:mult:rem}
\emph{The fusion rules do not cover the case of two odd edges with internal orientations both
opposite to the external orientation after fusion, because this case will never occur for us.
For the fusion of two odd edges with discordant internal orientations, we note that
$a_1,a_2$ are $\pm{\textstyle{\frac12}}$, so $a_1-a_2=a_2-a_1$ in $\zetadue$.}
\end{rem}

The following fact, proved in Fig.~\ref{notassoc:fig},
\begin{figure}
    \begin{center}
    \includegraphics[scale=.6]{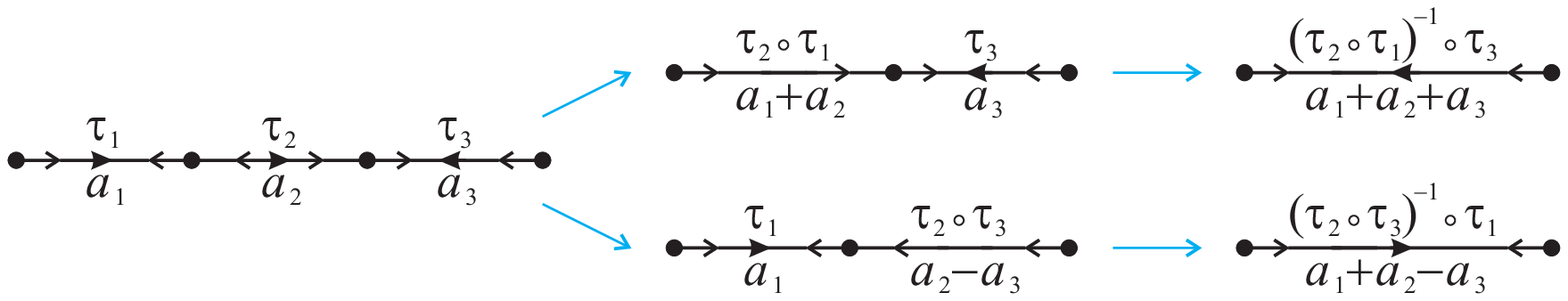}
    \end{center}
\vspace{-.5cm}\mycap{The fusion rules for graphs in $\calAtil_{\textrm{w}}$ are not associative.
In this example both the
internal orientation and the numerical weight $\pm{\textstyle{\frac12}}$ depend on the order in which fusions are performed;
note however that $(\tau_2\compo\tau_1)^{-1}\compo\tau_3=(\tau_2\compo\tau_3)^{-1}\compo\tau_1$, coherently
with the fact that the fusion rules for unweighted graphs \emph{are} associative.\label{notassoc:fig}}
\end{figure}
must be taken into account:

\begin{prop}
The weighted fusion rules of Fig.~\ref{weightedcalAtilmult:fig} are not associative.
\end{prop}

\subsection{Moves on weighted graphs}
We now introduce certain moves on $\calN_{\textrm{w}}$, to define which we also use $\calAtil_{\textrm{w}}$.
To begin we call \emph{elementary move} on $\calAtil_{\textrm{w}}$ one of
$I, \overline{I},\duerom,\overline{\duerom}, M, \overline{M}$
from Figg.~\ref{vertexmovesfromminusplus:fig}
\begin{figure}
    \begin{center}
    \includegraphics[scale=.6]{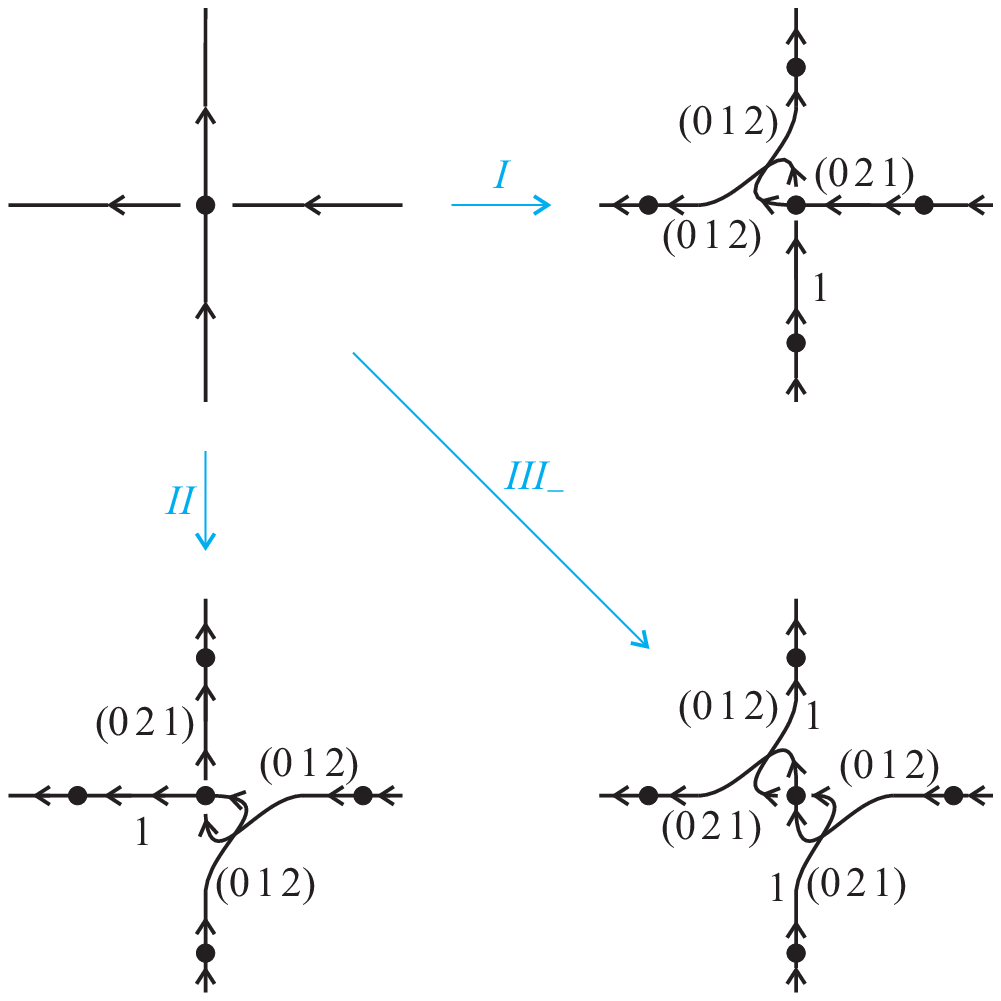}
\qquad
    \includegraphics[scale=.6]{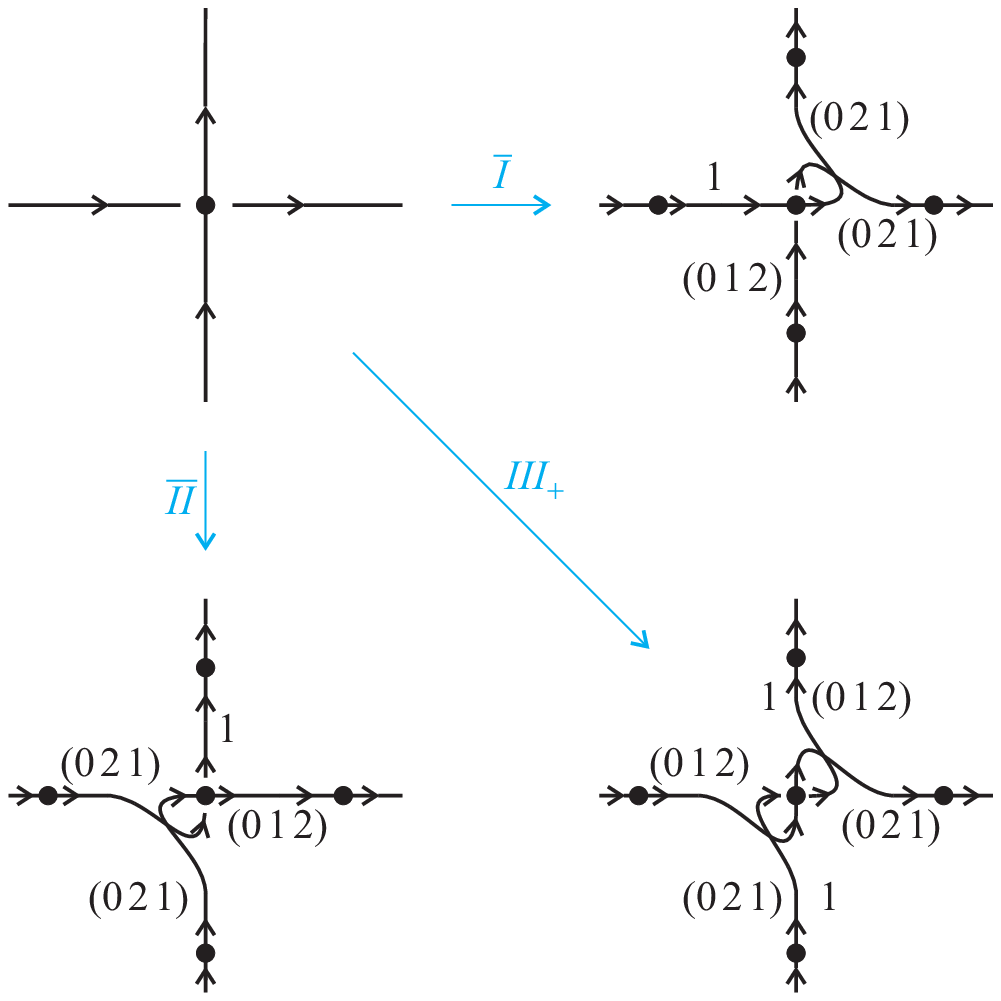}
    \end{center}
\vspace{-.5cm}\mycap{Moves on $\calAtil_{\textrm{w}}$ derived from those on $\calNtil_{\textrm{w}}$.\label{vertexmovesfromminusplus:fig}}
\end{figure}
and~\ref{MMbarNNbarwitharrows:fig}.
\begin{figure}
    \begin{center}
    \includegraphics[scale=.6]{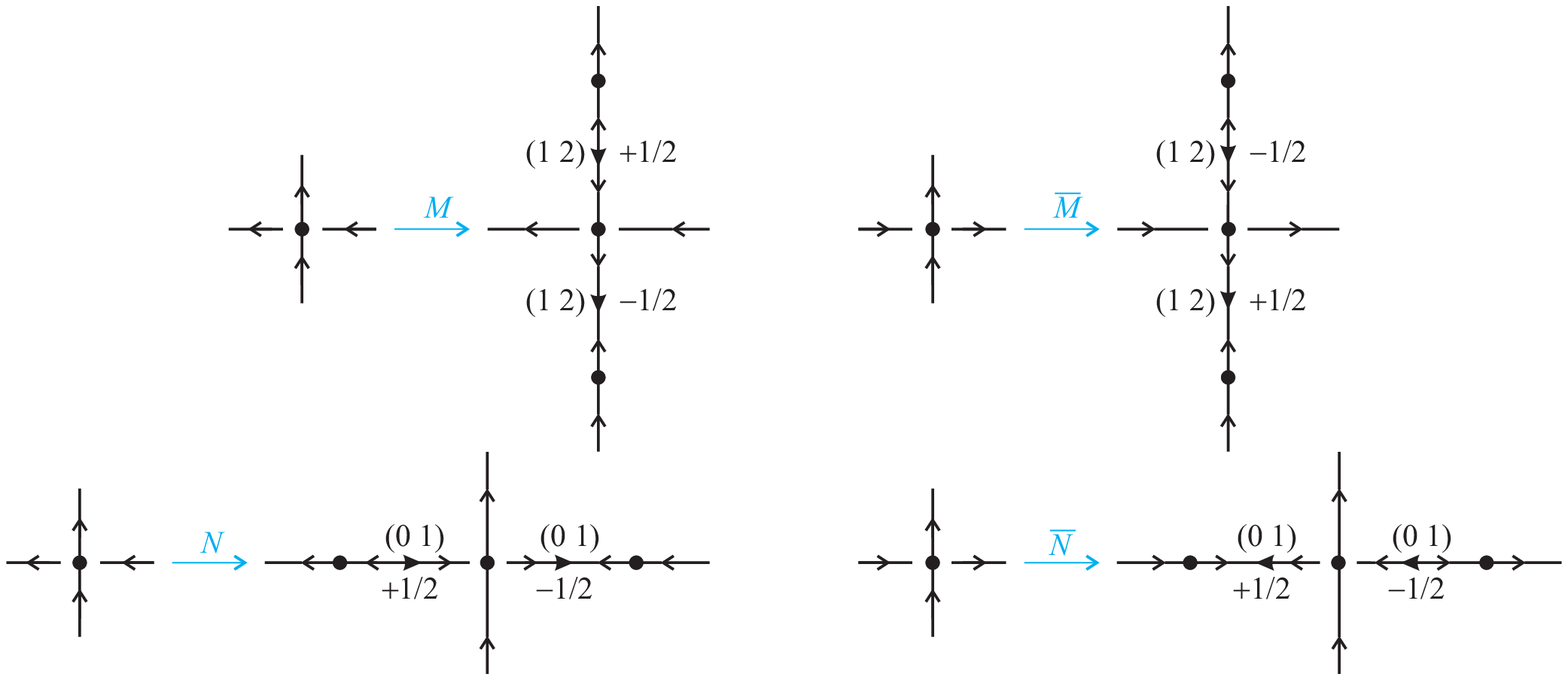}
    \end{center}
\vspace{-.5cm}\mycap{More moves on weighted graphs.\label{MMbarNNbarwitharrows:fig}}
\end{figure}
The pictures contain more moves whose r\^ole will be explained soon.

\begin{rem}
\begin{itemize}
\item \emph{In the symbols denoting the moves, overlining and subscripts are used to indicate
the type of index transition $\pm1\mapsto\pm1$;}
\item \emph{The moves $I$ and $\duerom$ are those of Fig.~\ref{vertexmoves:fig} and at the level of $\calN_{\textrm{w}}$, namely under the
\emph{associative} fusion rules for $\calNtil_{\textrm{w}}$, we have
$\overline{I}=I^{-1}$ and $\overline{\duerom}=\duerom^{-1}$; moreover
${\trerom}_-={I}\!\cdot\!\overline{{\duerom}}={\duerom}\!\cdot\!\overline{{I}}$
and ${\trerom}_+=\overline{{I}}\!\cdot\!{\duerom}=\overline{{\duerom}}\!\cdot\!{I}$;}
\item \emph{In the product of two moves that to the left applies first;
not all products make sense.}
\end{itemize}
\end{rem}

We now establish some results concerning relations between moves:

\begin{prop}\label{elem:eq:prop}
Consider a vertex as in Fig.~\ref{Wvertices:fig}, apply to it one of the following combination
of weighted moves, and locally apply near the vertex the weighted fusion rules of
Fig.~\ref{weightedcalAtilmult:fig}; then the result is the same as indicated:
$$M\!\cdot\! \overline{M}=\text{\emph{id}}_-,\quad
N\!\cdot\! \overline{N}=\text{\emph{id}}_-,\quad
\overline{M}\!\cdot\! M=\text{\emph{id}}_+,\quad
\overline{N}\!\cdot\! N=\text{\emph{id}}_+,$$
$$\trerom_-\!\cdot\! M=N\!\cdot\! \trerom_+,\quad
\trerom_+\!\cdot\! \overline{M}=\overline{N}\!\cdot\! \trerom_-,\quad
M\!\cdot\! \overline{N}=N\!\cdot\!\overline{M},\quad
\overline{M}\!\cdot\! N=\overline{N}\!\cdot\! M.$$
\end{prop}

\begin{rem}
\emph{Since the fusion rules in $\calAtil_{\textrm{w}}$ are \emph{not} associative,
these equalities \emph{do not} imply that at the level $\calA_{\textrm{w}}$ we have
the relations
$$\overline{M}=M^{-1},\quad \overline{N}=N^{-1},\quad N=\trerom_-\!\cdot\! M\!\cdot\! \trerom_+$$
but these relations do make sense and hold in a restricted context, see below.}
\end{rem}

The proofs of some of the equalities in Proposition~\ref{elem:eq:prop} are given in
Figg.~\ref{MMbareq:fig} to~\ref{MNcommute:fig}; they all crucially use the
weighted fusion rules of Fig.~\ref{weightedcalAtilmult:fig} and the convention that
weights are viewed up to $\zetadue$-coboundaries; the other proofs are similar.
\begin{figure}
    \begin{center}
    \includegraphics[scale=.6]{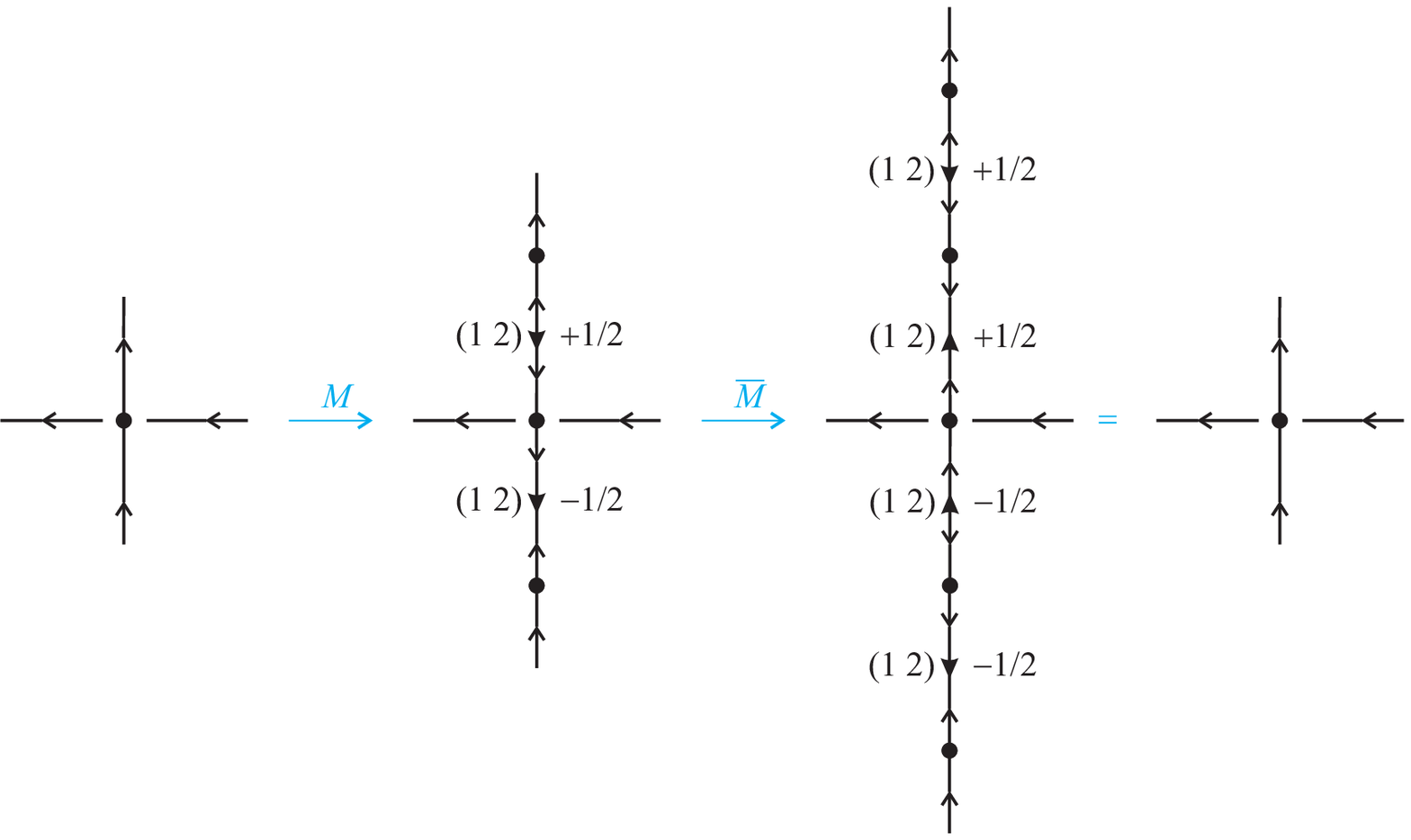}
    \end{center}
\vspace{-.5cm}\mycap{Proof that $M\!\cdot\! \overline{M}=\textrm{id}_-$ under local application of weighted fusion.\label{MMbareq:fig}}
\end{figure}
\begin{figure}
    \begin{center}
    \includegraphics[scale=.6]{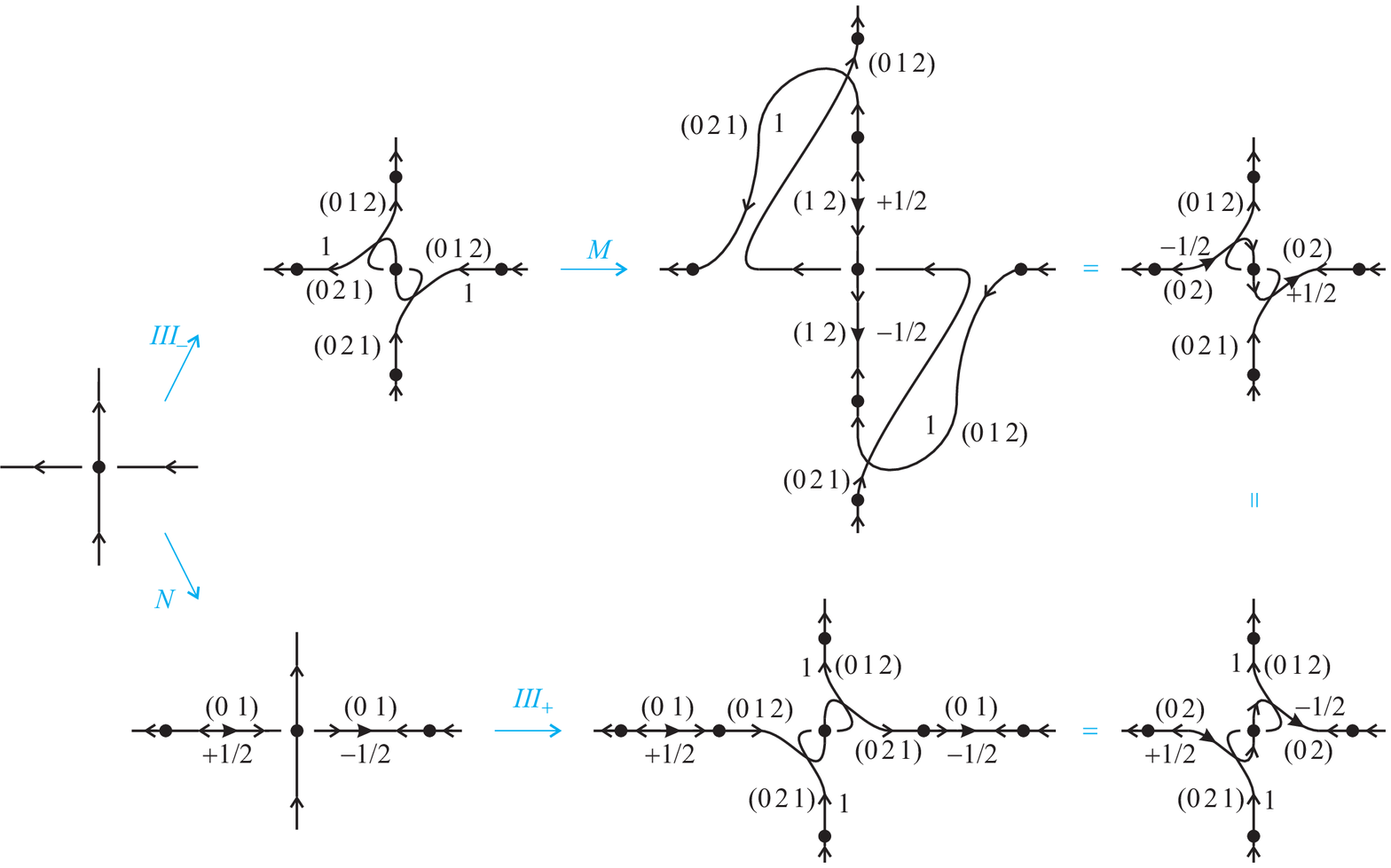}
    \end{center}
\vspace{-.5cm}\mycap{Proof that $\trerom_-\!\cdot\! M=N\!\cdot\! \trerom_+$ under local application of weighted fusion.\label{MgeneratesN:fig}}
\end{figure}
\begin{figure}
    \begin{center}
    \includegraphics[scale=.6]{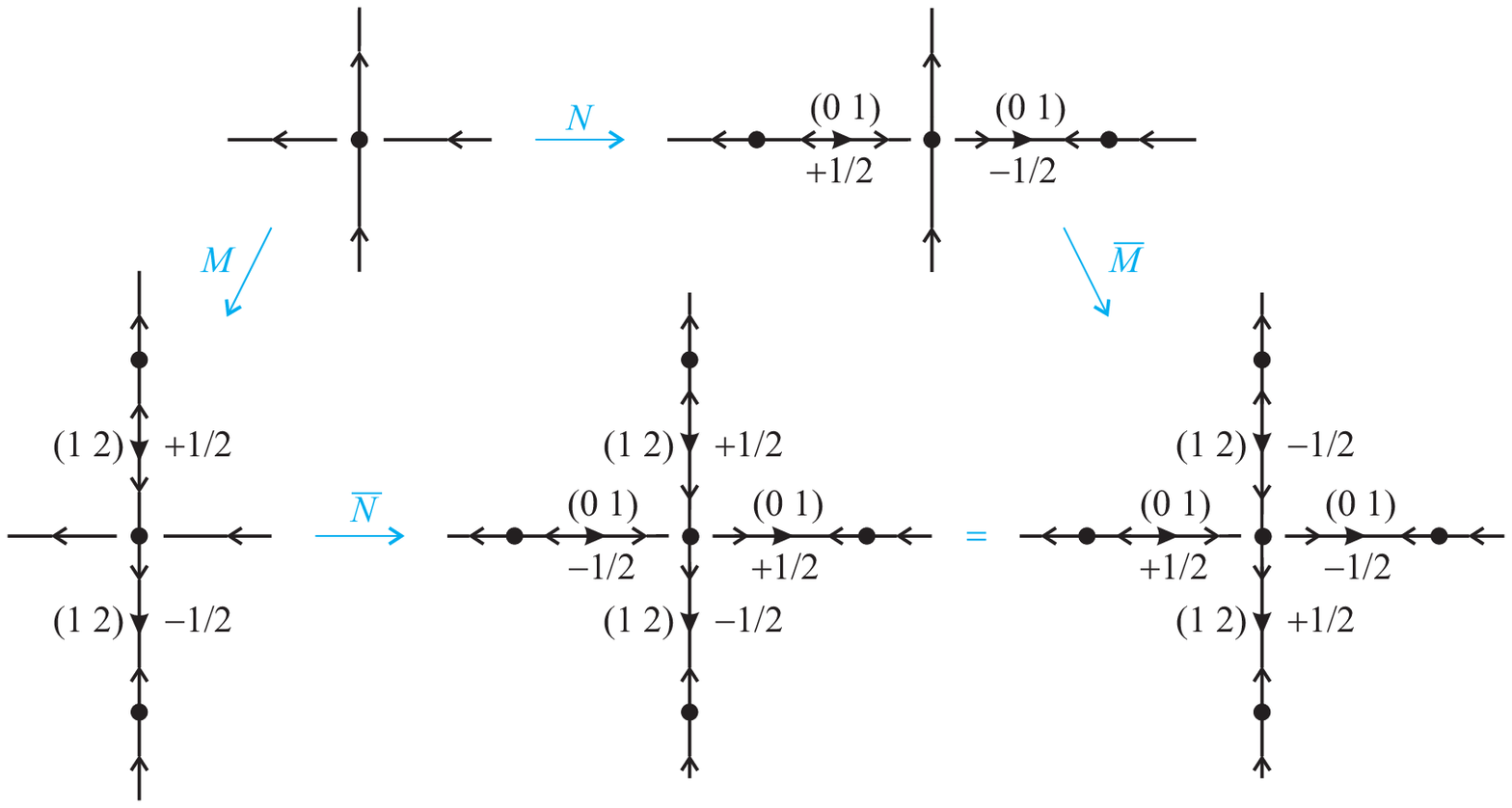}
    \end{center}
\vspace{-.5cm}\mycap{Proof that $M\!\cdot\! \overline{N}=N\!\cdot\!\overline{M}$ under local application of weighted fusion.\label{MNcommute:fig}}
\end{figure}

We now call \emph{weighted move} on a vertex as in Fig.~\ref{Wvertices:fig} any sequence
of elementary weighted moves (not followed by any fusion). We first have the following:

\begin{prop}\label{assoc:moves:for:W:prop}
Take $\Gamma\in\calN_{\text{\emph{w}}}$ and apply to each of its vertices a weighted move to get
$\widetilde\Gamma\in\calAtil$. Suppose that by applying (in some order) the weighted fusion rules of
Fig.~\ref{weightedcalAtilmult:fig} one gets $\Theta\in\calN_{\text{\emph{w}}}$.
Then the system of weights on $\Theta$
is well-defined independently of the order of application of the weighted fusion rules.
\end{prop}

\begin{proof}
The statement contains the implicit claim that the rules
of Fig.~\ref{weightedcalAtilmult:fig}
suffice to go from
$\widetilde\Gamma$ to some $\Theta$, namely that no
situation as in Remark~\ref{not:all:mult:rem} occurs.

We prove the proposition ignoring the colours in $\permu_3$, because we already know
that fusion is associative at the $\permu_3$ level. We concentrate on a single
edge of $\Gamma$ and we imagine $e$ is initially drawn in front of us
with orientation from left to right and weight $w\in\{0,1\}$. Replacing external
orientations of edges by letters $r/\ell$ and internal orientations by $R/L$
we then have in $\Gamma$ an initial edge $\frac{rr}w$ that gets replaced in
$\widetilde\Gamma$ by a concatenation $\widetilde e$ of edges of the form
$$\textstyle{\frac{rr}{u}},\quad\textstyle{\frac{\ell \ell}u},\quad\textstyle{\frac{rR\ell}a},
\quad\textstyle{\frac{rL\ell}a},\quad\textstyle{\frac{\ell Rr}a},\quad\textstyle{\frac{\ell Lr}a}.$$
A careful examination of the elementary weighted moves
actually shows that the possibilities
for $\widetilde e$ are only as follows:
\begin{equation}\label{first:etil:eq}
\left(\left(\textstyle{\frac{rr}u}\right)^{\!*}\!\cdot\!
\textstyle{\frac{rR\ell}a}\!\cdot\! \left(\textstyle{\frac{\ell \ell}u}\right)^{\!*}\!\cdot\!\textstyle{\frac{\ell Lr}a}\right)^{\!*}
\!\cdot\!\left(\textstyle{\frac{rr}u}\right)^{\!*}\!\cdot\! \textstyle{\frac{rr}w}\!\cdot\!\left(\textstyle{\frac{rr}u}\right)^{\!*}\!\cdot\!
\left(\textstyle{\frac{rL\ell}a}\!\cdot\!\left(\textstyle{\frac{\ell \ell}u}\right)^{\!*}\!\cdot\!
\textstyle{\frac{\ell Rr}a}\!\cdot\!\left(\textstyle{\frac{rr}u}\right)^{\!*}\right)^{\!*}
\end{equation}
\begin{equation}\label{second:etil:eq}
\left(\textstyle{\frac{\ell \ell}u}\right)^{\!*}\!\cdot\!\textstyle{\frac{\ell Lr}a}\!\cdot\!
(\ref{first:etil:eq})
\!\cdot\!\textstyle{\frac{rL\ell}a}\!\cdot\!\left(\textstyle{\frac{\ell \ell}u}\right)^{\!*}
\end{equation}
where $y^{\!*}$ means any number (including 0) of repetitions of a string $y$, and
the weight $u$ (respectively, $a$) can have a different value
in $\{0,1\}$ (respectively, $\pm{\textstyle{\frac12}}$) each time it appears.
It is then clear that we never get any of the adjacencies
$\frac{rL\ell}a\!\cdot\!\frac{\ell Lr}a$ or $\frac{\ell Rr}a\!\cdot\!\frac{rR\ell}a$ not contemplated by the weighted fusion rules of
Fig.~\ref{weightedcalAtilmult:fig}. Moreover these rules can be expressed as
$$\textstyle{\frac{rr}{u_1}}\!\cdot\!\textstyle{\frac{rr}{u_2}}=\textstyle{\frac{rr}{u_1+u_2}},\quad
\textstyle{\frac{rr}u}\!\cdot\!\textstyle{\frac{rD\ell}a}=\textstyle{\frac{rD\ell}{u+a}},\quad \textstyle{\frac{\ell Dr}a}\!\cdot\!\textstyle{\frac{rr}u}=\textstyle{\frac{\ell Dr}{a+u}},$$
$$\textstyle{\frac{\ell \ell}{u_1}}\!\cdot\!\textstyle{\frac{\ell \ell}{u_2}}=\textstyle{\frac{\ell \ell}{u_1+u_2}},\quad
\textstyle{\frac{\ell \ell}u}\!\cdot\!\textstyle{\frac{\ell Dr}a}=\textstyle{\frac{\ell Dr}{u+a}},\quad \textstyle{\frac{rD\ell}a}\!\cdot\!\textstyle{\frac{\ell \ell}u}=\textstyle{\frac{rD\ell}{a+u}},$$
$$\textstyle{\frac{\ell Lr}{a_1}}\!\cdot\!\textstyle{\frac{rL\ell}{a_2}}=\textstyle{\frac{\ell \ell}{a_1+a_2}},\quad
\textstyle{\frac{rR\ell}{a_1}}\!\cdot\!\textstyle{\frac{\ell Rr}{a_2}}=\textstyle{\frac{rr}{a_1+a_2}},$$
$$\textstyle{\frac{\ell Lr}{a_1}}\!\cdot\!\textstyle{\frac{rR\ell}{a_2}}=\textstyle{\frac{\ell \ell}{a_1-a_2}},\quad
\textstyle{\frac{\ell Rr}{a_1}}\!\cdot\!\textstyle{\frac{rL\ell}{a_2}}=\textstyle{\frac{\ell \ell}{a_1-a_2}},\quad
\textstyle{\frac{rR\ell}{a_1}}\!\cdot\!\textstyle{\frac{\ell Lr}{a_2}}=\textstyle{\frac{rr}{a_1-a_2}},\quad
\textstyle{\frac{rL\ell}{a_1}}\!\cdot\!\textstyle{\frac{\ell Rr}{a_2}}=\textstyle{\frac{rr}{a_1-a_2}}.$$
We must show that by applying them as long as possible to~(\ref{first:etil:eq}) or~(\ref{second:etil:eq}) we get a well-defined result.
Note first that each edge $\frac{\ell \ell}u$ or $\frac{rr}u$ can be ignored;
in fact, its contribution is independent of the time it is involved in weighted fusions, because:
\begin{itemize}
\item On the internal orientation it acts identically to the right and to the left;
\item Its numerical weight is in $\{0,1\}$, so it is insensitive to later sign change.
\end{itemize}
We then have to deal with concatenations of the form
\begin{equation}\label{pure:conc:1:eq}
\textstyle{\frac{rR\ell}{a_1}}\!\cdot\!
\textstyle{\frac{\ell Lr}{b_1}}\!\cdot\!
\ldots
\textstyle{\frac{rR\ell}{a_k}}\!\cdot\!
\textstyle{\frac{\ell Lr}{b_k}}\!\cdot\!
\textstyle{\frac{rL\ell}{d_h}}\!\cdot\!
\textstyle{\frac{\ell Rr}{c_h}}\!\cdot\!
\ldots
\textstyle{\frac{rL\ell}{d_1}}\!\cdot\!
\textstyle{\frac{\ell Rr}{c_1}},\end{equation}
\begin{equation}\label{pure:conc:2:eq}
\textstyle{\frac{\ell Lr}{b_0}}\!\cdot\!
\textstyle{\frac{rR\ell}{a_1}}\!\cdot\!
\textstyle{\frac{\ell Lr}{b_1}}\!\cdot\!
\ldots
\textstyle{\frac{rR\ell}{a_k}}\!\cdot\!
\textstyle{\frac{\ell Lr}{b_k}}\!\cdot\!
\textstyle{\frac{rL\ell}{d_h}}\!\cdot\!
\textstyle{\frac{\ell Rr}{c_h}}\!\cdot\!
\ldots
\textstyle{\frac{rL\ell}{d_1}}\!\cdot\!
\textstyle{\frac{\ell Rr}{c_1}}\!\cdot\!
\textstyle{\frac{rL\ell}{d_0}}\end{equation}
but we also consider the following (that arise starting
from an edge $\frac{\ell \ell}{w}$ in $\Gamma$):
\begin{equation}\label{pure:conc:3:eq}
\textstyle{\frac{\ell Lr}{a_1}}\!\cdot\!
\textstyle{\frac{rR\ell}{b_1}}\!\cdot\!
\ldots
\textstyle{\frac{\ell Lr}{a_k}}\!\cdot\!
\textstyle{\frac{rR\ell}{b_k}}\!\cdot\!
\textstyle{\frac{\ell Rr}{d_h}}\!\cdot\!
\textstyle{\frac{rL\ell}{c_h}}\!\cdot\!
\ldots
\textstyle{\frac{\ell Rr}{d_1}}\!\cdot\!
\textstyle{\frac{rL\ell}{c_1}},\end{equation}
\begin{equation}\label{pure:conc:4:eq}
\textstyle{\frac{rR\ell}{b_0}}\!\cdot\!
\textstyle{\frac{\ell Lr}{a_1}}\!\cdot\!
\textstyle{\frac{rR\ell}{b_1}}\!\cdot\!
\ldots
\textstyle{\frac{\ell Lr}{a_k}}\!\cdot\!
\textstyle{\frac{rR\ell}{b_k}}\!\cdot\!
\textstyle{\frac{\ell Rr}{d_h}}\!\cdot\!
\textstyle{\frac{rL\ell}{c_h}}\!\cdot\!
\ldots
\textstyle{\frac{\ell Rr}{d_1}}\!\cdot\!
\textstyle{\frac{rL\ell}{c_1}}\!\cdot\!
\textstyle{\frac{\ell Rr}{d_0}}.\end{equation}
We now claim that, regardless of the order in which the weighted fusion rules are applied,
the numerical edge weight on the final result is
$$\textstyle{\sum\limits_{i=1,\ldots,k}}a_i-\textstyle{\sum\limits_{i=1,\ldots,k}}b_i
+\textstyle{\sum\limits_{j=1,\ldots,h}}c_j-\textstyle{\sum\limits_{j=1,\ldots,h}}d_j
\quad\textrm{for}\ (\ref{pure:conc:1:eq})\ \textrm{and}\ (\ref{pure:conc:3:eq}),$$
$$\textstyle{\sum\limits_{i=1,\ldots,k}}a_i-\textstyle{\sum\limits_{i=0,\ldots,k}}b_i+\textstyle{\sum\limits_{j=1,\ldots,h}}c_j-
\textstyle{\sum\limits_{j=0,\ldots,h}}d_j
\quad\textrm{for}\ (\ref{pure:conc:2:eq})\ \textrm{and}\ (\ref{pure:conc:4:eq}).$$
The claim of course implies the conclusion, and we can prove it by induction on half the length of the concatenation.
The base step of the induction is with length $0$ in cases~(\ref{pure:conc:1:eq}) and~(\ref{pure:conc:3:eq}), so it is empty,
and with length $2$ in cases~(\ref{pure:conc:2:eq}) and~(\ref{pure:conc:4:eq}), so it follows directly
from the weighted fusion rules (remember that $-b_0-d_0=b_0+d_0$ because both $b_0$ and $d_0$ are $\pm{\textstyle{\frac12}}$).
For the inductive step we must analyze what happens by applying one weighted fusion to one of~(\ref{pure:conc:1:eq})-(\ref{pure:conc:4:eq}).
In all four cases we can distinguish between the ``central'' fusion
$\frac{\ell Lr}{b_k}\!\cdot\!\frac{rL\ell}{d_h}\to\frac{\ell\ell}{b_k+d_h}$ or
$\frac{rR\ell}{b_k}\!\cdot\!\frac{\ell Rr}{d_h}\to\frac{rr}{b_k+d_h}$ and any
``lateral'' fusion. Dealing with lateral fusions is easier, and we make it explicit
only for case~(\ref{pure:conc:1:eq}) and for a fusion performed to the left of the centre;
this fusion will  be
$\frac{rR\ell}{a_t}\!\cdot\!\frac{\ell Lr}{b_t}\to\frac{rr}{a_t-b_t}$ or
$\frac{\ell Lr}{b_t}\!\cdot\!\frac{rR\ell}{a_{t+1}}\to\frac{\ell\ell}{b_t-a_{t+1}}=\frac{\ell\ell}{a_{t+1}-b_t}$;
then we can forget the fused edge (remembering that its weight
must be added to the final one) so we are led to a shorter concatenation of type~(\ref{pure:conc:1:eq});
the inductive assumption then easily implies the conclusion.

Turning to the central fusion, in case~(\ref{pure:conc:1:eq}) forgetting the fused edge
we get the shorter concatenation of type~(\ref{pure:conc:4:eq})
$$\textstyle{\frac{rR\ell}{a_1}}\!\cdot\!
\textstyle{\frac{\ell Lr}{b_1}}\!\cdot\!
\textstyle{\frac{rR\ell}{a_2}}\!\cdot\!
\ldots
\textstyle{\frac{\ell Lr}{b_{k-1}}}\!\cdot\!
\textstyle{\frac{rR\ell}{a_k}}\!\cdot\!
\textstyle{\frac{\ell Rr}{c_h}}\!\cdot\!
\textstyle{\frac{rL\ell}{d_{h-1}}}\!\cdot\!
\ldots
\textstyle{\frac{\ell Rr}{c_2}}\!\cdot\!
\textstyle{\frac{rL\ell}{d_1}}\!\cdot\!
\textstyle{\frac{\ell Rr}{c_1}}$$
whence, by the inductive assumption, independently of the order, a final weight
\begin{equation}\label{doubleformula:eq}
\begin{array}{cl}
& b_k+d_h+\textstyle{\sum\limits_{i=1,\ldots,k-1}}b_i-
\textstyle{\sum\limits_{i=1,\ldots,k}}a_i+\textstyle{\sum\limits_{j=1,\ldots,h-1}}d_j-\textstyle{\sum\limits_{j=1,\ldots,h}}c_j\\
= & \textstyle{\sum\limits_{i=1,\ldots,k}}a_i-\textstyle{\sum\limits_{i=1,\ldots,k}}b_i+
\textstyle{\sum\limits_{j=1,\ldots,h}}c_j-\textstyle{\sum\limits_{j=1,\ldots,h}}d_j
\end{array}
\end{equation}
as desired. The central fusion in~(\ref{pure:conc:2:eq}) gives the type~(\ref{pure:conc:3:eq}) concatenation
$$\textstyle{\frac{\ell Lr}{b_0}}\!\cdot\!
\textstyle{\frac{rR\ell}{a_1}}\!\cdot\!
\ldots
\textstyle{\frac{\ell Lr}{b_{k-1}}}\!\cdot\!
\textstyle{\frac{rR\ell}{a_k}}\!\cdot\!
\textstyle{\frac{\ell Rr}{c_h}}\!\cdot\!
\textstyle{\frac{rL\ell}{d_{h-1}}}\!\cdot\!
\ldots
\textstyle{\frac{\ell Rr}{c_1}}\!\cdot\!
\textstyle{\frac{rL\ell}{d_0}}$$
whence final weight
precisely as in~(\ref{doubleformula:eq}), as desired.
The central fusion in cases~(\ref{pure:conc:3:eq}) and~(\ref{pure:conc:4:eq}) is similarly reduced
to the inductive assumption in cases~(\ref{pure:conc:2:eq}) and~(\ref{pure:conc:1:eq}) respectively.
\end{proof}

\begin{cor}\label{loc:assoc:cor}
If a sequence of elementary weighted moves is applied to a vertex as in Fig.~\ref{Wvertices:fig}
and the weighted fusions are applied as long as possible to the edges generated by these moves,
the result is independent of the order of fusions.
\end{cor}

\begin{proof}
By the argument showing Proposition~\ref{assoc:moves:for:W:prop} we must prove that concatenations of the form
$$\textstyle{\frac{rR\ell}{a_1}}\!\cdot\!
\textstyle{\frac{\ell Lr}{b_1}}\!\cdot\!
\ldots
\textstyle{\frac{rR\ell}{a_k}}\!\cdot\!
\textstyle{\frac{\ell Lr}{b_k}},\qquad
\textstyle{\frac{\ell Lr}{a_1}}\!\cdot\!
\textstyle{\frac{rR\ell}{b_1}}\!\cdot\!
\ldots
\textstyle{\frac{\ell Lr}{a_k}}\!\cdot\!
\textstyle{\frac{rR\ell}{b_k}}$$
$$\textstyle{\frac{\ell Lr}{b_0}}\!\cdot\!\textstyle{\frac{rR\ell}{a_1}}\!\cdot\!
\textstyle{\frac{\ell Lr}{b_1}}\!\cdot\!
\ldots
\textstyle{\frac{rR\ell}{a_k}}\!\cdot\!
\textstyle{\frac{\ell Lr}{b_k}},\qquad
\textstyle{\frac{rR\ell}{b_0}}\!\cdot\!\textstyle{\frac{\ell Lr}{a_1}}\!\cdot\!
\textstyle{\frac{rR\ell}{b_1}}\!\cdot\!
\ldots
\textstyle{\frac{\ell Lr}{a_k}}\!\cdot\!
\textstyle{\frac{rR\ell}{b_k}}$$
give a well-defined result. By induction on the length one can indeed see that
the first two give $\textstyle{\sum\limits_{i=1,\ldots,k}}b_i-\textstyle{\sum\limits_{i=1,\ldots,k}}a_i$ and the last two give
$\textstyle{\sum\limits_{i=0,\ldots,k}}b_i-\textstyle{\sum\limits_{i=1,\ldots,k}}a_i$.
\end{proof}

The two previous results imply that:
\begin{itemize}
\item We can define a \emph{weighted move} at a vertex as in Fig.~\ref{Wvertices:fig}
as a sequence of elementary weighted moves followed by weighted fusion;
\item If we apply to a graph in $\calN_{\textrm{w}}$ some weighted moves and after
weighted fusion we get another graph in
$\calN_{\textrm{w}}$, the weights on this last graph are well-defined.
\end{itemize}

\begin{prop}\label{well-def:moves:prop}
Two weighted moves at a vertex that coincide as unweighted moves also coincide as weighted moves.
\end{prop}

\begin{proof}
We will prove the result for moves turning a vertex of index $-1$ to another vertex of index $-1$,
the general case following by pre-composition with move $I$ and/or post-composition with move
$\overline{I}$. The $12$ moves described form a group $\Pi_-$ which is intrinsically isomorphic to
the alternating group $\permu_4^+$ on $4$ objects. This isomorphism is made explicit
with the choice of generators and the resulting presentation as follows:
$$\alpha=\trerom_-,\qquad\beta=I\cdot \overline{M},\qquad
\Pi_-=\left\langle\alpha,\beta|\ \alpha^2,\ \beta^3,\ (\alpha\cdot\beta)^3\right\rangle$$
(with moves and relations understood without weights). To conclude it is then sufficient
to show that the three relations hold also in a weighted sense. For $\alpha^2$ this was already implicit above
and very easy anyway; the other two weighted relations
are established in Figg.~\ref{betacubed:fig} and~\ref{alphabetacubed:fig}
\begin{figure}
    \begin{center}
    \includegraphics[scale=.6]{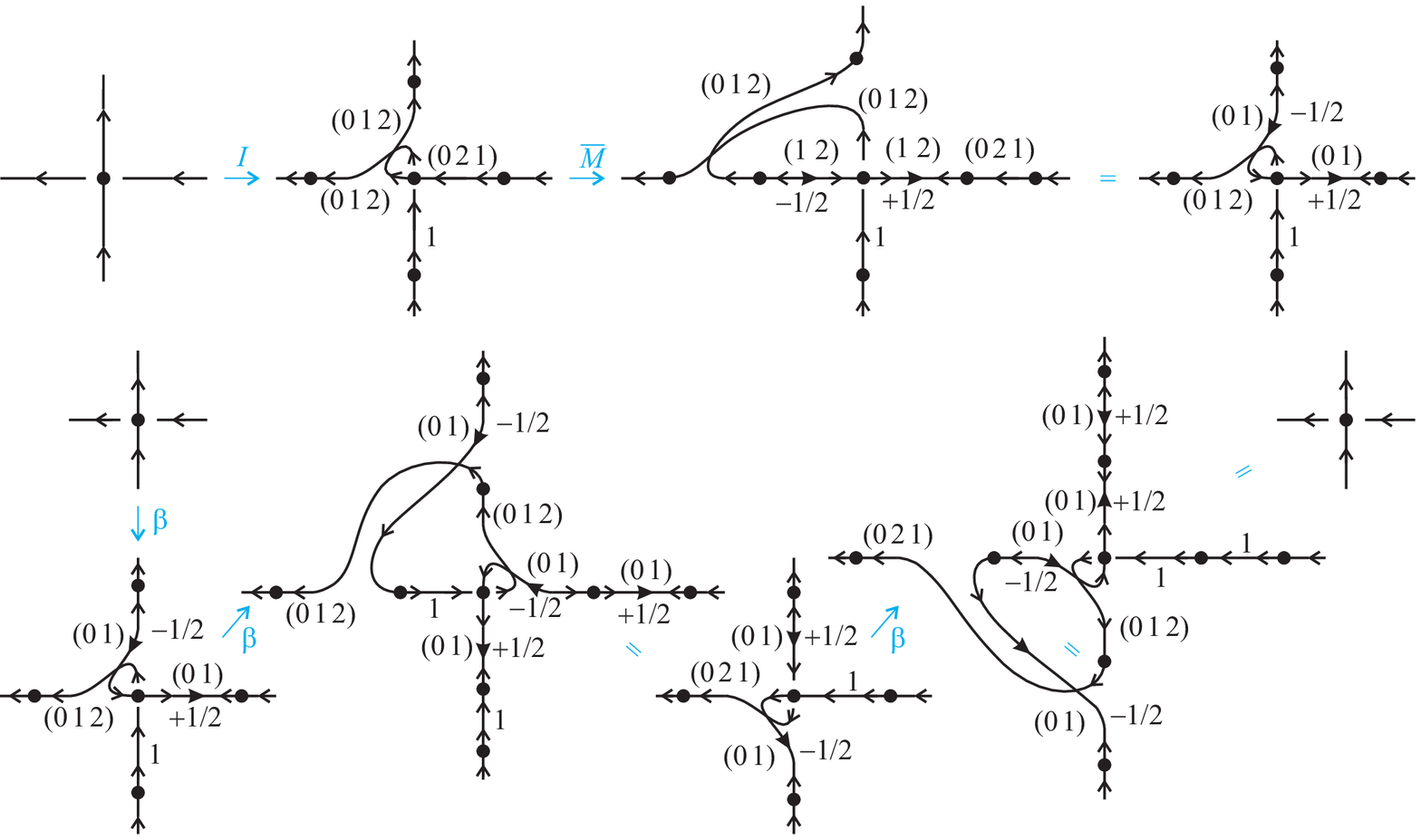}
    \end{center}
\vspace{-.5cm}\mycap{Computation of $\beta=I\cdot\overline{M}$ and proof that $\beta^3=\textrm{id}_-$ in a weighted sense.\label{betacubed:fig}}
\end{figure}
\begin{figure}
    \begin{center}
    \includegraphics[scale=.6]{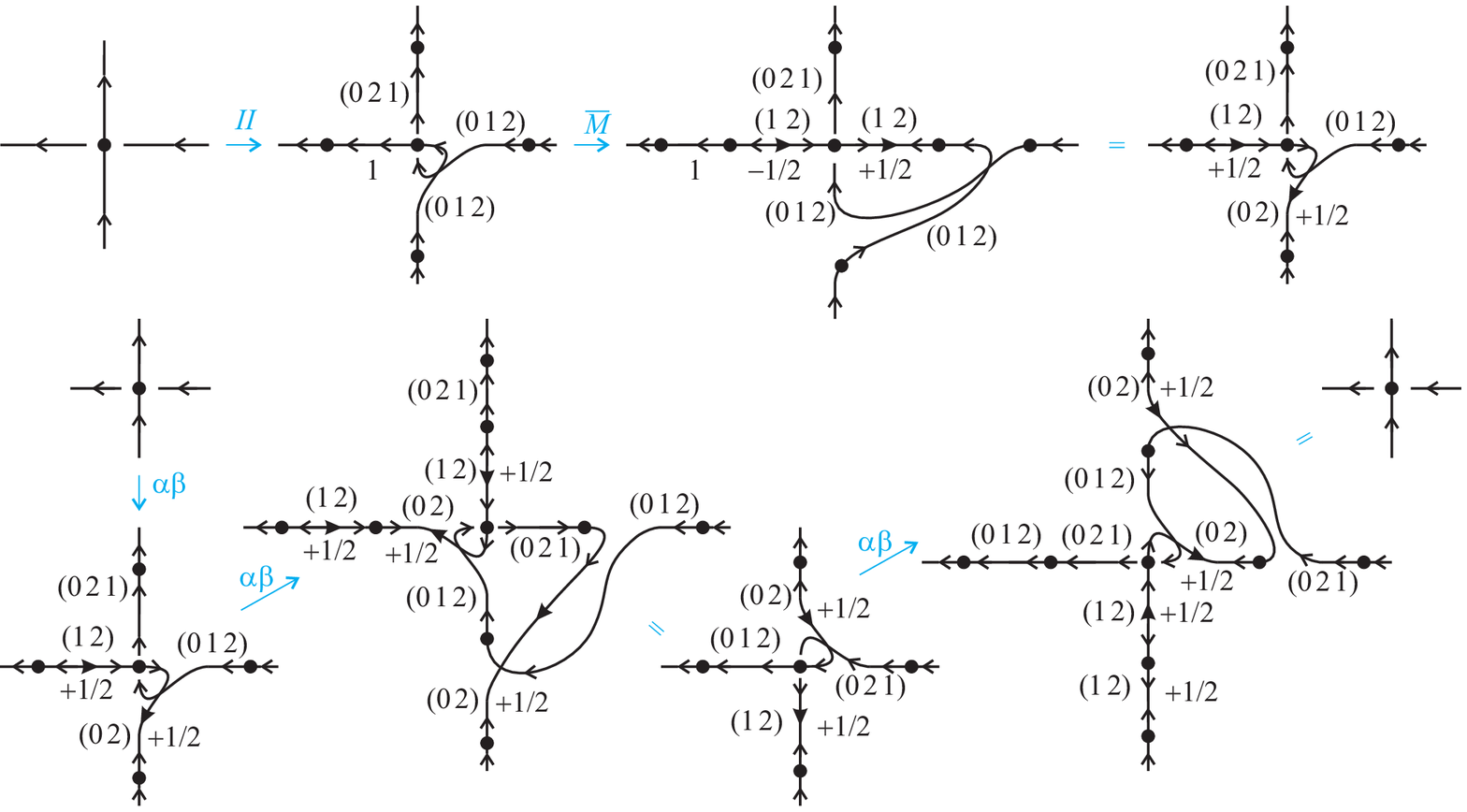}
    \end{center}
\vspace{-.5cm}\mycap{Computation of $\alpha\cdot\beta=\trerom_-\cdot I\cdot \overline{M}=
\duerom\cdot\overline{I}\cdot I\cdot \overline{M}=\duerom\cdot \overline{M}$ and
proof that $(\alpha\cdot\beta)^3=\textrm{id}_-$ in a weighted sense.\label{alphabetacubed:fig}}
\end{figure}
\end{proof}

We are eventually ready to establish our main result of this section:

\begin{thm}\label{loc:moves:thm}
Two graphs in $\calN_{\text{\emph{w}}}$ represent the same spine of some manifold $M$ and
the same spin structure $s$ on $M$ if and only if they can be obtained from each other
by a combination of the moves $I,\duerom,M,\overline{I},\overline{\duerom},\overline{M}$ and
weighted fusion.
\end{thm}

\begin{proof}
Suppose that $\Gamma_1,\Gamma_2\in\calN_{\textrm{w}}$
represent the same $(M,s)$. Then they are related bymoves $I,\duerom,\overline{I},\overline{\duerom}$ and circuit moves. To get
the desired conclusion it is then enough to show that the moves $M,\overline{M}$ generate
the circuit move, which is done in Fig.~\ref{newcircuitgeneration:fig}
\begin{figure}
    \begin{center}
    \includegraphics[scale=.6]{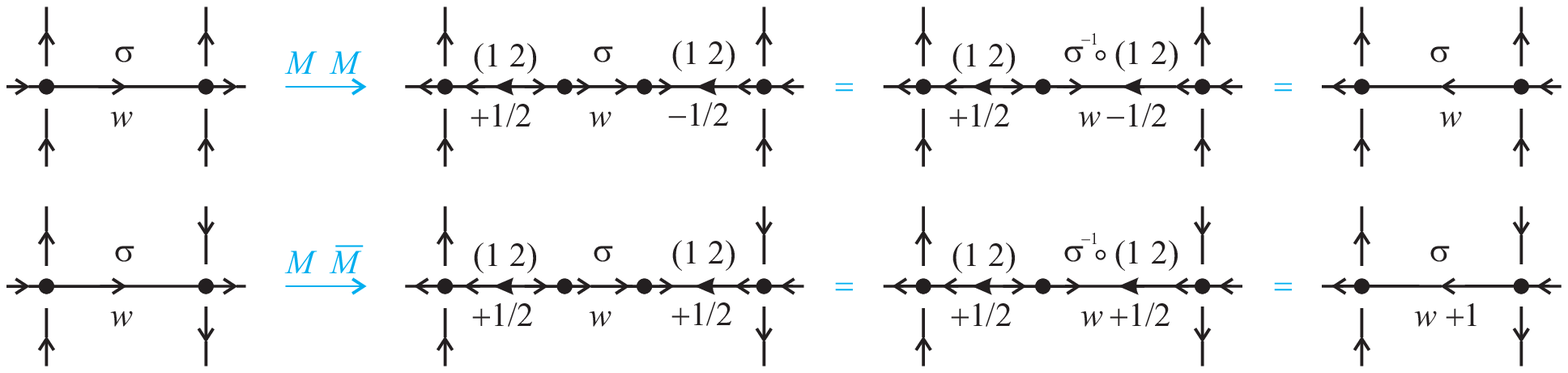}
    \end{center}
\vspace{-.5cm}\mycap{Generation of the circuit move via the moves $M,\overline{M}$ and weighted fusion.\label{newcircuitgeneration:fig}}
\end{figure}
for an edge of a circuit with first end of index
$-1$ (the cases with first end $+1$ being similar).

For the opposite implication we need two preliminary results. The proof of the first one is an easy
variation of the argument showing Proposition~\ref{circuit:move:prop}:

\begin{prop}\label{many:circuits:prop}
Suppose that in $\Gamma\in\calN_{\text{\emph{w}}}$ there are some (possibly intersecting)
oriented circuits $\gamma_1,\ldots,\gamma_n$, and that each $\gamma_j$ is either an undercircuit
(an overpass at all its vertices) or an
overcircuit (an underpass at all its vertices). Then the spin structure defined by $\Gamma$ is also
defined by the graph obtained from $\Gamma$ by reversing the orientation of each edge $e$ of
$\gamma_1\cup\ldots\cup\gamma_n$ and adding $1$ to the weight of $e$ if the ends if $e$
have different indices.
\end{prop}

\begin{prop}\label{Fede:prop}
If $\Gamma\in\calN_{\text{\emph{w}}}$ then using the moves
$I,\duerom,\overline{I},\overline{\duerom}$ at the vertices of $\Gamma$, followed
by fusion, one can get $\Theta\in\calN_{\text{\emph{w}}}$ such that each edge of $\Theta$ is
either an overpass at both its ends or an underpass at both its ends.
\end{prop}

\begin{proof}
To begin we note that given a vertex $V$ of $\Gamma$ and the choice of two germs of edges of $\Gamma$ at $V$
having consistent orientation through $V$, the moves $I,\duerom,\overline{I},\overline{\duerom}$
allow to put the two chosen germs of edges in the overpass position at $v$.
It is then enough to show that we can attach labels $o$ (over) and $u$ (under) to
the germs of edges of $\Gamma$ at vertices, so that:
\begin{itemize}
\item For each edge the labels at its ends are the same;
\item At each vertex the germs having the same label have consistent orientation.
\end{itemize}
One such labeling will be termed \emph{good}, and the coming argument proving its existence is due to Federico Petronio.
We choose a vertex $V$ of $\Gamma$ and attach any label to any of the germs of edge at $V$.
Then we propagate the labeling along a path in $\Gamma$ by applying alternatively the following rules
until $V$ is reached again:
\begin{itemize}
\item If an end of an edge has a label, give the other end the same label;
\item If at a vertex an incoming (respectively, outgoing) germ has a label, give the other
incoming (respectively, outgoing) germ the other label.
\end{itemize}
Note that the propagation path need not be simple, but at each vertex visited twice the labeling
is good ---see Fig.~\ref{freddyproof:fig}-left.
\begin{figure}
    \begin{center}
    \includegraphics[scale=.6]{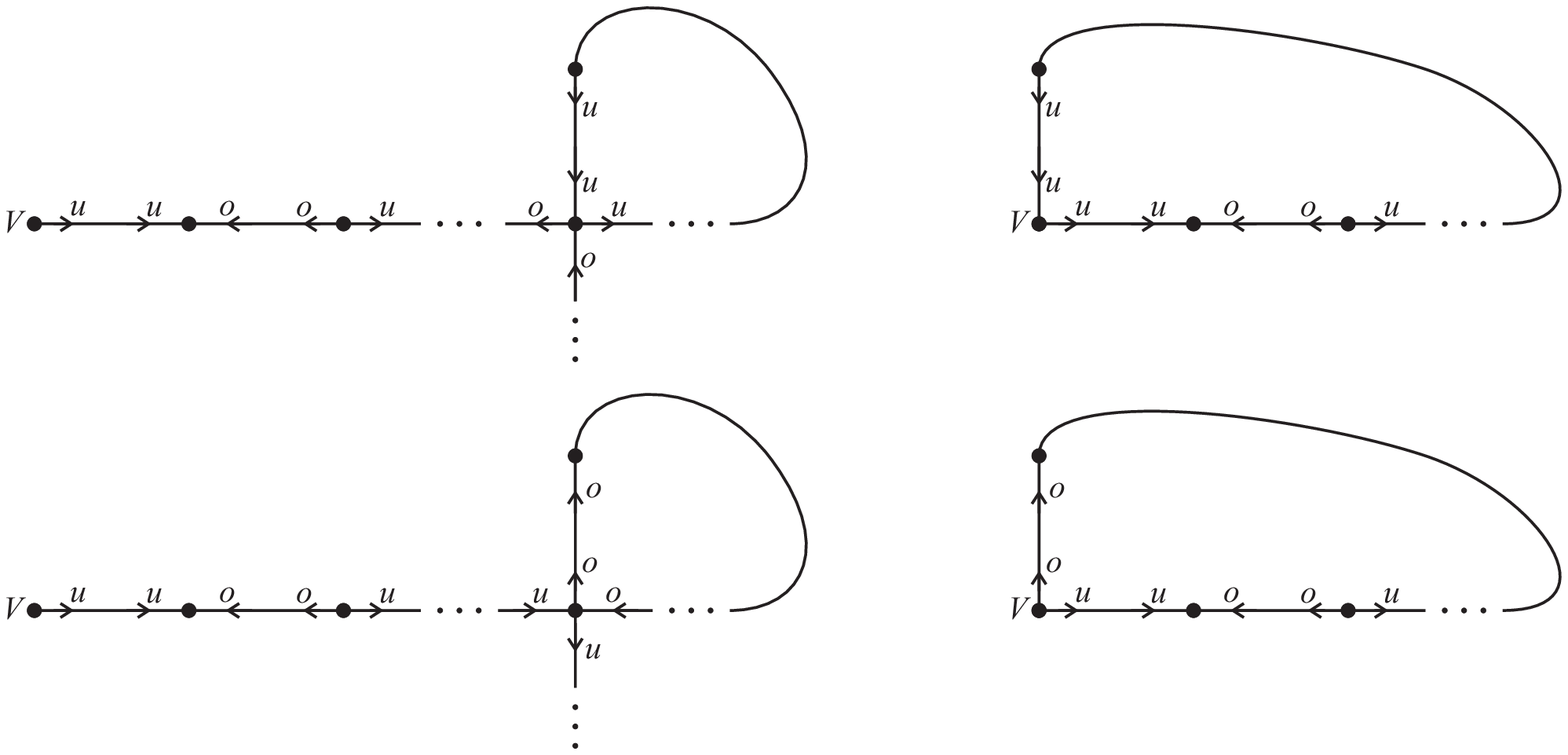}
    \end{center}
\vspace{-.5cm}\mycap{Extension of the labeling in case of initial label $u$ on
an outgoing germ at $V$.\label{freddyproof:fig}}
\end{figure}
When $V$ is reached again we have one of the situations in Fig.~\ref{freddyproof:fig}-right; in the top one we
proceed by applying the second rule, and eventually get back to $V$ again with a good labeling; in
the bottom one we proceed with an arbitrary choice of the label, but once more we get back to $V$ with a good labeling.
We can now similarly start from some other vertex, until all the germs of edges at vertices are labeled.\end{proof}

Back
to the proof of Theorem~\ref{loc:moves:thm}, suppose that $\Gamma_2\in\calN_{\textrm{w}}$
is obtained from $\Gamma_1\in\calN_{\textrm{w}}$ by a combination of weighted moves $I,\duerom,M,\overline{I},\overline{\duerom},\overline{M}$ and
weighted fusion. Let $\Delta$ be the union of the edges of $\Gamma_2$ having a different orientation in $\Gamma_1$.
By Proposition~\ref{Fede:prop} we can find weighted moves generated by $I,\duerom$ turning $\Gamma_2$ into
$\Gamma_3\in\calN_{\textrm{w}}$
in which $\Delta$ appears as a union of overcircuits and undercircuits.
Note that $\Gamma_3$ carries the same spin structure as $\Gamma_2$ by Proposition~\ref{vertex:moves:real:prop}.
With pictures similar to Fig.~\ref{newcircuitgeneration:fig} one can now see
that the multiple circuit moves of Proposition~\ref{many:circuits:prop}
are generated by the moves $M,\overline{M},N,\overline{N},D_-=M\cdot\overline{N}=N\cdot\overline{M},
D_+=\overline{M}\cdot N=\overline{N}\cdot M$
(the move $D_-$ is shown in Fig.~\ref{MNcommute:fig}, and $D_+$ is obtained similarly).

This shows that we can find a combination of the moves $I,\duerom,M,\overline{I},\overline{\duerom},\overline{M}$
that, after weighted fusion, turn $\Gamma_2$ into some $\Gamma_4$ carrying the same spin structure as $\Gamma_2$
and the same pre-branching as $\Gamma_1$. Proposition~\ref{vertex:moves:real:prop} then implies that
via moves $I,\duerom,\overline{I},\overline{\duerom}$ we can turn $\Gamma_4$ into some
$\widetilde{\Gamma}_1$ carrying the same spin structure as $\Gamma_2$ and different from $\Gamma_1$ possibly only
for the weights. We
then have a sequence
of weighted moves $I,\duerom,M,\overline{I},\overline{\duerom},\overline{M}$ that under weighted fusion give
$$\Gamma_1\longrightarrow\Gamma_2\longrightarrow\Gamma_3\longrightarrow\Gamma_4\longrightarrow\widetilde{\Gamma}_1$$
and that ignoring weights give the identity of $\Gamma_1$
(namely, they give the identity at every vertex of $\Gamma_1$). Proposition~\ref{well-def:moves:prop} then
implies that $\widetilde{\Gamma}_1$ coincides with $\Gamma_1$ also as a weighted graph (up to coboundaries).
This shows that $\Gamma_2$ carries the same spin structure as $\Gamma_1$.
\end{proof}

\subsection{Obstruction computation on graphs with split edges}

Even if this is not strictly necessary for our main results, we provide here
two methods for the computation of the obstruction $\alpha(P,\omega,b)$ carried
by a graph $\widetilde\Gamma\in\calAtil$ that after fusion becomes a graph in $\Theta\in\calN$ defining a triple
$(P,\omega,b)$. The first method is general, direct and easy;
the second one only applies to a $\widetilde\Gamma$ resulting from the application to some
$\Gamma\in\calN$ of the moves of Proposition~\ref{assoc:moves:for:W:prop}
(ignoring the numerical weights but using internal orientations),
and it is more complicated, but it also shows that
some non-trivial algebra underlies the computation.

\medskip

\noindent\textsc{First method}.
Take $\widetilde\Gamma\in\calAtil$ that after fusion gives $\Gamma\in\calN$ representing
$(P,\omega,b)$. We claim that $\alpha(P,\omega,b)$ can be computed from $\widetilde\Gamma$
by considering on the boundary of each region of $P$ some numerical
contributions in $G=\left(\frac12\cdot\matZ\right)/_{\!2\matZ}$
and some arrows, as in Proposition~\ref{real:obstruction:computation:prop}.
Contributions from vertices and from even edges are the same as in
Proposition~\ref{real:obstruction:computation:prop}, while those from an
odd edge $e$ are described as follows
(with the regions labeled 0,1,2 as in Fig.~\ref{regionlabels:fig}
and contributions $0$ not mentioned):
\begin{center}
\begin{tabular}{c||c|c|c}
$e$    & $\tau=(0\,1)$ & $\tau=(0\,2)$ & $\tau=(1\,2)$ \\ \hline\hline
    \includegraphics[scale=.5]{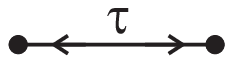}
&
    regions  $0$ and $1$ get $+{\textstyle{\frac12}}$ & all regions get $1$ & regions  $1$ and $2$ get $-{\textstyle{\frac12}}$ \\ \hline
    \includegraphics[scale=.5]{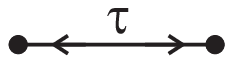}
&
    regions  $0$ and $1$ get $-{\textstyle{\frac12}}$ & all regions get $1$ & regions  $1$ and $2$ get $+{\textstyle{\frac12}}$
\end{tabular}
\end{center}
The proof that this recipe works follows from the fact that the contributions
combine consistently under fusion, which is shown on examples in Fig.~\ref{alphaAassoc:fig}.
\begin{figure}
    \begin{center}
    \includegraphics[scale=.6]{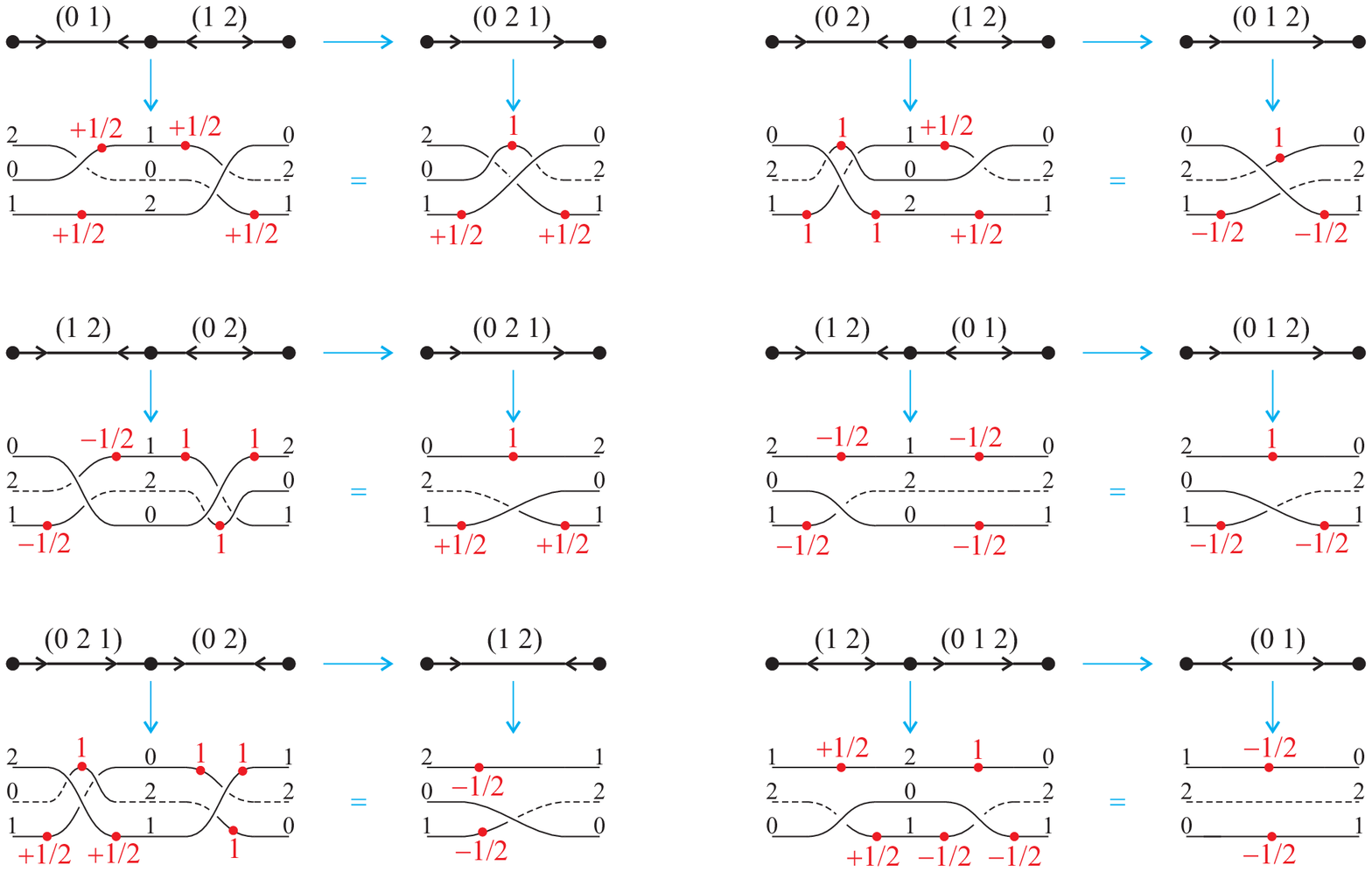}
    \end{center}
\vspace{-.5cm}\mycap{Associativity of the computation of $\alpha$ on a graph in $\calAtil$.\label{alphaAassoc:fig}}
\end{figure}

\medskip

\noindent\textsc{Second method}.
We begin with an apparently unrelated algebraic result. For any set $G$ we consider
the right action of $\permu_3$ on $G$ given by
$$(g_0,g_1,g_2)\cdot \eta=\left(g_{\eta(0)},g_{\eta(1)},g_{\eta(2)}\right).$$
We check that indeed this is a right action on an example:
\begin{eqnarray*}
\big((g_0,g_1,g_2)\cdot (0\,1)\big)\cdot(1\,2) & = &
(g_1,g_0,g_2)\cdot(1\,2)=(g_1,g_2,g_0)\\
(g_0,g_1,g_2)\cdot \big((0\,1)\compo(1\,2)\big) & = &
(g_0,g_1,g_2)\cdot (0\,1\,2)=(g_1,g_2,g_0).
\end{eqnarray*}
If $G$ is an Abelian group of course we have
$$\big((g_0,g_1,g_2)+(h_0,h_1,h_2)\big)\cdot \eta=
(g_0,g_1,g_2)\cdot \eta+(h_0,h_1,h_2)\cdot \eta$$
so we can define the semidirect product
$\permu_3\amalg G^3$ as $\permu_3\times G^3$ with operation
$$(\eta,(g_0,g_1,g_2))\cdot(\theta,(h_0,h_1,h_2))=
(\eta\compo\theta,(g_0,g_1,g_2)\cdot \theta+(h_0,h_1,h_2).$$
We now specialize our choice to $=\left({\textstyle{\frac12}}\cdot\matZ\right)/_{2\matZ}$ and we
establish the following:

\begin{prop}\label{section:prop}
Define $s:\permu_3\to G^3$ by
$$\begin{array}{rclrclrcl}
s(\emptyset) &\!\! =\!\! & (0,0,0)\ &\
    s((0\,1\,2)) &\!\! =\!\! & \left(-{\textstyle{\frac12}},-{\textstyle{\frac12}},1\right)\ &\
        s((0\,1)) &\!\! =\!\! & \left(-{\textstyle{\frac12}},+{\textstyle{\frac12}},0\right)\\
s((0\,2)) &\!\! =\!\! & (1,0,1)\ &\
    s((0\,2\,1)) &\!\! =\!\! & \left(1,+{\textstyle{\frac12}},+{\textstyle{\frac12}}\right)\ &\
        s((1\,2)) &\!\! =\!\! & \left(0,-{\textstyle{\frac12}},+{\textstyle{\frac12}}\right).\end{array}$$
Then $\Psi:\permu_3\to \permu_3\amalg G^3$ given by $\Psi(\eta)=(\eta,s(\eta))$ is a group homomorphism.
\end{prop}

\begin{proof}
If $x=(0\,1)$ and $y=(1\,2)$ we have the presentation of $\permu_3$ given by
$$\left\langle x,y|\ x^2,\ y^2,\ (x\cdot y)^3\right\rangle$$
with $(0\,1\,2)=x\cdot y$, $(0\,2\,1)=y\cdot x$, $(0\,2)=x\cdot y\cdot x$. The proposition will then be a consequence of the relations
$$\Psi(x)^2=\Psi(y)^2=\left(\Psi(x)\cdot\Psi(y)\right)^3=(\emptyset,(0,0,0)),\quad
\Psi((0\,2))=\Psi(x)\cdot\Psi(y)\cdot\Psi(x)$$
$$\Psi((0\,1\,2))=\Psi(x)\cdot\Psi(y),\quad
\Psi((0\,2\,1))=\Psi(y)\cdot\Psi(x).$$
We start with
\begin{eqnarray*}
\Psi(x)^2 & = &
\left((0\,1),\left(-{\textstyle{\frac12}},+{\textstyle{\frac12}},0\right)\right)\cdot
\left((0\,1),\left(-{\textstyle{\frac12}},+{\textstyle{\frac12}},0\right)\right)\\
& = & \left((0\,1)\compo(0\,1),
\left(+{\textstyle{\frac12}},-{\textstyle{\frac12}},0\right)+\left(-{\textstyle{\frac12}},+{\textstyle{\frac12}},0\right)\right)=(\emptyset,(0,0,0)).
\end{eqnarray*}
The computation of $\Psi(y)^2$ is similar. Before checking
that $\Psi(x)\cdot\Psi(y)$ has vanishing cube we compute it, checking it is
$\Psi((0\,1\,2))$:
\begin{eqnarray*}
\Psi(x)\cdot\Psi(y) & = &
\left((0\,1),\left(-{\textstyle{\frac12}},+{\textstyle{\frac12}},0\right)\right)\cdot
\left((1\,2)\left(0,-{\textstyle{\frac12}},+{\textstyle{\frac12}}\right)\right)\\
& = & \left((0\,1)\compo(1\,2),
\left(-{\textstyle{\frac12}},0,+{\textstyle{\frac12}}\right)+
\left(0,-{\textstyle{\frac12}},+{\textstyle{\frac12}}\right)\right)\\
& = & \left((0\,1\,2),\left(-{\textstyle{\frac12}},-{\textstyle{\frac12}},1\right)\right).
\end{eqnarray*}
And now we conclude:
\begin{eqnarray*}
\left(\Psi((0\,1\,2))\right)^3 & = &
\left((0\,1\,2),\left(-{\textstyle{\frac12}},-{\textstyle{\frac12}},1\right)\right)^3\\
& = &
\left((0\,1\,2)\compo(0\,1\,2),\left(-{\textstyle{\frac12}},1,-{\textstyle{\frac12}}\right)+
\left(-{\textstyle{\frac12}},-{\textstyle{\frac12}},1\right)\right)
\cdot\Psi((0\,1\,2))\\
& = & \left((0\,2\,1),\left(1,+{\textstyle{\frac12}},+{\textstyle{\frac12}}\right)\right)\cdot
\left((0\,1\,2),\left(-{\textstyle{\frac12}},-{\textstyle{\frac12}},1\right)\right)\\
& = &
\left((0\,2\,1)\compo(0\,1\,2),
\left(+{\textstyle{\frac12}},+{\textstyle{\frac12}},1\right)+\left(-{\textstyle{\frac12}},-{\textstyle{\frac12}},1\right)\right)=(\emptyset,(0,0,0));
\end{eqnarray*}
\begin{eqnarray*}
\Psi(y)\cdot\Psi(x) & = &
\left((1\,2),\left(0,-{\textstyle{\frac12}},+{\textstyle{\frac12}}\right)\right)\cdot
\left((0\,1),\left(-{\textstyle{\frac12}},+{\textstyle{\frac12}},0\right)\right)\\
& = & \left((1\,2)\compo(0\,1),
\left(-{\textstyle{\frac12}},0,+{\textstyle{\frac12}}\right)+\left(-{\textstyle{\frac12}},+{\textstyle{\frac12}},0\right)\right)\\
& = & \left((0\,2\,1),\left(1,+{\textstyle{\frac12}},+{\textstyle{\frac12}}\right)\right);\\
\Psi(x)\cdot \Psi(y)\cdot\Psi(x) & = &
\left((0\,1),\left(-{\textstyle{\frac12}},+{\textstyle{\frac12}},0\right)\right)\cdot
\left((0\,2\,1),\left(1,+{\textstyle{\frac12}},+{\textstyle{\frac12}}\right)\right)\\
& = & \left((0\,1)\compo(0\,2\,1),
\left(0,-{\textstyle{\frac12}},+{\textstyle{\frac12}}\right)+\left(1,+{\textstyle{\frac12}},+{\textstyle{\frac12}}\right)\right)\\
& = & \left((0\,2),\left(1,0,1\right)\right).
\end{eqnarray*}
\end{proof}

\begin{rem}
\emph{The previous result remains true, with the same proof, if the values on $s$ on
the transpositions are redefined as}
$$s((0\,1))=\left(+{\textstyle{\frac12}},-{\textstyle{\frac12}},1\right),\quad
s((1\,2))=\left(1,+{\textstyle{\frac12}},-{\textstyle{\frac12}}\right),\quad
s((0\,2))=(0,1,0).$$
\end{rem}

Let us then turn to the computation of the obstruction $\alpha(P,\omega,b)$.
We start from $\Gamma\in\calN$, we apply to it some of the
moves of Proposition~\ref{assoc:moves:for:W:prop}
(but neglecting the numerical weight) and we call $\widetilde\Gamma$ the result.
Next, we assume that applying fusion to $\widetilde\Gamma$ we get $\Theta\in\calN$ defining
$(P,\omega,b)$. Note that every edge of $\widetilde\Gamma$ carries
an internal orientation (that for an even edge we stipulate to be the the same as the
orientations at the ends). Let us concentrate on an edge $e$ of $\Theta$,
that in $\widetilde\Gamma$ (before fusion) will be subdivided into several edges.
Since in $\Theta\in\calN$ the edge $e$ is oriented, we can speak of a global orientation of
$e$ (that coincides with the internal orientations of the two extremal subedges of $e$).
Now note that each subedge $e'$ of $e$ brings three portions of strands of attaching
circles of $P$ to $S(P)$, and that these strands are numbered $0,1,2$ at both ends of $e'$
according to the orientation of these. The recipe for the computation of
$\alpha(P,\omega,b)$ now uses the map $s$ of
Proposition~\ref{section:prop}, and goes at follows:
\begin{itemize}
\item Let $\eta\in\permu_3$ be the permutation attached to $e'$, and
define $(h_0,h_1,h_2)$ to be $s(\eta)$ if the internal orientation of $e'$ is consistent
with the global one, otherwise define $(h_0,h_1,h_2)$ as $s\left(\eta^{-1}\right)$;
\item At the first end of $e'$ with respect to the \emph{global} orientation,
attach to the strands $0,1,2$ the weights $h_0,h_1,h_2$.
\end{itemize}
A formal proof that summing the contributions of the various $e'$ one gets the edge contributions to $\alpha(P,\omega,b)$
as in Proposition~\ref{real:obstruction:computation:prop}
employs Proposition~\ref{section:prop}, but we confine ourselves here to some
examples only, see Fig.~\ref{section123:fig}
\begin{figure}
    \begin{center}
    \includegraphics[scale=.6]{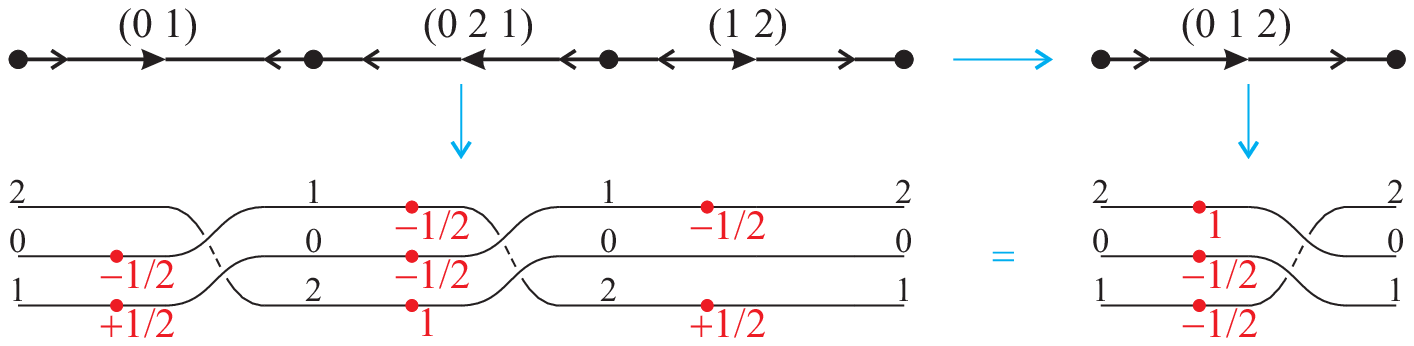}\\ \ \\
    \includegraphics[scale=.6]{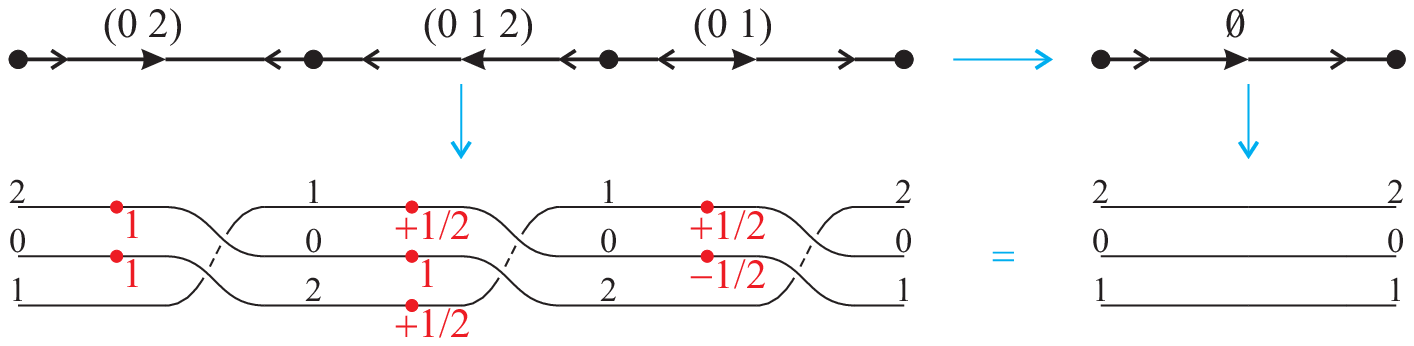}\\ \ \\
    \includegraphics[scale=.6]{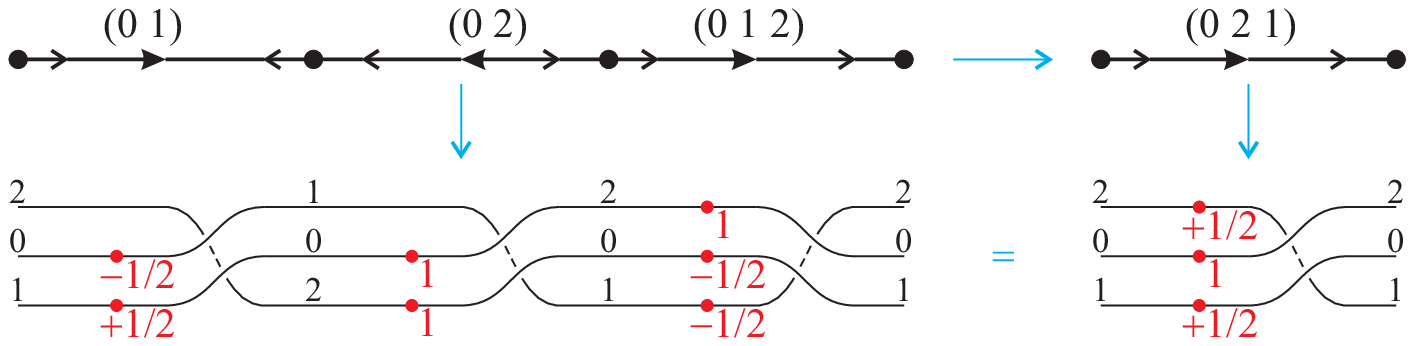}
    \end{center}
\vspace{-.5cm}\mycap{Examples of computation of $\alpha$ with the second method.\label{section123:fig}}
\end{figure}
and~\ref{section456:fig}.
\begin{figure}
    \begin{center}
    \includegraphics[scale=.6]{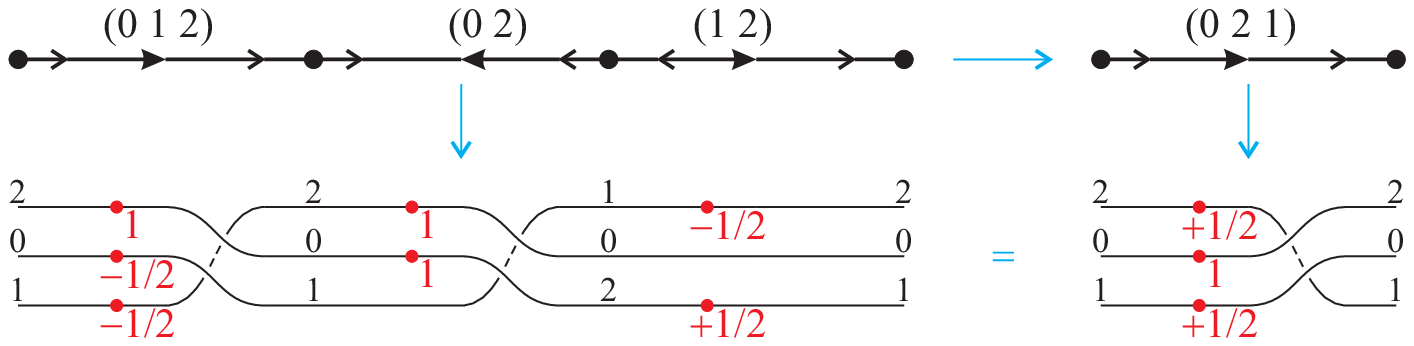}\\ \ \\
    \includegraphics[scale=.6]{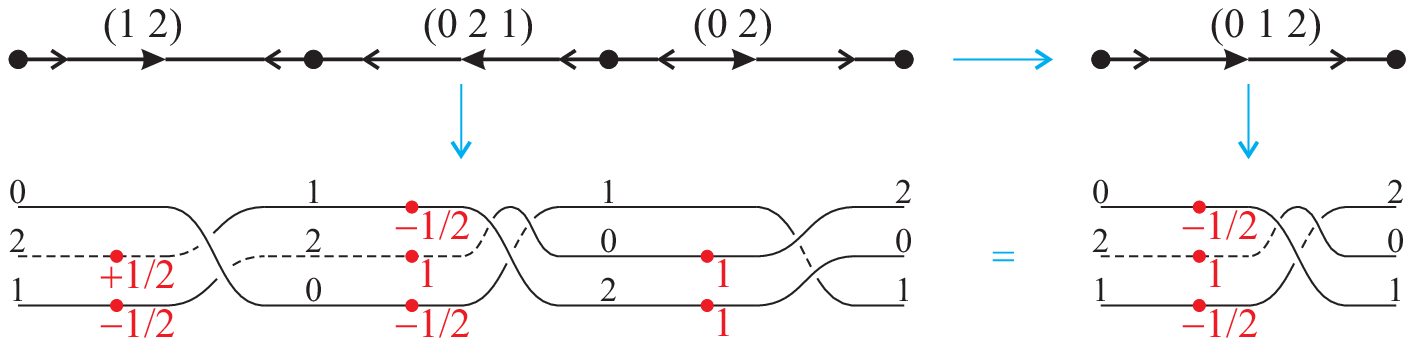}\\ \ \\
    \includegraphics[scale=.6]{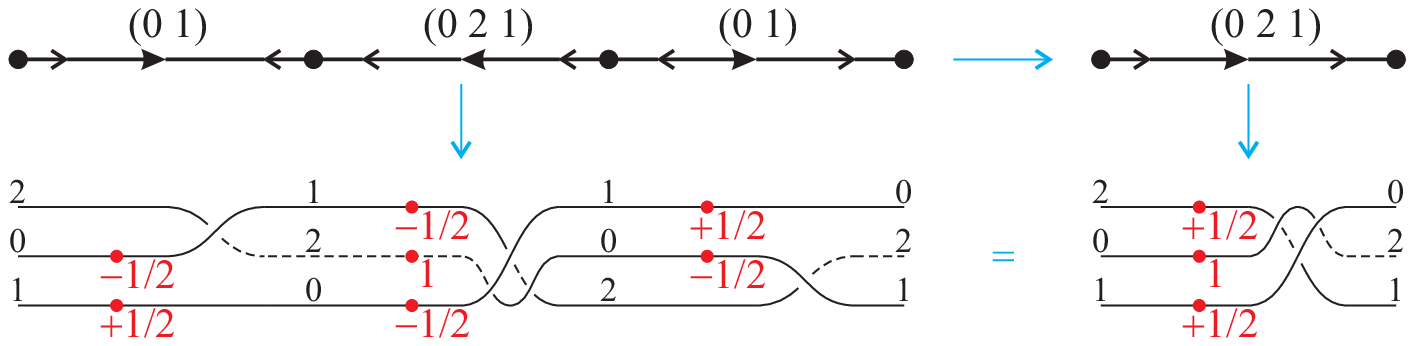}
    \end{center}
\vspace{-.5cm}\mycap{More examples of computation of $\alpha$ with the second method.\label{section456:fig}}
\end{figure}

\vspace{.5cm}

\noindent
Dipartimento di Matematica\quad
Universit\`a di Pisa\\
Largo Bruno Pontecorvo 5\quad 56127 PISA -- Italy\\
benedett@dm.unipi.it\quad
petronio@dm.unipi.it

\end{document}